\documentclass[11pt]{amsart}

\usepackage[svgnames,x11names]{xcolor}
\usepackage{amssymb,amsfonts}
\usepackage[pagebackref,colorlinks,urlcolor={DarkOliveGreen4},
linkcolor={DarkSlateBlue},
citecolor={Chocolate4}]{hyperref}
\usepackage[capitalize]{cleveref}
\usepackage{amssymb,textcomp}
\usepackage{graphicx}
\numberwithin{equation}{section}
\usepackage{enumerate}

\usepackage[utf8]{inputenc}
\usepackage[T1]{fontenc}

\usepackage[mathscr]{eucal}

\usepackage[a4paper, twoside=false, vmargin={2cm,3cm}, includehead]{geometry}

\newtheorem{lemma}{Lemma}[section]
\newtheorem{theorem}[lemma]{Theorem}
\newtheorem*{theorem*}{Theorem}
\newtheorem{corollary}[lemma]{Corollary}

\newtheorem*{question*}{Question}
\newtheorem{proposition}[lemma]{Proposition}
\newtheorem*{proposition*}{Proposition}

\newtheorem{conjecture}{Conjecture}

\newtheorem*{problem*}{Problem}

\theoremstyle{definition}
\newtheorem{definition}[lemma]{Definition}
\newtheorem*{claim*}{Claim}

\newtheorem*{remark}{Remark}
\newtheorem*{remarks}{Remarks}

\newcommand{\legendre}[2]{\ensuremath{\left( \frac{#1}{#2} \right) }}

\newcommand{\C}{{\mathbb C}}
\newcommand{\E}{{\mathbb E}}
\newcommand{\D}{{\mathbb D}}

\newcommand{\N}{{\mathbb N}}
\renewcommand{\P}{{\mathbb P}}
\newcommand{\Q}{{\mathbb Q}}
\newcommand{\R}{{\mathbb R}}
\renewcommand{\S}{\mathbb{S}}

\newcommand{\Z}{{\mathbb Z}}
\newcommand{\U}{{\mathbb U}}

\newcommand{\CA}{{\mathcal A}}
\newcommand{\CC}{{\mathcal C}}

\newcommand{\CF}{{\mathcal F}}

\newcommand{\CM}{{\mathcal M}}
\newcommand{\CN}{{\mathcal N}}

\newcommand{\CP}{{\mathcal P}}

\DeclareMathOperator{\Imag}{Im}
\renewcommand{\Im}{\Imag}

\newcommand{\rem}[1]{{\bf $\clubsuit$ #1 $\clubsuit$}}
\renewcommand{\rem}[1]{}
\begin{document}
\title{Partition regularity of generalized Pythagorean pairs}  

\author{Nikos Frantzikinakis}
\address[Nikos Frantzikinakis]{University of Crete, Department of Mathematics and Applied Mathematics, Heraklion, Greece}
\email{frantzikinakis@gmail.com}

\author{Oleksiy Klurman}
\address[Oleksiy Klurman]{University of Bristol, School of Mathematics, Bristol, UK}
\email{lklurman@gmail.com}

\author{Joel Moreira}
\address[Joel Moreira]{University of Warwick, Department of Mathematics, Coventry, UK}		 \email{ joel.moreira@warwick.ac.uk}
\date{\today}
\begin{abstract}
	We address partition regularity problems 
	for  homogeneous quadratic equations.  
	A consequence of our main results is that, under natural conditions on the coefficients $a,b,c$, for
	any finite coloring of the positive integers, there exists a solution to $ax^2+by^2=cz^2$ where $x$ and $y$ have the same color (and similar results for $x,z$ and $y,z$).  
	For certain choices of $(a,b,c)$, our result is conditional on an Elliott-type conjecture.
	Our proofs build on and extend previous arguments of the authors dealing  with the Pythagorean equation. We make use of new uniformity properties of aperiodic multiplicative functions and concentration estimates for multiplicative functions along arbitrary binary quadratic forms.  
	
\end{abstract}

\thanks{The first author was supported  by the Research Grant ELIDEK HFRI-NextGenerationEU-15689 and the third by the EPSRC Frontier Research Guarantee grant ref. EP/Y014030/1. For the purpose of open access, the authors have applied a Creative Commons Attribution (CC BY) licence to any Author Accepted Manuscript version arising from this submission.}



\date{\today}

\maketitle
\setcounter{tocdepth}{1}
\tableofcontents

\section{Introduction}
\rem{These are changes made in 2025: Wrote $P_c(Qm+a,Qn+b))^{it}$ in Proposition 4.8, if we want to remove $a,b$ we need to add a term that depends on $P,a,b,t$ and goes to $0$ as $N\to\infty$. No change is needed in Proposition 4.9 since we take the limit as $N\to\infty$ there.

	Also in Proposition 4.9 (copy paste the whole proposition in the online file) wrote 
	
"		where  $c:=(P(a, b),Q)$, 		
	$P_c:=P/c$, $\Phi'_K$ is a subset of $\Phi_K$, $S_K$ consists of all  $a,b\in [-A_K,A_K]$ such that   $q$ and $\prod_{p\leq K}p$ divide $Q/c$ for all $Q\in \Phi'_K$,   and  $G_{P,N},\D_P$ are as in   \eqref{E:GNP}, \eqref{E:DPxy}"
	because    $q$ and $\prod_{p\leq K}p$  to divide $Q/c$ for every $Q\in \Phi_K'$. In the proof of the conditional result we'll take $\Phi'_K$ to be the singleton $Q_K:=\prod_{p\leq K}p^{2K}$".  
}

\subsection{Partition regularity results}
A central problem in arithmetic Ramsey theory is to determine when a polynomial equation is \emph{partition regular} -- this means that in every partition of $\N$ into finitely many cells there is a solution to the equation with all variables in the same cell.
A still unsolved special case of this problem is to decide whether the Pythagorean equation $x^2+y^2=z^2$ is partition regular.
In the previous paper \cite{FKM23} we studied this and more general quadratic equations of the form 
\begin{equation}\label{E:abc}
	ax^2+by^2=cz^2, 
\end{equation}
where $a,b,c$ are arbitrary  integers. Additional results about partition regularity of quadratic equations can be found in \cite{Be87,BP17,Ch22,CLP21,CGS12,DL18,FH17, HKM16, KS06, Mo17}, and we refer the reader to the introduction of \cite{FKM23} for a more comprehensive account of partition regularity problems and results related to the ones considered in the present article.
While our methods, in the current form, do not allow us to establish full partition regularity for equations of the form \eqref{E:abc}, we made progress in understanding partition regularity for pairs.
\begin{definition}
	Given nonzero $a,b,c\in \Z$, we say that the equation \eqref{E:abc} is {\em partition regular with respect to $x,y$}  if for every finite coloring of $\N$ there exist distinct $x,y\in \N$, with the same color, and $z\in \N$ that satisfy \eqref{E:abc}. Similarly, we define the partition regularity of \eqref{E:abc}  with respect to the variables $x,z$ and $y,z$.
\end{definition}

\begin{remarks}
	$\bullet$  If \eqref{E:abc} is partition regular with respect to all three variables, then it is obviously partition regular for any pair of variables. 
	However the converse does not hold: \cite[Theorem 1.1]{FKM23} implies that the equation $x^2+y^2=4z^2$ is partition regular with respect to any pair of variables, but this equation is not partition regular with respect to all three variables. 
	The last fact follows by combining  Rado's theorem \cite{R33} with the fact that  partition regularity of the equation  $P(x^2,y^2,z^2)=0$ implies partition regularity of the equation $P(x,y,z)=0$. 
	
	$\bullet$ It follows from \cite[Proposition~2.13]{DLMS23} that  equation \eqref{E:abc} (or any other homogeneous equation) is partition regular  with respect to the variables $(x,y)$ over $\N$ if and only if the same holds over  $\mathbb{Q}_+$. 
\end{remarks}

In this paper we address the following 
conjecture explicitly formulated in \cite{FKM23} (see the  remark following \cite[Problem~2]{FKM23}).
\begin{conjecture}\label{C:conj1}
	If $a,b,c\in \Z$ are nonzero and 
	at least one of the integers $ac,bc, (a+b)c$ is a square, then \eqref{E:abc}  is partition regular with respect to $x,y$.
\end{conjecture}

In \cref{sec_necessaryconditions} we elaborate on why the assumptions in \cref{C:conj1} are natural, and perhaps even necessary.

Equation \eqref{E:abc} can be written as $a'z^2+b'y^2=c'x^2$ with $(a',b',c'):=(-c,b,-a)$; therefore \cref{C:conj1} is equivalent to the statement that \eqref{E:abc} is partition regular with respect to $y,z$ whenever at least one of the integers $ac,-ab,(c-b)a$ is a square. 
	Throughout the present paper, prototypical working examples to keep in mind are the equations
	$x^2+2y^2=z^2$ and $x^2+y^2=2z^2$, which  are expected to be partition regular with respect to any two variables. 	We will see later why these two equations encapsulate the various difficulties of the more general problem. 
	
	Our first main result verifies part of \cref{C:conj1}.
	\begin{theorem}\label{T:PairsPartition}
			If  $a,b,c\in \Z$ are nonzero and  $ac$ or $bc$ is a square, then \eqref{E:abc} is partition regular with respect to $x,y$.
	\end{theorem}
	It follows that if $a,b,c\in \Z$ are nonzero  and $ac$ or $-ab$   is a square, then \eqref{E:abc} is partition regular with respect to $y,z$.  
	Theorem~\ref{T:PairsPartition} extends  \cite[Theorem~1.1]{FKM23} that  covers the case where  {\it both} $ac$ and $bc$ are  squares (which in turn extends  \cite[Theorem~2.7]{FH17} that covers the case where {\it all three} quantities $ac,bc, (a+b)c$ are squares). Specializing to $a=1,b=2,c=1$ this solves \cite[Problem 2(i)]{FKM23}
	where it was conjectured that $x^2+2y^2=z^2$ is partition regular with respect to $x,y$ and $y,z$. 
	
	However, \cref{T:PairsPartition} does not cover, 
	for example, partition regularity for the equation $x^2+y^2=2z^2$ (with respect to any two variables) or the equation $x^2+2y^2=z^2$ with respect to the variables $x,z$. 
	We provide  a conditional result  that covers the missing cases, assuming  the following ``Elliott-type'' conjecture holds. 
	\begin{conjecture}\label{C:2quadratic}	 
		If $P_1,P_2\in \Z[m,n]$ are irreducible binary quadratic forms that are not multiples of each other,
		then they are good for vanishing of correlations of aperiodic multiplicative  functions (see Definition~\ref{D:CorVan} below).
	\end{conjecture}
	The conjecture is open even for the Liouville function, for instance it is not known whether
	$$
	\lim_{N\to\infty} \E_{m,n\in[N]}\, \lambda(m^2+n^2)\cdot \lambda(m^2+2n^2)=0,
	$$
	even if we replace the Ces\`aro averages with logarithmic averages (which is sufficient for the applications we have in mind). In fact, there is no specific pair of irreducible binary quadratic forms $P_1,P_2$ for which the identity 
	$$
	\lim_{N\to\infty} \E_{m,n\in[N]}\, \lambda(P_1(m,n))\cdot \lambda(P_2(m,n))=0
	$$
	is known to hold. On the other hand, if one of the two binary quadratic forms is irreducible and the other one is reducible, then it follows from \cref{T:corrvanisha} below and simple manipulations that the vanishing property of \cref{C:2quadratic} holds for these polynomials.

	\begin{theorem}\label{T:PairsPartitionCond}
		Suppose that Conjecture~\ref{C:2quadratic} holds. 
			If  $a,b,c\in \Z$  are nonzero and  $(a+b)c$ is a square (perhaps zero), then \eqref{E:abc} is partition regular with respect to $x,y$.
		\end{theorem}
		
		As a corollary of Theorems \ref{T:PairsPartition} and \ref{T:PairsPartitionCond}, we have that \cref{C:2quadratic} implies \cref{C:conj1}.
		
		
		The proof method of \cref{T:PairsPartitionCond} allows us to establish the following unconditional result, which gives further credibility to \cref{C:conj1} when $(a+b)c$ is a square.
		\begin{theorem}\label{T:levelsets}
			Suppose that  $a,b,c\in \Z$ are nonzero and  $(a+b)c$ is a square. Let
			$f_1,\ldots, f_r\colon \N\to \S^1$ be 
			pretentious  multiplicative functions (see \cref{D:aperiodic}),  and  $I$ be an open arc of $\S^1$  containing $1$. 
			Then there exist distinct $x,y\in \N$ such that $f_j(x)\in I$ and $f_j(y)\in I$, for $j=1,\ldots, r$, and $ax^2+by^2=cz^2$ for some $z\in \N$.  
		\end{theorem}
		
		Finally, we give some consequences of 
		Theorems~\ref{T:PairsPartition} and \ref{T:PairsPartitionCond}. A well-known open problem 
		is whether  a necessary and sufficient condition for partition regularity of the equation \eqref{E:abc} with respect to {\em all three variables} is that the triple $(a,b,c)$ of positive integers satisfies Rado's condition for partition regularity of $ax+by=cz$, namely, $a=c$,  or  $b=c$, or $a+b=c$.   By combining Theorems~\ref{T:PairsPartition} and \ref{T:PairsPartitionCond}, we get the following result, which adds some credibility to a positive answer. 
		\begin{corollary}\label{corollary_radoconditionpairs}
			Let $a,b,c\in \N$ be such that either $a=c$, or $b=c$, or $a+b=c$,  and suppose that  \cref{C:2quadratic} holds. Then equation \eqref{E:abc} is partition regular with respect to all three pairs of variables $(x,y)$, $(x,z)$, $(y,z)$.
		\end{corollary}
		\begin{remark}
			Our methods can also be used to establish partition regularity for pairs of non-diagonal  homogeneous quadratic equations $ax^2+by^2+cz^2+dxy+exz+fyz=0$. 
			We don't pursue this direction here as several cases have to be considered and it is not clear if there is a natural easy to state extension of \cref{corollary_radoconditionpairs} to this generality.
		\end{remark}
		To prove the next two results we use the   partition regularity of the equations $x^2\pm ay^2=z^2$
		with respect to the variables $x,y$, which follows from Theorem~\ref{T:PairsPartition}, and argue
		as in  the proof of Corollaries 1.3 and 1.4 in \cite{FKM23} (which cover the cases $a=\pm1$).  
		\begin{corollary}
			Let $a\in \Z$. 	For every finite coloring of $\N$, there exist distinct $m,n\in \N$ such that $n^2$ and $m^2+an^2$ have the same color.  
		\end{corollary}
		\begin{corollary}
			Let $a\in \Z$. 	For every finite coloring of  the squares,  there exist distinct squares $x,y$ with the same color such that $x+ay$ is a square.  
		\end{corollary}
		\subsection{Density regularity results}
		To prove the partition regularity results formulated above, we establish stronger density versions. 
		To this end, we recall some standard notions. 
		A \emph{multiplicative F\o lner sequence in $\N$} is a sequence $\Phi=(\Phi_K)_{K=1}^\infty$ of finite subsets of $\N$ that is  asymptotically invariant under dilation, in the sense that
		$$\forall x\in\N,\qquad\lim_{K\to\infty}\frac{\big|\Phi_K \cap (x\cdot \Phi_K)\big|}{|\Phi_K|}=1.$$
		An example of a multiplicative F\o lner sequence is given by 
		$$
		\Phi_K	:=\Big\{\prod_{p\leq K}p^{a_p}\colon 0\leq  a_p\leq K\Big\}, \quad K\in \N.
		$$	
		The \emph{upper multiplicative density} of a set $\Lambda\subset\N$ with respect to a multiplicative F\o lner sequence $\Phi=(\Phi_K)_{K=1}^\infty$ is the quantity
		$$\bar{d}_\Phi(\Lambda):=\limsup_{K\to\infty}\frac{\big|\Phi_K\cap \Lambda\big|}{|\Phi_K|}.$$
		\begin{definition}
			Given $a,b,c\in \Z$, we say that the equation \eqref{E:abc} is {\em density regular with respect to $x,y$}  if for every multiplicative F\o lner sequence $\Phi$ and every set $\Lambda\subset \N$ with $\bar{d}_\Phi(\Lambda)>0$, there exist 
			distinct $x,y\in \Lambda,$ and $z\in \N$ that satisfy \eqref{E:abc}. Similarly we define the density regularity of \eqref{E:abc}  with respect to the variables $x,z$ and $y,z$.
		\end{definition}
		A finite coloring of $\N$ always contains a monochromatic cell with positive multiplicative density, so density regularity implies partition regularity. 
		Theorem~\ref{T:PairsPartition} thus follows from the following  stronger density statement. 
		\begin{theorem}\label{T:PairsDensity}
				If  $a,b,c\in \Z$ are nonzero and  $ac$ or $bc$ is a square, then \eqref{E:abc} is density  regular with respect to $x,y$.			
		\end{theorem}
		Likewise, 	Theorem~\ref{T:PairsPartitionCond} follows from the following  stronger density statement. 
		\begin{theorem}\label{T:PairsDensityCond}
			Suppose that \cref{C:2quadratic} holds.		
			If  $a,b,c\in \Z$ are nonzero and  $(a+b)c$ is a square (perhaps zero), then \eqref{E:abc} is density regular with respect to $x,y$.
		\end{theorem}
		
		\subsection{Parametric form}\label{SS:parametric}  It will be convenient to use the parametric form of the solution set of the given equations and to  work with the respective  notion of density regularity.  Two cases to keep in mind are our working examples, the equation 
		$$x^2+y^2=2z^2,$$
		which  has parametric solutions of the form (below, $k,m,n$ are arbitrary integers)
		$$
		x=k\, (m^2-n^2+2mn), \quad y=k\, (m^2-n^2-2mn),\quad  z=k\, (m^2+n^2),
		$$
		and  the equation
		$$
		x^2+2y^2=z^2,
		$$ 
		which  has parametric solutions of the form
		$$
		x=k\, (m^2-2n^2), \quad y=k\, (2mn), \quad z=k\, (m^2+2n^2).
		$$

		More generally, if $ac=d^2$ for some $d\in \N$ (the case $bc=d^2$ is similar), then a simple computation shows that the following are solutions of $ax^2+by^2=cz^2$
		$$
		x=k\, c(am^2 -bn^2),\quad
		y= k\, 2ac  mn, \quad
		z=k\,   d(am^2+bn^2), 
		$$
		$k,m,n\in \Z$.
		
		\begin{definition}
			We say that the pair  $P_1,P_2\in \Z[m,n]$ is  {\em good for density regularity} if  whenever $\Lambda\subset \N$ satisfies  $\bar{d}_\Phi(\Lambda)>0$
			for some multiplicative F\o lner sequence $\Phi$,  there exist  $m,n\in\N$, such that
			$P_1(m,n)$ and  $P_2(m,n)$ are distinct positive integers and
			$$
			\bar{d}_\Phi\big( (P_1(m,n))^{-1}\Lambda\cap (P_2(m,n))^{-1}\Lambda\big)>0.
			$$
		\end{definition}
		If the pair $P_1,P_2\in \Z[m,n]$ is good for density regularity, then for any $\Lambda\subset \N$ that satisfies  $\bar{d}_\Phi(\Lambda)>0$ for some multiplicative F\o lner sequence $\Phi$, the intersection $(P_1(m,n))^{-1}\Lambda\cap (P_2(m,n))^{-1}\Lambda$ has positive multiplicative density, and in particular is non-empty, for some $m,n\in\N$.
		Taking any $k\in(P_1(m,n))^{-1}\Lambda\cap (P_2(m,n))^{-1}\Lambda$ it follows that  $kP_1(m,n)$, $kP_2(m,n)$ are distinct and they both belong to $\Lambda$.
		Therefore, in order to prove \cref{T:PairsDensity} (and hence Theorem~\ref{T:PairsPartition}) it suffices to establish the following result.
		\begin{theorem}\label{T:PairsDensityParametric}
			Suppose that for some $\alpha,\gamma\in \N$  and nonzero $\beta \in \Z$ 
			we have 
			$$
			P_1(m,n)=\alpha m^2 +\beta n^2, \qquad P_2(m,n)=\gamma  mn.
			$$ 
			Then the pair of polynomials $P_1,P_2$ is good for density regularity.\footnote{More generally, we can prove a similar result when at least one of the non-trivial binary quadratic forms 
				$P_1(m,n)$ and $P_2(m,n)$ is reducible (and hence factors linearly) and has nonzero discriminant. Indeed,
				using linear substitutions we can  reduce  to the case 
				$P_2(m,n)=\gamma mn$ for some $\gamma\in \N$. If $P_1$ is also reducible, then we
				have established the result in \cite{FKM23}.  If $P_1$ is irreducible, then the result  
				follows from   case~\eqref{I:a} of \cref{T:DensityRegularQuadratic}. } 
		\end{theorem}
		If the polynomial $P_1$ in the statement of \cref{T:PairsDensityParametric} factors into linear terms, then the conclusion follows from our earlier work (this is explained in \cite[Section 1.5.2]{FKM23}).
		If $P_1$ is irreducible, then \cref{T:PairsDensityParametric} follows from combining part (i) of \cref{T:DensityRegularQuadratic} with \cref{T:corrvanisha} below (see the remarks following \cref{T:DensityRegularQuadratic}).

		Our next result concerns the parametric reformulation of Theorem~\ref{T:PairsDensityCond}.	 
		If $(a+b)c=d^2$ for some $d\in \Z_+$, then direct computation shows that for every $k,m,n\in \Z$, the integers
		\begin{equation}\label{E:xyz1}
			x:= k\, c (m^2-2b mn-ab n^2), \quad  y:=k\, c  (m^2+ 2a mn- ab n^2), \quad z:=k\, d(m^2+abn^2),  
		\end{equation}
		satisfy the equation $ax^2+by^2=cz^2$. 
		We use these parametric solutions when $a+b\neq 0$. If $a+b=0$, then we use the following parametric solutions
		\begin{equation}\label{E:xyz2}
			x:=k\, (m^2+acn^2),  \quad y:= k\,(m^2-acn^2), \quad 
			z:=k\,  2amn.
		\end{equation}
		Using the previous parametric solutions, it  turns out that to prove Theorem~\ref{T:PairsDensityCond}   it is sufficient to establish the following result:
		\begin{theorem}\label{T:PairsDensityParametricIrreducible}
			Suppose that for some $\alpha,\beta,\alpha',\beta'\in \Z$
			the  polynomials
			$$
			P_1(m,n)=m^2+\alpha mn+\beta n^2, \qquad P_2(m,n)=m^2+\alpha' mn+\beta' n^2
			$$
			are distinct and irreducible, and  \cref{C:2quadratic} holds
			for $P_1, P_2$. 
			Then the pair of polynomials $P_1,P_2$ is good for density regularity. 
		\end{theorem}
		It is important that the coefficients of $m^2$ (or $n^2$) in $P_1,P_2$ are equal, when this is not the case, the pair $P_1,P_2$ is typically not good for density regularity. 
		To verify that Theorem~\ref{T:PairsDensityCond} follows from \cref{T:PairsDensityParametricIrreducible} we need to consider a few cases:
		\begin{itemize}
			\item If $a+b=0$, either $|ac|$ is a square, in which case we can use \cref{T:PairsDensity}, or otherwise the polynomials $P_1(m,n)=m^2+acn^2$, $P_2(m,n)=m^2-acn^2$ in \eqref{E:xyz2} are distinct and irreducible, and hence we can use \cref{T:PairsDensityParametricIrreducible}.
			\item If $(a+b)c$ is a non-zero square, we can use \cref{T:PairsDensityParametricIrreducible} with the polynomials $P_1(m,n)=m^2-2b mn-ab n^2$ and $P_2(m,n)=m^2+ 2a mn- ab n^2$ in \eqref{E:xyz1}, provided they are irreducible (they will be distinct in this case).
			\item If $P_1$ in the previous item is not irreducible, then $b(a+b)$ is a square. Combined with the assumption that $(a+b)c$ is a square, this implies that $bc$ is a square and hence we can invoke \cref{T:PairsDensity}.
			The same argument works if $P_2$ is not irreducible.
		\end{itemize}

		Similarly, to prove \cref{T:levelsets}, it suffices to establish the following result.
		\begin{theorem}\label{T:levelsetsparam}
			Let  $P_1,P_2\in \Z[m,n]$ be the  polynomials in case  \eqref{I:b} of \cref{T:DensityRegularQuadratic},
			$f_1,\ldots, f_r\colon \N\to \S^1$ be pretentious  multiplicative functions (see \cref{D:aperiodic}), and  $I$ be an open arc of $\S^1$ containing $1$. Then there exist $k,m,n\in\N$ such that $P_1(m,n), P_2(m,n)$ are distinct, positive, and $f_j(kP_1(m,n))\in I$ and  $f_j(kP_2(m,n))\in I$ for $j=1,\ldots, r$. 
		\end{theorem}
		Thus, 	our goal is to prove Theorems~\ref{T:PairsDensityParametric}-\ref{T:levelsetsparam}, then
		Theorems~\ref{T:PairsDensity} and \ref{T:PairsDensityCond} (and hence
		Theorems~\ref{T:PairsPartition}-\ref{T:levelsets})  will follow. 
		The proof of  Theorems~\ref{T:PairsDensityParametric} and \ref{T:PairsDensityParametricIrreducible}  builds on the general strategy developed in \cite{FKM23} to handle the partition regularity of Pythagorean pairs.  After some initial maneuvering, inspired by 
		Furstenberg's ergodic proof of S\'ark\"ozy's theorem, we are led to study an integral of 
		averages of products of completely multiplicative functions evaluated on binary quadratic forms. Roughly speaking, we perform analysis resembling the ``circle method'' over the group $\mathcal{M},$ considering separately the contribution of the ``minor arcs'' 
		consisting of aperiodic multiplicative functions, which are treated using Gowers uniformity results, and the contribution of appropriately averaged ``major arcs'' consisting of pretentious multiplicative functions (in a suitable sense), which are treated using concentration estimates. The more general setting covered in this article presents additional difficulties over the ones considered in \cite{FKM23} in both the ``major arcs'' and ``minor arcs'' parts of the argument. In particular, dealing with  binary quadratic forms of positive discriminant  causes substantial problems as we have to work with number fields with infinitely many units. Moreover, dealing with two irreducible quadratic forms in \cref{T:PairsDensityParametricIrreducible} leads to subtleties that are not present in \cite{FKM23}.
		The outline of our proofs is explained in much more detail in \cref{S:Roadmap}.

		\subsection{Necessity in \cref{C:conj1}}
		\label{sec_necessaryconditions}
		\cref{C:conj1} stipulates sufficient conditions on the coefficients $a,b,c$ for the equation \eqref{E:abc} to be partition regular with respect to $x,y$, but it is not completely apparent whether those conditions are necessary. In this direction we have the following result.
		
		\begin{proposition}
			\label{P:partialconverse1}
			Let $a,b,c\in\Z$ and assume that \eqref{E:abc} is partition regular with respect to $x,y$.
			Then at least one of $ac,bc,(a+b)c, ab(a+b)c$ is a perfect square. 
		\end{proposition}
		In view of \cref{P:partialconverse1}, one may wonder whether $ab(a+b)c$ being a perfect square is a sufficient condition for \eqref{E:abc} to be partition regular with respect to $x,y$, but as the next  two results show, this is not the case.
		
		\begin{proposition}\label{P:partialconverse2}
			If $a\equiv2\bmod4$, $b$ is odd and $ab(a+b)c$ is a perfect square, then \eqref{E:abc} is not partition regular with respect to $x,y$.
		\end{proposition}
		In particular, this implies that the equation $x^2+2y^2=6z^2$ is not partition regular with respect to $x,y$, despite it satisfying the conclusion of \cref{P:partialconverse1}.
		
		\begin{proposition}\label{P:partialconverse3}
			Suppose that $a$ and $b$ are odd, $c\equiv2\bmod4$ and $ab(a+b)c$ is a non-zero perfect square.
			If \eqref{E:abc} is partition regular with respect to $x,y$, then $ab\equiv1\bmod8$.
		\end{proposition}
		For instance, the equation $3x^2+5y^2=30z^2$ is not partition regular with respect to $x,y$, but satisfies the conclusion of \cref{P:partialconverse1}.
		
		None of our results suffices to decide whether, say, any of the equations
		$$x^2+17y^2=34 z^2,\quad 8x^2+3y^2=66z^2,\quad 5x^2+7y^2=105z^2
		$$
		are partition regular with respect to $x,y$.
		
		The proofs of Propositions~\ref{P:partialconverse1}-\ref{P:partialconverse3} are elementary and are given in \cref{S:Conj1}.
		
		\subsection{Notation} \label{SS:notation}
		We let  $\N:=\{1,2,\ldots\}$,  $\Z_+:=\{0,1,2,\ldots \}$, $\Z_r:=\Z/(r\Z)$, $\R_+:=[0,+\infty)$, $\S^1\subset\C$ be the unit circle, and $\U$ be  the closed complex unit disk.
		With $\P$ we
		denote the set of primes and throughout we use the letter $p$ to denote primes.


		For $t\in \R$, we let $e(t):=e^{2\pi i t}$.
		For $z\in \C$, with $\Re(z)$, $\Im(z)$  we denote the real  and imaginary parts of $z$ respectively, and we let  $\exp(z):=e^z$.
		
		For  $N\in\N$, we let $[N]:=\{1,\dots,N\}$. We often denote sequences $a\colon \N\to \U$
		by  $(a(n))$, instead of $(a(n))_{n\in\N}$.
		
		If $A$ is a finite non-empty subset of the integers and $a\colon A\to \C$, we let
		$$
		\E_{n\in A}\, a(n):=\frac{1}{|A|}\sum_{n\in A}\, a(n).
		$$
		We write $a(n)\ll b(n)$ if for some $C>0$ we have $a(n)\leq C\, b(n)$ for every $n\in \N$.

		Throughout this article, the letter $f$ is typically  used for multiplicative functions  and the letter $\chi$ for Dirichlet characters.
		
			\subsection{Acknowledgement}  The authors would like to thank James Leng for pointing out an error in a previous version of the manuscript, which led to a modification of the weight sequences used in the proofs.
		
		\section{Roadmap to the proofs} \label{S:Roadmap}
		We shall begin with the general strategy developed in our previous work \cite{FKM23}.	In order to explain our argument we introduce  a few notions. 
		A function $f\colon \N\to \U$, where $\U$ is the complex unit disk,  is called {\em multiplicative} if
		$$
		f(mn)=f(m)\cdot f(n)  \quad \text {  whenever  }  (m,n)=1.
		$$
		Whenever necessary, we extend $f$ to an even function in $\Z$ by letting $f(0):=0$ and $f(-n):=f(n)$ for $n\in \N$.
		It is called {\em completely multiplicative} if the previous equation holds for all $m,n\in\N$.
		Let
		$$
		\CM:=\{f\colon \N\to \S^1 \text{ is a completely multiplicative function}\}.
		$$
		Throughout, we assume that $\CM$ is equipped with the topology of pointwise convergence; then $\CM$ is a metrizable compact space with this topology.
		\begin{definition}\label{D:aperiodic} We say that the multiplicative function $f\colon \N\to \U$ is 
			(see \eqref{E:DPn} for the definition of $\D(f,g)$)
			\begin{itemize}
				\item {\em pretentious} if $\D(f,\chi\cdot n^{it})<+\infty$  for some  $t\in \R$ and Dirichlet character $\chi$. 
				
				\item 	 {\em  aperiodic } if it is not pretentious.
			\end{itemize}
		\end{definition}
		\begin{remark}  A  well-known result of Daboussi-Delange~\cite[Corollary~1]{DD82}  states that a multiplicative function is aperiodic if and only if   
			$\lim_{N\to\infty}\frac1N\sum_{n=1}^N\, f(an+b)=0$
			for every $a,b\in\N$. But such a clean characterization is not possible for the more general notion of $P$-aperiodicity, which we will introduce shortly.
		\end{remark}
		We will also need the following more refined notions of aperiodicity and pretentiousness (the usual notion given in the previous definition  corresponds to the case $P(m,n)=n$ in the following definition).
		\begin{definition}\label{D:Paperiodic}
			Let  $P\in \Z[m,n]$ be a homogeneous polynomial.	 
			We say that the multiplicative function $f\colon \N\to \U$ is (see \eqref{E:DP} for the definition of $\D_P(f,g)$)
			\begin{itemize}
				\item {\em $P$-pretentious} if $\D_P(f, \chi \cdot  n^{it})<\infty$  for some $t\in\R$ and Dirichlet character $\chi$ and write $f\sim_P \chi \cdot  n^{it}$. 
				\item 
				{\em $P$-aperiodic } if it is not $P$-pretentious. 
			\end{itemize}
		\end{definition}
		\begin{remark}
			If  the multiplicative function $f$ is pretentious, then it is $P$-pretentious for every  $P\in \Z[m,n]$, and conversely, if $f$ is $P$-aperiodic for some $P\in \Z[m,n]$, then it is aperiodic. More generally, if $P_1,P_2\in \Z[m,n]$ are homogeneous polynomials such that $\omega_{P_1}(p)\geq \omega_{P_2}(p)$ (see \cref{D:omega}) for all but finitely many $p\in \P$,  then a  $P_2$-aperiodic multiplicative function  is also $P_1$-aperiodic and thus aperiodic. Conversely, a $P_1$-pretentious multiplicative function is also $P_2$-pretentious.  
		\end{remark}
		The utility of aperiodic or $P$-aperiodic multiplicative functions for our arguments is due to the expectation that they have vanishing correlations along certain polynomials. The exact  property is defined below. 
		\begin{definition}\label{D:CorVan}
			We say that the polynomials $P_1,P_2\in \Z[m,n]$ are {\em good for vanishing of correlations of aperiodic multiplicative functions} whenever   the following property holds: If  $f_1,f_2\colon \N\to \U$ are multiplicative functions  such that either  $f_1$ is $P_1$-aperiodic or $f_2$ is $P_2$-aperiodic,
			then 	for every $Q\in \N$ and $a,b\in \Z_+$ we have 
			$$
			\lim_{N\to\infty}\E_{m,n\in [N]} \,{\bf 1}_S(Qm+a,Qn+b)\cdot  f_1(P_1(Qm+a,Qn+b)) \cdot f_2(P_2(Qm+a,Qn+b))=0,
			$$
			where $S:=\{m,n\in\N\colon P_1(m,n)>0,P_2(m,n)>0\}$. 
		\end{definition}
		
		The next result will play a crucial role in establishing  \cref{T:PairsDensityParametric} and immediately implies \cref{T:PairsDensityParametricIrreducible}.
		\begin{theorem}\label{T:DensityRegularQuadratic}
			Suppose that for some $\alpha,\beta,\alpha',\beta',\gamma\in \Z$ and  $\ell,\ell'\in \N$, the polynomials $P_1,P_2\in \Z[m,n]$ satisfy one of the following conditions:
			\begin{enumerate} \item\label{I:a}  $P_1(m,n)=\ell (m^2+\alpha mn+\beta n^2), \, P_2(m,n)=\ell' (m+\gamma  n)n$, and $P_1$ is irreducible; 
				
				\item \label{I:b}  $P_1(m,n)= m^2+\alpha mn+\beta n^2, \, P_2(m,n)= m^2+\alpha' mn+\beta' n^2,$  and $P_1,P_2$ are distinct and  irreducible.
			\end{enumerate}
			Suppose also that the pair $(P_1,P_2)$ is good 	for vanishing of correlations of aperiodic multiplicative functions.
			Then the pair $(P_1,P_2)$ is good for density regularity. 
		\end{theorem}
		To see why the polynomials $P_1,P_2$ described in part \eqref{I:a} of \cref{T:DensityRegularQuadratic} contain those in \cref{T:PairsDensityParametric}, suppose $\tilde P_1(m,n)=\tilde\alpha m^2+\tilde\beta n^2$ and $\tilde P_2(m,n)=\tilde\gamma mn$.
		Then letting $\ell=\tilde\alpha$, $\beta=\tilde\alpha\tilde\beta$, $\alpha=\gamma=0$, and $\ell'=\tilde\alpha\tilde\gamma$ it follows that $P_j(m,n)=\tilde P_j(m,\tilde\alpha n)$, for $j=1,2$.
		If $(P_1,P_2)$ is good for density regularity, then so is $(\tilde P_1,\tilde P_2)$.
		Therefore, in order to prove Theorem~\ref{T:PairsDensityParametric} it suffices to combine case~\eqref{I:a} 
		of \cref{T:DensityRegularQuadratic} with the next result.
		
		\begin{theorem}\label{T:corrvanisha}
			If  $P_1,P_2\in \Z[m,n]$ are as in  case~\eqref{I:a} of Theorem~\ref{T:DensityRegularQuadratic}, then they are good for vanishing of correlations of aperiodic multiplicative functions. 
		\end{theorem}
		We prove a more general result in \cref{T:aperiodicNew}. The key idea is to use a variant of an orthogonality criterion of Daboussi-K\'atai for quadratic number fields (see Corollary~\ref{C:Katai}, which is easily deduced  from a recent result in \cite{Su23}) in order to eventually reduce matters to a statement about vanishing of correlations of products of aperiodic multiplicative functions evaluated on linear forms (see \cref{T:aperiodicKnown}, which is borrowed from \cite{FH17}).

		\subsection{Proof plan for Theorem~\ref{T:DensityRegularQuadratic}} \label{SS:Plan}
		Following the argument used in \cite[Section~2.1]{FKM23}, we see that  to prove \cref{T:DensityRegularQuadratic},  it is sufficient  to establish the next result:  \rem{From this point on I started adding the parameter $c$ in various places}
		\begin{theorem}\label{T:MainMulti'}
			Let $\sigma$ be a  bounded Borel measure on $\CM$ such that $\sigma(\{1\})>0$ and
			\begin{equation}\label{E:positivers}
				\int_\CM f(r)\cdot \overline{f(s)}\, d\sigma(f)\geq 0\quad  \text{for every } r,s\in \N.
			\end{equation}
			Let also $P_1,P_2\in \Z[m,n]$ satisfy the assumptions of \cref{T:DensityRegularQuadratic} and
			$w_{\delta,c,P_1,P_2}(m,n)$ be the weight defined in \eqref{E:weight2} of Lemma~\ref{L:SdeltaGeneral}.	Then   there exist $\delta\in (0,1/2)$ and $c\in \R$, such that 
			\begin{equation}\label{E:sigmapositive2}
				\liminf_{N\to\infty} \E_{m,n\in[N]}\, w_{\delta,c,P_1,P_2}(m,n)\cdot \int_{\CM} f\big(P_1(m,n)\big)\cdot \overline{f\big(P_2(m,n)\big)}\, d\sigma(f)>0.
			\end{equation}
Furthermore, in case \eqref{I:b} of \cref{T:DensityRegularQuadratic} we can take $c=0$. 
		\end{theorem}
		\begin{remarks}
			$\bullet$ 	Since a homogeneous polynomial $P\in \Z[m,n]$ that vanishes on a set of positive upper density (with respect to squares $[N]\times [N]$) vanishes identically, we deduce that  if \eqref{E:sigmapositive2}
			holds for distinct polynomials $P_1,P_2$, then there exist $m,n\in \N$ such  that $P_1(m,n)$ and $P_2(m,n)$ are  positive, distinct, and $\int_{\CM} f(P_1(m,n))\cdot \overline{f(P_2(m,n))}\, d\sigma(f)>0$.
			
			$\bullet$ 	We show in \cref{T:corrvanisha} 				
			that if  $P_1,P_2\in \Z[m,n]$ are as in  case~\eqref{I:a} of \cref{T:DensityRegularQuadratic}, then they are good  for vanishing 
			of correlations of aperiodic multiplicative functions and so in this case   \eqref{E:sigmapositive2} holds unconditionally.
			
			$\bullet$ 	Our argument shows that, unconditionally,  if 
			the polynomials 
			are as in  case~\eqref{I:b} of \cref{T:DensityRegularQuadratic} and the measure $\sigma$ is supported on the set $\CM_p$ of pretentious multiplicative functions defined in \eqref{E:Mp} (or the larger set $\CM_{p,P_1,P_2}$ defined in \eqref{E:MP12}) and satisfies $\sigma(\{1\})>0$,  then \eqref{E:sigmapositive2} holds. We will show in \cref{SS:levelset} how to derive  \cref{T:levelsetsparam}  from this fact.
		\end{remarks}
		To prove \cref{T:MainMulti'} it will be crucial to restrict the values of the polynomials $P_1,P_2$ to suitable arithmetic progressions in a uniform way (this is needed to prove \cref{P:nonnegative}). To this end, we introduce  the following notation:  for $\delta>0$, $c\in \R$, $Q\in\N$, $a,b\in \Z_+$,  and multiplicative function $f\colon \N\to \U$, let\footnote{We could instead use the weight $\tilde{w}_{\delta,c,P_1,P_2}(Qm+a,Qn+b)$ in the definition of 	$L_{\delta,c,N}$, the reason we use $\tilde{w}_{\delta,c,P_1,P_2}(m,n)$ is that   later on in our argument we crucially use that the weight does not depend on $Q$. In any case, in view of \cref{L:Pab} given below, it follows that  the two weights can be used interchangeably.}   
		\begin{multline}\label{E:LQdab}
			L_{\delta,c,N}(f,Q;a,b):=\\
			\E_{m,n\in [N]}\,  \tilde{w}_{\delta,c,P_1,P_2}(m,n)\cdot f\big(P_1(Qm+a,Qn+b)\big)\cdot \overline{f\big(P_2(Qm+a,Qn+b)\big)},
		\end{multline}
		(note that $L_{\delta,c,N}$ also depends on $P_1,P_2$ but we suppress this dependence)
		where the normalized weights $\tilde{w}_{\delta,c,P_1,P_2}$ are given by 
		\begin{equation}\label{E:wtilde}
			\tilde{w}_{\delta,c,P_1,P_2}(m,n):= 	\mu^{-1}_{\delta,c,P_1,P_2} \cdot w_{\delta,c,P_1,P_2}(m,n),  
		\end{equation}
		where 
		$$
		\mu_{\delta,c,P_1,P_2}:=	\lim_{N\to \infty} \E_{m,n\in [N]}\,  w_{\delta,c,P_1,P_2}(m,n).
		$$
		Note that for every $\delta\in (0,1/2)$ the last limit exists and is positive by \cref{L:SdeltaGeneral} below
		and our assumptions about the polynomials $P_1,P_2$. 
		It follows from the positivity property \eqref{E:positivers}, \cref{L:Pab}, and Fatou's lemma, that in order to establish
		\eqref{E:sigmapositive2} it suffices to show that there exist $\delta\in (0,1/2)$, $c\in \R$,  $Q\in \N$, and  $a,b\in \Z_+$,  such that
		\begin{equation}\label{E:sigmapositive1p}
			\int_{\CM} \liminf_{N\to\infty} \Re\big( L_{\delta,c,N}(f,Q; a,b)\big)\, d\sigma(f)>0.
		\end{equation}
		In fact, to prove case~\eqref{I:a} of \cref{T:DensityRegularQuadratic}, we will take $a=1$, $b=0$, but for case \eqref{I:b} a more complicated choice of $a,b$ is needed that depends on $Q$.
		
		To establish \eqref{E:sigmapositive1p}, we use the theory of completely multiplicative functions. In the aperiodic case, which roughly speaking corresponds to the ``minor arcs'' in $\mathcal{M},$ we use the following result. 
		\begin{proposition}\label{P:aperiodicQ12}
			Let $P_1,P_2$ be homogeneous polynomials of the same degree such that the pair  $(P_1,P_2)$ is good for vanishing of correlations of aperiodic multiplicative functions.	Let $f\colon \N\to \U$ be an aperiodic multiplicative function that is  either $P_1$ or $P_2$ aperiodic. Then  for every $\delta\in (0,1/2)$, $c\in \R$,    $Q\in \N$, and $a,b\in \Z_+$,  we have
			\begin{equation}\label{E:wdmnB}
				\lim_{N\to\infty} L_{\delta,c,N}(f,Q; a,b)=0.
			\end{equation}
		\end{proposition}
		
		The positivity property for integrals in \eqref{E:sigmapositive1p}
		will follow from positivity properties  on the sets $\CM_p$ and $\CM_{p,P_1,P_2}$ respectively.
		To state these properties, we need to introduce some additional notation. For each $K\in\N$ let
		\begin{equation}\label{E:PhiK'}
			\Phi_K	:=\Big\{\prod_{p\leq K}p^{a_p}\colon K< a_p\leq 2K\Big\}.
		\end{equation}
		The sequence $(\Phi_K)$ is a multiplicative F\o lner sequence with the property that for every $q\in\N$, if $K$ is sufficiently large then for every $Q\in\Phi_K$ we have  $q|Q$. 
		
		Let 
		\begin{equation}\label{E:Mp}
			\CM_p:=\{f\colon \N\to \S^1\colon f \text{ is pretentious}\}.
		\end{equation}
		Furthermore, if $P_1,P_2\in \Z[m,n]$ are homogeneous polynomials  let 
		\begin{equation}\label{E:MP12}
			\CM_{p,P_1,P_2}:=\{f\colon \N\to \S^1\colon f \text{ is } P_1 \text{ and } P_2 \text{ pretentious}\}.
		\end{equation}
		In \cref{L:Borel}  we show that $\CM_p$ and $\CM_{p,P_1,P_2}$ are  Borel subsets of $\CM$.
	We also let 
	\begin{equation}\label{E:CA}
		\CA:=\{n^{it}\colon t\in \R\}.
	\end{equation}

		We now give the key non-negativity property, which  is proved using known concentration estimates for linear forms (as in \cref{P:concentration1})  and new concentration estimates for irreducible binary quadratic forms (as in \cref{P:concentration2General}).
		
		\rem{Change below, in the first case we have to  integrate,  in the second it remains a positivity property for every $f$.}
		
		\begin{proposition}\label{P:nonnegative}
			Let $P_1,P_2\in \Z[m,n]$ be as  in  Theorem~\ref{T:DensityRegularQuadratic} and 	$L_{\delta,c,N}(f,Q;a,b)$ be as in \eqref{E:LQdab}.\footnote{In both of the following statements,  our argument does not allow us to exchange $\liminf_{N\to\infty}$ and $\E_{Q\in \Phi_K}$.} 
			\begin{enumerate}
				\item\label{I:anonnegative} In case \eqref{I:a} of Theorem~\ref{T:DensityRegularQuadratic},  if $\sigma$ is a bounded Borel measure on $\CM_p$ with $\sigma(\{1\})>0$, then 
			there exists a real number $c_0=c_0(\sigma,P_1,P_2)$	such that 
				$$
				\liminf_{\delta\to 0^+} \liminf_{K\to\infty}	\liminf_{N\to \infty}   \E_{Q\in \Phi_K} \Re\Big(\int_{\CM_p}	L_{\delta,c_0, N}(f, Q; 1,0) \, d\sigma(f)    \Big)> 0,
				$$
				where  $(\Phi_K)$ is as in  \eqref{E:PhiK'}.
				
				\item \label{I:bnonnegative}  In case \eqref{I:b} of Theorem~\ref{T:DensityRegularQuadratic}, for every $f\in \CM_{p,P_1,P_2}$  
				we have 
				$$
				\liminf_{r\to\infty}	\liminf_{\delta\to 0^+} \liminf_{K\to\infty} 		\liminf_{N\to \infty}
				\E_{Q_1\in \Phi_{1,K}, Q_2\in \Phi_{2,K}}\Re\big(L_{\delta,0, N}(f, Q_K; a_{r,K,Q_1,Q_2},b_{r,K,Q_1,Q_2})  \big)\geq 0,
				$$
				for suitably chosen $Q_K\in \N$, $K\in \N$ (as in \eqref{E:QK}), 
				integer subsets $(\Phi_{j,K})_{K\in \N}$, $j=1,2$ (as in \eqref{E:PhijK}), and  $a_{r,K,Q_1,Q_2},b_{r,K,Q_1,Q_2}\in [Q_K]$ (as in \cref{P:PQ12}). 
			\end{enumerate}
		\end{proposition}
\rem{New paragraph}	
To establish  part~\eqref{I:anonnegative}, we split the integral over $\CM_p$ into the Archimedean part $\CA$ (see \eqref{E:CA}) and its complement. The Archimedean part is treated only after integration. Writing $f(n)=n^{it_f}$ on $\CA$, the weight $\tilde{w}_{\delta, c, P_1,P_2}$  makes the contribution of $L_{\delta,c,N}(f,Q;1,0)$ close to $e^{ic t_f}$, after first letting $N\to\infty$ and then $\delta\to 0^+$. An averaging argument is then used to find  a value of $c$ for which the real part of the integral over $\CA$ has the required positivity property.
	 
 On the complementary part $\CM_p\setminus \CA$, we have    $f\in \CM_p$ but $f\neq n^{it}$ for every $t\in \R$.
		Using concentration estimates we manage to show that for $Q\in \Phi_K$ the averages $L_{\delta, c,N}(f, Q; 1,0)$ are asymptotically equal to   $f(Q)\cdot Q^{-it}$  times an expression that depends only on  $K$. Using this, it is easy to show that   after taking multiplicative averages over $Q\in \Phi_K$  the resulting expression vanishes pointwise for $f\in\CM_p\setminus \CA$. Combining the above, we get the claimed positivity.
		
		In  part~\eqref{I:bnonnegative} it is not immediately  clear how to  implement the previous strategy, since  the irreducibility of the polynomials makes it difficult  to factor out a term
		$f(Q)\cdot Q^{-it}$. 
		The key observation that lets us get started, is to restrict the range of $m,n$ to $Q\N+a,Q\N+b$ for
		values of $a,b$ that  satisfy suitable congruences (see \cref{L:Pi12}) and apply concentration estimates
		along this lattice. Suppose that for $j=1,2$ we have $f\sim \chi_j \cdot  n^{it}$  for some $t\in \R$ and Dirichlet characters $\chi_j$.
		We get that  if for $j=1,2$ all the prime factors of the  integers
		$Q_j$ belong to  certain sets $\CP_j$ ($\CP_j$ are the primes  $p$ for which the equation $P_j(m,1)\equiv 0 \pmod{p}$ has a solution),
		then for $Q_j\in \Phi_{j,K}$ the   averages $L_{\delta,0, N}(f,Q;a,b)$, for suitable $Q=Q_K$ and 
		$a,b\in \N$, depending on $Q_1,Q_2$, are asymptotically equal to 
		$\tilde{f}_1(Q_1)\cdot \overline{\tilde{f}_2(Q_2)}$ times an expression depending only on $K$, where 
		$\tilde{f}_j(n):=f(n)\cdot \overline{\chi_j(n)}\cdot n^{-it}$. We conclude that if for $j=1$ or $2$ we have $\tilde{f}_j(p)\neq 1$ for some prime  $p\in \CP_j$, then the multiplicative averages over $Q_1\in \Phi_{1,K}$ and $Q_2\in \Phi_{2,K}$ ($\Phi_{j,K}$ are as in \eqref{E:PhijK})
		of $L_{\delta,0, N}(f, Q_K; a_{r,K,Q_1,Q_2},b_{r,K,Q_1,Q_2})$ vanish. 
		On the other hand, if for $j=1$ and $2$ we have $\tilde{f}_j(p)= 1$ for all primes  $p\in \CP_j$, then after using concentration estimates, and taking advantage of the fact  that in this case some unwanted oscillatory terms cancel out, we get that  
		for $\delta$ small enough  the averages
		$L_{\delta,0, N}(f, Q_K; a_{r,K,Q_1,Q_2},b_{r,K,Q_1,Q_2})$  are  positive. Combining the above we get the asserted non-negativity for all $f\in \CM_{p,P_1,P_2}$.

		We conclude this section by noting how the previous results allow us to reach our goal, which is to
		prove  Theorem~\ref{T:MainMulti'},  thus completing the proof of 
		Theorem~\ref{T:DensityRegularQuadratic}.
		\begin{proof}[Proof of Theorem~\ref{T:MainMulti'} assuming \cref{P:aperiodicQ12} and \cref{P:nonnegative}]
			We assume that the polynomials $P_1,P_2\in \Z[m,n]$ are as in case \eqref{I:b} of Theorem~\ref{T:DensityRegularQuadratic},
			the argument is similar  and slightly simpler if they are as in case \eqref{I:a} of Theorem~\ref{T:DensityRegularQuadratic}
			(the only difference is that we use part~\eqref{I:anonnegative} of \cref{P:nonnegative} instead of part~\eqref{I:bnonnegative}).
			
			Note that  if  $f\notin \CM_{p,P_1,P_2}$,  then $f$ is either $P_1$-aperiodic or $P_2$-aperiodic, hence
			\cref{P:aperiodicQ12} gives that  for every $\delta\in (0,1/2)$ and  $Q\in \N$, $a,b\in \Z_+$,  we have
			$$
			\lim_{N\to\infty}L_{\delta,0,N}(f,Q; a,b)=0.
			$$
			It follows that 
			the  non-negativity property in part~\eqref{I:bnonnegative} of \cref{P:nonnegative}  extends to all multiplicative functions $f\in \CM$. 	Using this 
			and 	Fatou's lemma we get (we follow the notation of \cref{P:nonnegative}) 
			\begin{multline}\label{E:nonnegative}
				\liminf_{r\to\infty}	\liminf_{\delta\to 0^+} \liminf_{K\to\infty} 	\liminf_{N\to \infty}\\
				\E_{Q_1\in \Phi_{1,K}, Q_2\in \Phi_{2,K}}
				\Re\big(\int_{\CM\setminus \{1\}} L_{\delta, 0, N}(f, Q_K; a_{r,K,Q_1,Q_2},b_{r,K,Q_1,Q_2})  \, d\sigma(f)\big)\geq 0.
			\end{multline}
			
			We deduce 
			from  \eqref{E:nonnegative} that 
			there exist $\delta_0>0$ and $r_0,K_0\in \N$, depending only on $\sigma$, $P_1$, $P_2$, such that
			\begin{equation}\label{E:nonnegative'}
				\liminf_{N\to \infty}
				\E_{Q_1\in \Phi_{1,K_0}, Q_2\in \Phi_{2,K_0}}\, 	\int_{\CM\setminus \{1\}} L_{\delta_0,0, N}(f, Q_{K_0}; a_{r_0,K_0,Q_1,Q_2},b_{r_0,K_0,Q_1,Q_2})  \, d\sigma(f)>- \frac{\rho}{2},
			\end{equation}
			where 
			$$
			\rho:=\sigma(\{1\})>0.
			$$
			Note that for every  $\delta\in (0,1/2)$, $Q\in \N$, $a,b\in \Z_+$, we have 
			$$
			\lim_{N\to\infty}L_{\delta,0, N}(1, Q; a,b)=\mu^{-1}_{\delta,0,P_1,P_2}\cdot \lim_{N\to\infty}\E_{m,n\in[N]}\, w_{\delta,0,P_1,P_2}(m,n)=1.
			$$ 
			Using this 
			and  the lower bound \eqref{E:nonnegative'}, we deduce  that 				
			$$
			\liminf_{N\to \infty}
			\E_{Q_1\in \Phi_{1,K_0}, Q_2\in \Phi_{2,K_0}}\, 	\int_{\CM} L_{\delta_0,0, N}(f, Q_{K_0}; a_{r_0,K_0,Q_1,Q_2},b_{r_0,K_0,Q_1,Q_2})  \, d\sigma(f)> \frac{\rho}{2}.
			$$
			(The real part is no longer needed since all integrals are known to be real by \eqref{E:positivers}.)
			
			The last estimate implies that there exist   $Q_{j,N}\in \Phi_{j,K_0}$, $j=1,2$, and      $a_{r_0,K_0,Q_{1,N},Q_{2,N}}$, $b_{r_0,K_0,Q_{1,N},Q_{2,N}}\in [Q_{K_0}]$, $N\in \N$,  such that
			\begin{equation}\label{E:Lrho0}
				\liminf_{N\to \infty}
				\int_{\CM} L_{\delta_0,0, N}(f, Q_{K_0}; a_{r_0,K_0,Q_{1,N},Q_{2,N}},b_{r_0,K_0,Q_{1,N},Q_{2,N}})  \, d\sigma(f)\geq\frac{\rho}{2}.
			\end{equation}
			
			Using \cref{L:Pab} below and since the integers $a_{r_0,K_0,Q_{1,N},Q_{2,N}}$, $b_{r_0,K_0,Q_{1,N},Q_{2,N}}$ belong to a finite set that is independent of $N$, we get that the same estimate holds if  in the definition of 
			$L_{\delta_0, 0, N}(f, Q_{K_0}; a_{r_0,K_0,Q_{1,N},Q_{2,N}},b_{r_0,K_0,Q_{1,N},Q_{2,N}})$
			we replace the weight  $\tilde{w}_{\delta_0,0,P_1,P_2}(m,n)$ with 
			$\tilde{w}_{\delta_0,0,P_1,P_2}(Q_{K_0}m+a_{r_0,K_0,Q_{1,N},Q_{2,N}},Q_{K_0}n+b_{r_0,K_0,Q_{1,N},Q_{2,N}})$. 	
			We use  this,    the positivity property
			$$
			\tilde{w}_{\delta_0,0, P_1,P_2}(m,n)\int_\CM  f\big(P_1(m,n)\big)\cdot \overline{f\big(P_2(m,n)\big)}\, d\sigma(f)\geq 0 \quad \text{for every } m,n\in\N,
			$$
			which follows from  \eqref{E:positivers}, and the fact that
			the integers $a_{r_0,K_0,Q_{1,N},Q_{2,N}}$, $b_{r_0,K_0,Q_{1,N},Q_{2,N}}$ belong to a finite set that is independent of $N$. We deduce from these facts and 
			\eqref{E:Lrho0} that
			$$
			\liminf_{N\to\infty} \E_{m,n\in[N]}\,\tilde{w}_{\delta_0,0,P_1,P_2}(m,n)\cdot \int_{\CM} f\big(P_1(m,n)\big)\cdot \overline{f\big(P_2(m,n)\big)}\, d\sigma(f)\geq
			\frac{\rho }{2(Q_{K_0})^2}.
			$$
			This establishes \eqref{E:sigmapositive2}  and completes the proof of \cref{T:MainMulti'}.
		\end{proof}
		To prove  Theorem~\ref{T:MainMulti'}, it remains  to prove \cref{P:aperiodicQ12} and  \cref{P:nonnegative}. The proof of  \cref{P:aperiodicQ12}  is relatively simple and is given in Section~\ref{SS:proof1}. 
		The proof of \cref{P:nonnegative}
		is much harder, and we will complete it  in Section~\ref{S:PRproofs}, after having  established the concentration estimates for binary quadratic forms, stated in 
		Proposition~\ref{P:concentration2General} and proved  in Sections~\ref{S:Concentration1} and \ref{S:Concentration2}.
		Finally,  \cref{T:corrvanisha}, which is used to complete the proof of \cref{T:PairsDensityParametric},  is proved in Section~\ref{S:Aperiodic}.

		\section{Vanishing of correlations for  aperiodic multiplicative functions}\label{S:Aperiodic} 
		\subsection{Statement of main result} Our goal is to prove 
		\cref{T:corrvanisha}, which is covered by the following more general result. 
		\begin{theorem}\label{T:aperiodicNew}
			Let  $s\in \N$ and suppose  that the completely multiplicative functions  $f_1,\ldots, f_s,g\colon \N\to \mathbb{U}$ are extended to even sequences in $\Z$ and that $f_1$ is aperiodic.  Let
			$L_1,\ldots, L_s$ be linear forms in two variables with integer coefficients  and suppose that either $s=1$ and $L_1$ is non-trivial, or $s\geq 2$ and the forms $L_1,L_j$ are linearly independent for $j=2,\ldots,s$.  Let also  
			$$
			P(m,n):=\alpha m^2+\beta mn+\gamma n^2,
			$$
			where $\alpha,\beta, \gamma\in \Z$,   	  	   be 
			irreducible, i.e., $D:=\beta^2-4\alpha \gamma$ is not a square, and $Q\in \N$, $a,b\in \Z_+$.
			Finally, let $K$ be  a convex subset of $\R_+^2$  that is homogeneous, i.e. $(x,y)\in K$ implies $(kx,ky)\in K$ for every $k\in \N$.  Then
			$$
			\lim_{N\to\infty} \E_{m,n\in [N]}\, {\bf 1}_{K}(Qm+a,Qn+b)  \cdot \prod_{j=1}^sf_j(L_j(Qm+a,Qn+b))\cdot g(P(Qm+a,Qn+b))=0.
			$$
		\end{theorem}
		\begin{remark}
			If $P_1,P_2\in \Z[m,n]$  are as in case~\eqref{I:a} of \cref{T:DensityRegularQuadratic}, then $P_1$ is an irreducible binary quadratic form and  $P_2(m,n)=L_1(m,n)\cdot L_2(m,n)$ where $L_1,L_2$ are pairwise independent linear forms. 	Applying \cref{T:aperiodicNew} for $s:=2$, $f_1:=f_2$, $P:=P_1$, and
			$K=\{(x,y)\in \R^2_+ \colon P_j(x,y)>0, j=1,2\}$, which is a disjoint union of finitely many homogeneous convex subsets of $\R_+^2$, we get the conclusion of \cref{T:corrvanisha}.   
		\end{remark}
		In the case of completely multiplicative functions, 
		\cref{T:aperiodicNew}  generalizes \cite[Theorem~9.7]{FH17} that covered  the case  where $P(m,n)=m^2+dn^2$ for $d\equiv 1,2\! \pmod{4}$ and
		$
		P(m,n)=m^2+mn+\frac{d+1}{4}n^2$ for  $d\equiv 3\! \pmod{4}$ with $d>0;$ see also \cite{M18} for a quantitative aspect of this result for linear forms.  
		The case of binary quadratic forms with negative discriminant can be reduced to the previous cases by a simple trick explained in Section~\ref{SS:PfThm31}. On the other hand, when the binary quadratic form $P$ has a positive discriminant, new difficulties arise  that were not present
		in \cite{FH17}. The problem in this case is to find a suitable
		variant of the Daboussi-K\'atai orthogonality criterion 
		that takes into account  the existence of infinitely many units
		of the number field on which $P$ can be realized as  a norm form. This
		has recently been established in \cite{Su23}; we describe the construction in Section~\ref{SS:Katai} and adapt it for our purposes. 
		
		\subsection{Known result about aperiodic multiplicative functions}
		We will use the following result from \cite{FH17}.
		\begin{lemma}\label{T:aperiodicKnown}
		Let   $s\in \N$  and suppose that the completely multiplicative functions  $f_1,\ldots, f_s$, $g\colon \N\to \mathbb{U}$ are extended to even sequences in $\Z$ and that $f_1$ is aperiodic. Let
			$L_1,\ldots, L_s$ be linear forms in two variables with integer coefficients  and suppose that either $s=1$ and $L_1$ is non-trivial, or $s\geq 2$ and the forms $L_1,L_j$ are linearly independent for $j=2,\ldots,s$.   
			Then  for every $r\in \N$ we have 
			$$
			\lim_{N\to\infty} \E_{m,n\in [rN]}\, {\bf 1}_{K_N}(m,n)\cdot  \prod_{j=1}^sf_j(L_j(m,n)) =0,
			$$
			where  for $N\in \N$  we have  $K_N:=K\cap A_1([N]\times[N])\cap A_2([N]\times[N])$ 
			for some convex subset $K$ of $\R_+^2$ and  linear transformations $A_1,A_2\colon \Z^2\to \Z^2$ with nonzero determinant. 
		\end{lemma}
		If $K_N$ are  convex subsets of $[0,rN]\times [0,rN]$ 
		the result follows by combining \cite[Theorem~2.5]{FH17} and \cite[Lemma~9.6]{FH17}. The argument given in the proof of \cite[Lemma~9.6]{FH17} can be adapted in a straightforward way to obtain the advertised statement.

		\subsection{Quadratic number fields - Notation}
		Let $d\in \Z$ be a square-free integer. If $\mathcal{O}_K$ is 
		the ring of integers (i.e. solutions of monic polynomials with coefficients in $\mathbb{Z}$) of a quadratic number field $K=\mathbb{Q}(\sqrt{-d})$, then there exists $\tau_d\in \mathcal{O}_K$ such that every element $z\in 
		\mathcal{O}_K$ can be uniquely expressed in the form $m+n\tau_d$ for some $m,n\in \Z$. In fact, we can take 
		\begin{equation}\label{E:taud}
			\tau_d:=\begin{cases} \sqrt{-d}, \quad & \, \text{ if } d\equiv 1,2\! \pmod{4} \\
				\frac{1+\sqrt{-d}}{2}, & \,  \text{ if } d\equiv 3\! \pmod{4}.
			\end{cases}
		\end{equation}
		We let 
		\begin{equation}\label{E:RN}
			R_N:=\{z\in \mathcal{O}_K\colon z=m+n\tau_d, m,n\in [N]\}.
		\end{equation}
		If $z=m+n\sqrt{-d}$  we  let $\overline{z}:=m-n \sqrt{-d}$ and  $\mathcal{N}(z):=z\cdot \overline{z}=m^2+dn^2$. Note that  if $z=m+n\tau_d$, then
		$\mathcal{N}(z)=P_d(m,n)$ where 
		\begin{equation}\label{E:Qd}
			P_d(m,n):=\begin{cases} m^2+dn^2, \quad & \, \text{ if } d\equiv 1,2\! \pmod{4} \\
				m^2+mn+\frac{d+1}{4}n^2, & \,  \text{ if } d\equiv 3\! \pmod{4}.
			\end{cases}
		\end{equation}
		\subsection{Variant of the Daboussi-K\'atai orthogonality criterion}\label{SS:Katai} 
		An important fact used  in the proof of the Daboussi-K\'atai orthogonality criterion  on $\Z$~\cite{D74,Ka} is that for $r\in \N$ the cardinality of the 
		set $r^{-1}[N]:=\{n\in\N\colon rn\in [N]\}$ is approximately $N/r$. 
		A similar statement holds for Gaussian integers, namely, if $R_N=\{m+in\colon  m,n\in[N]\}$ and $z\in \Z[i]$ is nonzero, then 
		$$
		z^{-1}R_N=\{w\in \Z[i]\colon zw\in R_N\} \subset R_{C\cdot \mathcal{N}(z)^{-\frac{1}{2}}N}
		$$ for some $C>0$.
		Similar inclusions also hold for all other quadratic fields with negative discriminant,
		but this is not the case for quadratic number fields with positive discriminant, which have infinitely many units (elements $\epsilon\in \mathcal{O}_K$ with $\mathcal{N}(\epsilon)=1$).
		In this case, every element $z\in\mathcal{O}_K$ has infinitely many associates $z\epsilon$, all of which have the same norm. 
		While not all of them satisfy the required property, Sun proved in \cite{Su23} that at least one of the associates of any given element does.
		We state here the result for quadratic number fields, since that  is all we need. 
		\begin{definition}[Definition~2.12 in \cite{Su23}]
			Let $\Z[\tau_d]$ be the ring of integers of a quadratic number field $\Q(\sqrt{-d})$ and $C>0$.  We say that $z\in\Z[\tau_d]$, $z\neq 0$,  is {\em $C$-regular} if for all $N\in \N$ we have 
			$$
			z^{-1}  R_{N}\subset R_{C|\mathcal{N}(z)|^{-\frac{1}{2}}N},
			$$
			where $R_N$ is as in \eqref{E:RN}.
		\end{definition}
		We will use the following special case of a result from \cite{Su23}.
		\begin{lemma}[Theorem~2.13 in \cite{Su23}]\label{L:Cregular}
			Let $\Z[\tau_d]$ be the ring of integers of a quadratic number field $\Q(\sqrt{-d})$. Then there exists $C>0$, depending on $d$ only, such that for every $z\in \Z[\tau_d]$ there exists a unit $\epsilon\in \mathcal{O}_K$ such that $\epsilon z$ is $C$-regular. 
		\end{lemma}
		
		Using this, the following extension of the Daboussi-K\'atai  orthogonality criterion was established in \cite{Su23} (in the setting of general number fields, but  we only need it for quadratic number fields).
		\begin{lemma}[Lemma~2.20 in \cite{Su23}]\label{L:KataiWenbo}
			Let $d\in\Z$ be square-free, $\Z[\tau_d]$ be the ring of integers of the quadratic number field $\Q(\sqrt{-d})$, and  $a\colon \Z[\tau_d]\to \C$ be a bounded sequence. Let also $C>0$ and $\CP$ be a collection of $C$-regular prime elements of $\Z[\tau_d]$  such that 
			$$
			\sum_{p\in \CP}\frac{1}{|\CN(p)|}=+\infty.
			$$
			Suppose that for all  $p,q\in \CP$ with coprime norms we have 
			$$
			\lim_{N\to\infty}\E_{z\in R_N}\, a(pz)\, \overline{a(qz)}=0.
			$$
			Then for every completely multiplicative function $f\colon \Z[\tau_d]\to \U$ we have 
			$$
			\lim_{N\to\infty}\E_{z\in R_N}\, a(z)\, f(z)=0.
			$$ 
		\end{lemma}
		\begin{remark}
			The case $d>0$ was previously covered by \cite[Proposition~9.5]{FH17}.
		\end{remark}

		We will use this lemma in the form of the following corollary that is better suited for our purposes.
		
		\begin{corollary}[Daboussi-K\'atai orthogonality criterion for quadratic number fields]\label{C:Katai}
			Let $d\in\Z$ be square-free, $\Z[\tau_d]$ be the ring of integers of the quadratic number field $\Q(\sqrt{-d})$, $a\colon \Z[\tau_d]\to \C$ be a bounded sequence, and $P_0>0$. 
			Suppose that for all $\alpha,\beta$ that are  prime elements of $\Z[\tau_d]$ with  $|\mathcal{N}(\alpha)|>P_0$  and $|\mathcal{N}(\beta)|>P_0$  and such that all elements $\alpha,\overline{\alpha},\beta,\overline{\beta}$ are pairwise  non-associates, 
			we have 
			$$
			\lim_{N\to\infty}\E_{z\in R_N}\, a(\alpha z)\, \overline{a(\beta z)}=0.
			$$
			Then for every completely multiplicative function $f \colon \Z[\tau_d]\to \U$ we have 
			$$
			\lim_{N\to\infty}\E_{z\in R_N}\, a(z)\, f(z)=0.
			$$ 
		\end{corollary}
		
\begin{proof}\rem{More explanations added. The previous argument  was using a property that was not true for every prime. }
By the Dirichlet density theorem for prime ideals in ideal classes
(see for example \cite[Exercise~11.2.8]{ME05}), applied to the principal
ideal class, there are enough principal prime ideals to ensure that the
sum of the reciprocals of their norms diverges. By \cite[Exercise~5.4.5]{ME05},
only finitely many rational primes ramify. Discarding the principal prime
ideals lying above these finitely many ramified primes, and also discarding
the principal prime ideals generated by rational primes, whose reciprocal norm
contribution is bounded by $\sum_p p^{-2}<\infty$, we obtain a set
${\mathcal P}$ of prime elements of $\Z[\tau_d]$ such that
$|\mathcal N(\alpha)|$ is a rational prime for every
$\alpha\in{\mathcal P}$, no two of these rational primes are equal, and
$$
\sum_{\alpha\in{\mathcal P}}\frac1{|\mathcal N(\alpha)|}=+\infty.
$$
Discarding finitely many further elements, we may assume that
$|\mathcal N(\alpha)|>P_0$ for every $\alpha\in{\mathcal P}$. Moreover,
no $\alpha\in{\mathcal P}$ is associated to $\overline\alpha$ or to a
rational integer. Indeed, the latter is ruled out by construction, and if
$\alpha$ were associated to $\overline\alpha$, then the rational prime
$|\mathcal N(\alpha)|$ would ramify.

Using \cref{L:Cregular}, after replacing each $\alpha\in{\mathcal P}$
by a suitable unit multiple,  we can assume that every $\alpha\in{\mathcal P}$ is $C$-regular, for
the constant $C$ given in \cref{L:Cregular}.

If $\alpha,\beta\in{\mathcal P}$ are distinct, then their norms are distinct
rational primes. Therefore the elements
$\alpha,\overline\alpha,\beta,\overline\beta$ are pairwise non-associate, and
the assumed correlation vanishing applies to every distinct pair
$\alpha,\beta\in{\mathcal P}$. We have shown that ${\mathcal P}$ satisfies
the conditions of \cref{L:KataiWenbo}, and hence the conclusion follows by
applying that lemma.
			
			
		\end{proof}

		\subsection{Proof of Theorem~\ref{T:aperiodicNew}}\label{SS:PfThm31}
		For $d\in \Z$ square-free let $P_d$ be the binary quadratic form defined in \eqref{E:Qd}.\footnote{We work with $P_d(m,n)$ because it is a norm form on the ring of integers of a quadratic field, and all such rings are Dedekind domains, hence every non-zero ideal factors uniquely into prime ideals. In contrast, the binary form $m^2+3n^2$ is a norm form on the ring $\Z[\sqrt{-3}]$, which is not a Dedekind domain.   We alleviate this problem by working with the binary form $m^2+mn+n^2$. This is a norm form on the ring $\Z[(1+\sqrt{-3})/2]$, which is a Dedekind domain. We then use the identity $m^2+3n^2=(m-n)^2+(m-n)(2n)+(2n)^2$,
			to transfer results from one form to the other.} 

		\medskip 
		
		{\bf Case where $P=P_d$, $Q=1,a=b=0$.} First, suppose  that $P=P_d$ and $Q=1,a=b=0$. 
		We will show that  under these additional assumptions, for every $C>0$ we have  
		\begin{equation}\label{E:fLgQd}
			\lim_{N\to\infty} \E_{m,n\in [CN]}\,  {\bf 1}_{K_N}(m,n)\cdot  \prod_{j=1}^sf_j(L_j(m,n))\cdot g(P_d(m,n))=0,
		\end{equation}
		where  for $N\in\N$ we have  $K_N:=K\cap A([N]\times[N])$
		for some convex subset $K$ of $\R_+^2$ and linear transformation $A\colon \Z^2\to \Z^2$ with nonzero determinant. We remark that the sets $A([N]\times[N])$ do not appear  in the statement of \eqref{T:aperiodicNew}, but they are   needed later on in order to reduce the general case to the one covered here.

		We argue as in the proof of \cite[Theorem~9.7]{FH17}, which covered  the case where $d>0$,  with minor changes (the important new ingredient is the variant of the Daboussi-K\'atai criterion we use, which   was essentially proved in \cite{Su23}).  The goal is to use \cref{C:Katai} in order to eliminate the ``quadratic term'' $g(P_d(m,n))$ and be left with products of  ``linear terms.''
		To this end, recall that if $\tau_d$ is as in \eqref{E:taud}, 
		then $P_d(m,n)$ is the norm of the element $m+n\tau_d$. 
		Since $g\in\CM$, the map $\tilde{g}\colon \Z[\tau_d]\to \U$ defined by 
		$$
		\tilde{g}(z):=g(\mathcal{N}(z))
		$$
		is a completely multiplicative function on $\Z[\tau_d]$. 
		Letting $\phi_d\colon \Z[\tau_d]\to \Z$ be  given by 
		$$
		\phi_d(m+n\tau_d):=n,
		$$
		for $j=1,\ldots,s$ we can find $\zeta_j\in\Z[\tau_d]$ such that
		\begin{equation}\label{E:Lj}
			L_j(m,n)=\phi_d(\zeta_j(m+n\tau_d)).
		\end{equation}
		Identifying $(m,n)$ with $z:=m+n\tau_d$ and recalling that $R_N$ is defined by \eqref{E:RN}, we see that \eqref{E:fLgQd} is equivalent to
		\begin{equation}\label{E:fLgQd2}
			\lim_{N\to\infty} \E_{z\in R_{CN}}\, {\bf 1}_{K_N}(z)\cdot  \prod_{j=1}^sf_j(\phi_d(\zeta_jz))\cdot \tilde{g}(z)=0.	    
		\end{equation}

		Let now $\alpha,\beta$ be  prime elements of  $\Z[\tau_d]$ such that
		\begin{equation}\label{E:conditions}
			\alpha,\overline{\alpha},\beta,\overline{\beta} \text{ are non-associates and  do not divide } 
			\zeta_j \text{ for }j=1,\ldots, s. 
		\end{equation}
		From  \cref{C:Katai}, we get that in order to establish \eqref{E:fLgQd2}, it is sufficient to show that for all distinct  $\alpha,\beta$ satisfying \eqref{E:conditions} we have 
		$$
		\lim_{N\to\infty} \E_{z\in R_{CN}}\, {\bf 1}_{K_N}(\alpha z)\cdot 
		{\bf 1}_{K_N}(\beta z)\cdot \prod_{j=1}^sf_j(\phi_d(\zeta_j \alpha z))\cdot  \prod_{j=1}^s\overline{f}_j(\phi_d(\zeta_j\beta  z))=0.
		$$
		Since the function $f_1$ is aperiodic we have by Lemma~\ref{T:aperiodicKnown}  that the needed vanishing property holds, provided that  we show that the linear form 
		$\phi_d(\zeta_1 \alpha z)$ is linearly independent of each of the forms 
		$\phi_d(\zeta_j \alpha z)$ for $j=2,\ldots, s$ and of each of the forms 
		$\phi_d(\zeta_j\beta  z)$ for $j=1,\ldots, s$. To see this, arguing by contradiction, suppose that for some nonzero $k,l\in \Z$ and  $j\in \{1,\ldots, s\}$ and $\alpha'\in \{\alpha,\beta\}$ with $(1,\alpha)\neq (j,\alpha')$, we have 
		$$
		k\phi_d(\zeta_1\alpha z)=l\phi_d(\zeta_j\alpha'z).
		$$
		Using this identity for $z:=1$ and $z:=\tau_d$ we immediately deduce that 
		\begin{equation}\label{E:kl}
			k\zeta_1\alpha =l\zeta_j\alpha'.
		\end{equation}
		We argue  as in the proof of Claim 5 in \cite[Section~9.7.2]{FH17} to get a contradiction.  
		We consider three cases:
		
		(i) Suppose that $j=1$ and $\alpha'=\beta$. In this case, equation~\eqref{E:kl} gives that $\alpha/\beta$ is a rational, which is impossible because by \eqref{E:conditions}  the elements $\alpha$ and  $\beta$ are  non-integer and non-associate  primes in $\Z[\tau_d]$. 
		
		(ii) Suppose that $j>1$ and $\alpha'=\alpha$. In this case we have that $\zeta_1/\zeta_j\in\Q$, and by~\eqref{E:Lj} the linear forms $L_1$ and $L_j$ are linearly dependent, contradicting our  hypothesis.
		
		(iii) It remains to consider the case where  $j>1$ and $\alpha'=\beta$. Let $r_k$ be the exponent of $(\alpha)$ in the factorization of the ideal $(k)$ into  prime ideals of $\Z[\tau_d]$. Since by  \eqref{E:conditions} the element $\alpha$ is a non-integer prime of $\Z[\tau_d]$, $\overline\alpha$ is  also a non-integer prime,  and since $k$ is real,  the exponent of $(\overline\alpha)$ in the factorization of $(k)$ is also equal to $r_k$.  Since by  \eqref{E:conditions} the elements 
		$\alpha$ and $\overline\alpha$ are non-associate primes, it follows that $k$ can be written as $k=\alpha^{r_k}\overline\alpha^{r_k}b$ for some $b\in\Z[\tau_d]$ not divisible by $\alpha$ or $\overline\alpha$.
		In the same way, $l=
		\alpha^{r_l}\overline\alpha^{r_l}c$ for
		some non-negative integer $r_l$ and some  $c\in\Z[\tau_d]$  not divisible by $\alpha$ or $\overline\alpha$.
		Equation~\eqref{E:kl} gives
		$$
		\alpha^{r_k+1}\overline \alpha^{r_k} b\zeta_1= \alpha^{r_l}\overline\alpha^{r_l} c\zeta_j\beta.
		$$
		By  \eqref{E:conditions},  the elements $\alpha$ and $\overline\alpha$ are non-associate primes, $\beta$ is a prime element non-associate to $\alpha$,  and $\zeta_j$ are not divisible by $\alpha$.
		Since $c$ is also not divisible by $\alpha$, it follows that $r_l\geq r_k+1$. Similarly, since $b$ and $\zeta_1$ are not divisible by $\overline\alpha$,  it follows that $r_k\geq r_l$, a contradiction.
		This completes the proof of the claim.

		\medskip
		
		{\bf General Case.} Suppose now that $P(m,n)=\alpha m^2+\beta mn+\gamma n^2$ for some $\alpha,\beta ,\gamma\in \Z$.  We have $\alpha\neq 0$ (since $P$ is irreducible)
		and we can assume that $\alpha\in \N$ (since all multiplicative functions involved are even). 
		We want to show that for every $Q\in \N$ and $a,b\in \Z_+$ we have
		$$
		\lim_{N\to\infty} \E_{m,n\in [N]}\, {\bf 1}_{K}(Qm+a,Qn+b)\cdot \prod_{j=1}^sf_j(L_j(Qm+a,Qn+b))\cdot g(P(Qm+a,Qn+b))=0.
		$$
		
		We first reduce to the case $Q=1$ and $a=b=0$.  \rem{Corrected argument since linear combos of Dirichlet characters cover only progressions with $a,b$ relatively prime to $Q$.} Let
		$$
		d_a:=(a,Q),\qquad q_a:=\frac{Q}{d_a},\qquad a_0:=\frac{a}{d_a},
		$$
		and define $d_b,q_b,b_0$ similarly. Then
		$$
		Qm+a=d_a(q_a m+a_0),
		\qquad
		Qn+b=d_b(q_b n+b_0),
		$$
		where $(a_0,q_a)=(b_0,q_b)=1$. Hence, it suffices to average
		over variables $u,v$ satisfying
		$$
		u\equiv a_0 \pmod {q_a},
		\qquad
		v\equiv b_0 \pmod {q_b}.
		$$
		Expanding these reduced congruence conditions using Dirichlet characters,
		and using linearity, it suffices to show that for every pair of Dirichlet
		characters $\chi$ modulo $q_a$ and $\chi'$ modulo $q_b$ we have
		$$
		\lim_{N\to\infty}
		\E_{u,v\in[N]}\, 
		{\bf 1}_{K'}(u,v)\,
		\chi(u)\chi'(v)\,
		\prod_{j=1}^s f_j\big(L_j'(u,v)\big)\,
		g\big(P'(u,v)\big)
		=0,
		$$
		where
		$$
		L_j'(u,v):=L_j(d_a u,d_b v),
		\qquad
		P'(u,v):=P(d_a u,d_b v),
		$$
		and
		$$
		K':=\{(u,v)\in\mathbb R_+^2\colon (d_a u,d_b v)\in K\}.
		$$
		The set $K'$ is again homogeneous and convex, the form $P'$ is irreducible,
		and the relevant linear-independence assumptions are preserved. Finally,
		the Dirichlet characters  $\chi(u)$ and $\chi'(v)$ are either absorbed into
		one of the existing multiplicative functions, if $u$ or $v$ is proportional
		to one of the linear forms already present, or are treated as additional
		multiplicative functions evaluated at the linear forms $u$ and $v$. Since
		multiplying an aperiodic multiplicative function by a Dirichlet character
		preserves aperiodicity, the case $Q=1$ and $a=b=0$ applies.

		We next reduce to the case $P=P_d$ where $d$ is the unique square-free integer that satisfies 
		$$
		4\alpha \gamma -\beta ^2=r^2d
		$$ 
		for some $r\in \N$, and $P_d$ is as in \eqref{E:Qd}.  
		To get started, note that direct computation gives 
		$$
		4\alpha\, P(m,n)= P_d(A_d(m,n)),
		$$
		where
		\begin{equation}\label{E:Ad}
			A_d(m,n):=\begin{cases} (2\alpha m+\beta n,rn), \quad & \ \text{ if } d\equiv 1,2\! \pmod{4} \\
				(2\alpha m+(\beta -r)n,2rn), & \  \text{ if } d\equiv 3\! \pmod{4}.
			\end{cases}
		\end{equation}

		Using this, the complete multiplicativity  of $f_j,g$ in the form  $f_j(4\alpha r)f_j(k)=f_j(4\alpha rk)$ and  $g(4\alpha )g(k)=g(4\alpha k)$ for every $k\in \N$,  the linearity of the $L_j$, and the homogeneity of the set $K$, 
		we get that it suffices to verify that 
		$$
		\lim_{N\to\infty} \E_{m,n\in [N]}\,  {\bf 1}_{K}(4\alpha rm, 4\alpha rn)\cdot  \prod_{j=1}^sf_j(L_j(4\alpha rm,4\alpha rn))\cdot g(P_d(A_d(m,n)))=0.
		$$
		Equivalently, it suffices to show that 
		$$
		\lim_{N\to\infty} \E_{m,n\in [N]}\, {\bf 1}_{K}(\tilde{L}_d(A_d(m,n)))\cdot  \prod_{j=1}^sf_j(L_{j,d}'(A_d(m,n)))\cdot g(P_d(A_d(m,n)))=0,
		$$
		where 
		$$
		\tilde{L}_d(m,n):=\begin{cases} 2\,(rm-\beta n,2\alpha n), \quad & \ \text{ if } d\equiv 1,2\! \pmod{4} \\
			(2rm-(\beta -r)n,2\alpha n), & \  \text{ if } d\equiv 3\! \pmod{4}
		\end{cases}
		$$ 
		and for $j=1,\ldots, s$, $L'_{j,d}=L_j\circ\tilde{L}_d$.
		Hence, it suffices to show that 
		$$
		\lim_{N\to\infty} \E_{m,n\in [CN]}\, {\bf 1}_{A_{d,N}}(m,n)\cdot 
		{\bf 1}_{K}(\tilde{L}_d(m,n))\cdot
		\prod_{j=1}^sf_j(L_{j,d}'(m,n))\cdot g(P_d(m,n))=0,
		$$
		where $C:=2|\alpha|+|\beta|+2r$,    $A_{d,N}:=A_d([N]\times [N])$, and we used that $A_{d,N}$ is a subset of $[CN]\times [CN]$ and 
		$A_d$ is an injective map.
		Note also that  ${\bf 1}_{K}(\tilde{L}_d(m,n))={\bf 1}_{K_d}(m,n)$ where
		$K_d:=\{(x,y)\in \R^2_+\colon \tilde{L}_d(x,y)\in K\}$ is a convex subset of $\R_+^2$.  
		Since our linear independence assumption for  the linear forms $L_1,\ldots, L_s$ passes to the linear forms $L_{1,d}',\ldots, L_{s,d}'$, the last vanishing property    follows from \eqref{E:fLgQd}.

		\section{Concentration estimates - Statements and preliminary results}\label{S:Concentration1}
		The main goal of this section and the next one  is to prove the concentration estimates for irreducible binary quadratic forms  as stated  in Propositions~ \ref{P:ConcentrationQuanti} and \ref{P:concentration2General}  below. Some  of our intermediate results apply without additional difficulty to arbitrary homogeneous polynomials, and we present them in this broader context.
		\subsection{Preliminaries and notation}
		We will need some notation in order to  state the main concentration estimates.
		\begin{definition}
			If $P\in \Z[m,n]$ is a homogeneous polynomial we let 
			\begin{equation}\label{E:CPP'}
				\CP_P:=\{p\in \P\colon P(m,n)\equiv 0\! \!  \pmod{p}  \text{ has a solution with } n\not\equiv 0 \! \! \pmod{p}  \}
			\end{equation}
			and note that because the polynomial $P$ is homogeneous we have 
			\begin{equation*}\label{E:CPP1}
				\CP_P=\{p\in \P\colon P(m,1)\equiv 0 \! \! \pmod{p} \text{ has a solution} \}.
			\end{equation*}
		\end{definition} 
		We remark that 	if $P$ is an irreducible binary quadratic form with discriminant $D$, then, up to a finite set of primes, $\CP_P$ is the set of primes such that 
		$D$ is a quadratic residue  $\! \! \mod{p}$. Thus, if the square-free parts of the discriminants of two binary quadratic forms $P,Q$  coincide, then $\CP_P=\CP_Q$ up to a finite set of primes.

		\begin{definition}\label{D:omega}
			If $P\in \Z[m,n]$ is a homogeneous polynomial and $r\in \N$ we let 
			\begin{equation}\label{E:omega}
				\omega_P(r):=\big|\{n\in \Z_{r} \colon P(n,1)\equiv 0 \! \!  \pmod{r}\} \big|.
			\end{equation}
			(Note that $\CP_P=\{p\in \P\colon \omega_P(p)>0  \}$.)
		\end{definition}
		\begin{remarks}
			$\bullet$ 	We trivially have  $\omega_P(r)\in \{0,\ldots, r\}$ for every $r\in \N$ and 
			for all but finitely many  $p\in \P$  we have $\omega_P(p)\leq \deg(P)$. 
			If we assume that the greatest common divisor of the coefficient of $P$ is $1$, then  $\omega_P(p)\leq \deg(P)$ for all $p\in \P$. Note, however, that this is not the case  for composite $r$, for instance, 
			if $P(m,n)=m^2$ we have $\omega_P(p^2)=p$ for every $p\in \P$.  
			
			$\bullet$ If $P(m,n)=m^2+dn^2$ and $p\in \P$ is  odd  such that $p\nmid d$, then  $\omega_P(p)=2$ if $-d$ is a quadratic residue $\! \! \mod{p}$ and $\omega(p)=0$ if $-d$ is not a quadratic residue $\! \! \mod{p}$. 
		\end{remarks}
		The following basic property is a consequence of the Chinese remainder theorem 
		\begin{equation}\label{E:omegamulti}
			\omega_P(rs)=\omega_P(r)\, \omega_P(s), \quad \text{whenever } (r,s)=1.
		\end{equation}
		In other words $\omega_P$ is a multiplicative function on $\N$. 
		
		\begin{lemma}\label{L:omegap}
			Let $P$ be an irreducible binary quadratic form with discriminant $D$ and suppose $d$ is the unique square-free integer   satisfying  $D=dr^2$ for some  $r\in \N$. Then for all but finitely many $p\in \P$ we have  $\omega_P(p)=2$ or $0$ depending on whether  $d$ is a quadratic residue $\!\! \mod{p}$ or not.  
		\end{lemma}
		\begin{proof}
			Let $P(m,n)=\alpha m^2+\beta mn+\gamma n^2$ for some $\alpha, \beta,\gamma\in \Z$ not all of them $0$. 
			Since the form is irreducible we have $\alpha\neq 0$ and $D:=\beta^2-4\alpha\gamma$ is not a square. 	
			Note that  
			$4\alpha P(m,1)=(2\alpha m+\beta )^2-dr^2$.
			It follows that if  $p\nmid 2\alpha d r$, then the number of solutions $\! \!\pmod{p}$ to the congruence $P(m,1)\equiv 0\!\!\pmod{p}$ is equal to  the number of solutions to the congruence $m^2\equiv d\!\!\pmod{p}$; hence, it is $2$ or $0$, depending on whether $d$ is a quadratic residue $\!\! \mod{p}$ or not.  
		\end{proof}
		We also introduce some quantities that will appear in our main concentration estimates. If  $P\in \Z[m,n]$ is a homogeneous polynomial, $t\in \R$ and $\chi $ is a Dirichlet character, we let 
		\begin{equation}\label{E:GNP}
			G_{P,N}(f,K):= \sum_{\substack{ K< p\leq N,\\ p\in \P}}\, \frac{\omega_P(p)}{p} \,(f(p)\cdot \overline{\chi(p)}\cdot p^{-it} -1);\footnote{Alternatively, in place of	$G_{P,N}(f,K)$ we could  use the expression, 
				$
				\tilde{G}_{P,N}(f,K):= \sum_{\substack{ K< p\leq N,\\ p\in \P}}\, \frac{\omega_P(p)}{p} \,\Im\big(f(p)\cdot \overline{\chi(p)}\cdot p^{-it} \big)$.
				This is so since $| G_{P,N}(f,K)-	\tilde{G}_{P,N}(f,K)|\leq D_P(f, n^{it}\cdot \chi; K,N)$ and the term  $D_P(f, n^{it}\cdot \chi; K,\sqrt{N})$ appears in all our estimates anyways, so their form will not change in essence.  Similarly for the term  $F_N(f,K)$.}
		\end{equation}
		(we suppress the dependence on $t$ and $\chi$)
		and
		\begin{equation}\label{E:DPxy}
			\D_P(f,g; x,y)^2:= \sum_{\substack{x< p\leq y, \\ p\in \P}} \frac{\omega_P(p)}{p}\, (1-\Re(f(p)\cdot \overline{g(p)})); 
		\end{equation}
		\begin{equation}\label{E:DP}
			\D_P(f,g):=\D_P(f,g; 1,+\infty).
		\end{equation}
		If $P(m,n)=n$, in which case $\omega_P(p)=1$ for every $p\in \P$,  we let 
		\begin{equation}\label{E:DPn}
			\D(f,g; x,y):=\D_P(f,g; x,y), \qquad 	\D(f,g):=	\D_P(f,g). 
		\end{equation}
		Recall that  by Definition~\ref{D:Paperiodic}, if 
		$P\in \Z[m,n]$ is a homogeneous polynomial we say that 
		the multiplicative function $f\colon \N\to \U$ is  $P$-pretentious if $\D_P(f,\chi\cdot n^{it})<+\infty$ for some $t\in \R$ and Dirichlet character $\chi$. Note that 
		since   for all but finitely many $p\in \P$ we have $\omega_P(p)\leq \deg(P)$,
		$f$ is $P$-pretentious if and only if 
		$$
		\sum_{p\in \CP_P} \frac{1}{p}\, (1-\Re (f(p)\cdot \overline{\chi(p)}\cdot p^{-it}))<+\infty,
		$$
		where $\CP_P$ is the set defined in \eqref{E:CPP'}.
		In this case we write $f\sim_P \chi\cdot n^{it}$. If $P(m,n)=n$, then $\omega_P(p)=1$ for all $p\in \P$ and we recover the usual notion of pretentiousness in the sense of Granville and Soundararajan~\cite{GS23}.	 
		More generally we have the following definition. 
		\begin{definition}
			For every bounded sequence $c\colon \P\to \R_+$ we  let
			$$
			\D_c(f,g; x,y)^2:= \sum_{\substack{x< p\leq y, \\ p\in \P}} \frac{c(p)}{p}\, (1-\Re(f(p)\cdot \overline{g(p)})).
			$$
			If $\D_c(f,g; 1,+\infty)<+\infty$ we write $f\sim_c g$. 
			Note that if $c_1(p)\leq c_2(p)$ for all but finitely many $p$, then the following implication holds $f\sim_{c_2} g\Rightarrow f\sim_{c_1}g$.
		\end{definition}
		It can be shown (as in \cite{GS08} or \cite[Section~2.1.1]{GS23}) that $\D_P$ satisfies the triangle inequality
		\begin{equation}\label{E:triangle}
			\D_P(f, g) \leq \D_P(f, h) + \D_P(h, g)
		\end{equation}
		for all  $f,g,h\colon \P\to \U$.
		Also, for all  $f_1, f_2, g_1, g_2\colon \P\to \U$,  we have (proved as in
		\cite[Lemma~3.1]{GS07})
		\begin{equation}\label{E:Df1f2}
			\D_P(f_1f_2, g_1g_2) \leq \D_P(f_1, g_1) + \D_P(f_2, g_2).
		\end{equation}
		It follows from \eqref{E:Df1f2} that if $f_1\sim_P g_1$ and $f_2\sim_P g_2$, then $f_1f_2\sim_P g_1g_2$. Similar estimates also hold with $\D_c$ in place of $\D_P$ where $c\colon \P\to \R_+$ is any bounded sequence.
		
		If $f\sim \chi \cdot  n^{it}$, it is known that the value of $t$ is unique.  A similar fact holds if $f\sim_P \chi \cdot  n^{it}$ and $P$ is an irreducible binary quadratic form. We will prove a more general fact in the next lemma. 
		\begin{lemma}\label{L:tunique}
			Let $P_1,P_2\in \Z[m,n]$ be  two irreducible binary quadratic forms, let  $f\colon \N\to \U$ be a multiplicative function,  
			and suppose that for $j=1,2$ we have $f\sim_{P_j} \chi_j\cdot n^{it_j}$ for some $t_j\in \R$ and Dirichlet characters $\chi_j$.   Then 
			$t_1=t_2$ and  $\chi_1(p)=\chi_2(p)$ for all but finitely many  $p\in \CP_{P_1}\cap \CP_{P_2}$.
		\end{lemma}
		\begin{proof}
			First, we  show that $t_1=t_2$. 
			Using \cref{L:omegap} and quadratic reciprocity, we can find $a\in \N$ such that $\omega_{P_1}(p)=\omega_{P_2}(p)=2$ for all but finitely many primes $p\in \mathcal{G}:=1+a\N$. 
			Indeed, if  $d_1,d_2$ are the square-free parts of the discriminants of $P_1,P_2$, and $q_1,\ldots, q_k$ are the prime factors of $d_1, d_2$, for any prime $p$ such that $q_1,\ldots, q_k$ are all quadratic residues $\pmod{p}$, both $d_1$ and $d_2$ are quadratic residues $\pmod{p}$ and hence by \cref{L:omegap} $\omega_{P_1}(p)=\omega_{P_2}(p)=2$.  
			Using quadratic reciprocity we get that any prime $p$ satisfying $p\equiv 1 \pmod{8\, q_1\cdots q_k}$ works, so it suffices to take $a:=8\, q_1\cdots q_k$.  
			
			So let $c(p):=2$ if $p\in \mathcal{G}$ and $c(p):=0$ otherwise, so that $c(p)\leq \omega_{P_j}(p)$ for $j=1,2$ and all but finitely many $p$.
			Therefore $f\sim_c \chi_1\cdot n^{it_1}$ and $f\sim_c \chi_2\cdot n^{it_2}$.  
			It follows
			from \eqref{E:Df1f2}, with $\D_c$ in place of $\D_P$,  that $\chi_1\cdot n^{it_1}\sim_c \chi_2\cdot n^{it_2}$. 
			Note also that 
			there exists $k\in \N$ such that $\chi_j^k(p)=1$, $j=1,2$, for all but finitely many primes $p\in \P$. Using the previous  facts and   \eqref{E:Df1f2} again, we deduce that  $n^{ikt_1}\sim_c n^{ikt_2}$, hence 
			$n^{ik(t_1-t_2)}\sim_c 1$. 
			Whenever $\sum_{p\in \mathcal{G}}\frac{1-\Re(p^{it})}{p}<\infty$, the prime number theorem on arithmetic progressions implies $t=0.$ Consequently, we have $t_1=t_2$. 
			
			Next note that if $P$ is an irreducible quadratic form with discriminant $D$, then 
			$\CP_P$  consists of those primes $p$ that satisfy $\legendre{D}{p}=1$. 
			Using  quadratic reciprocity we deduce that there exists $Q\in \N$ and a non-empty finite set $Q_0$, consisting of positive integers relatively prime to $Q$,   such that  
			\begin{equation}\label{E:P12cap}
				\CP_{P_1}\cap \CP_{P_2}=\{p\equiv  j \! \!  \pmod{Q}\colon j\in Q_0\}.
			\end{equation}
			By enlarging $Q$ if necessary, we can assume that both $\chi_1,\chi_2$ have period  $Q$.  
			Let $\omega(p)=2$ for $p\in \CP_{P_1}\cap \CP_{P_2}$ and $0$ otherwise. 
			Since $\omega_{P_1}(p)=\omega_{P_2}(p)=2$ for all but finitely many $p\in \CP_{P_1}\cap \CP_{P_2}$,  we have $f\sim_{\omega}\chi_j\cdot n^{it}$,  for $j=1,2$,  and we deduce using \eqref{E:triangle}, with $\D_c$ in place of $\D_P$, that 
			\begin{equation}\label{E:P12cap'}
				\sum_{ p\in \CP_{P_1}\cap \CP_{P_2}}\, \frac{1}{p} \,(1-\Re( \chi_1(p)\cdot \overline{\chi_2(p)} ))<\infty. 
			\end{equation}
			Since $\chi_1, \chi_2$  have period $Q$ and take finitely many values on primes, if 
			$\chi_1(p_0)\neq\chi_2(p_0)$ for some 	$p_0\in \CP_{P_1}\cap \CP_{P_2}$ which does not divide $Q$, we deduce from \eqref{E:P12cap}  and \eqref{E:P12cap'}  that  
			$$
			\sum_{ p\equiv p_0\! \! \! \pmod{Q}}\, \frac{1}{p} <\infty. 
			$$
			This contradicts Dirichlet's theorem on arithmetic progressions and completes the proof.
		\end{proof}
		\subsection{The case $P(m,n)=m^2+dn^2$}
		For $d\in \Z$ square-free  we let 
		$$
		P_d(m,n)=m^2+dn^2 \quad  \text{ and } \quad  \CP_{d}:=\CP_{P_d},
		$$
		where $\CP_{P_d}$ is given by  \eqref{E:CPP'}.
		Then
		\begin{equation}\label{E:Pd}
			\CP_{d}=\Big\{p\in \P\colon  \legendre{-d}{p}=1 \Big\}.
		\end{equation}
		For example, 
		$\CP_1:=\{ p\equiv 1\!  \pmod{4}\}$,
		$\CP_2:=\{p\equiv 1, 3\!\pmod{8}\}$,
		$\CP_{-2}:=\{ p\equiv 1, 7\!\!\pmod{8}\}$. 
		It can be shown that  for square-free $d$ the set $\CP_{d}$  has relative density $1/2$ on the primes and it always consists of the set of  primes that belong to  finitely many congruence classes. For example, using quadratic reciprocity, we get that  if $q$ is an odd prime, then 
		$$
		\CP_q=\{ p\equiv \pm k^2\! \! \pmod{4q}\colon  k\in [4q] 
		\text{ is odd  and coprime to } 4q \}.
		$$

		In the case where $P=P_d$ the  	
		expressions  \eqref{E:GNP} and \eqref{E:DPxy}  used in the statements of \cref{P:ConcentrationQuanti} and \cref{P:concentration2General} take  the form 
		\begin{equation}\label{E:GNd}
			G_{d,N}(f,K):= 2\, \sum_{\substack{ K< p\leq N,\\ p\in \CP_d}}\, \frac{1}{p} \,(f(p)\cdot \overline{\chi(p)}\cdot p^{-it} -1)
		\end{equation}
		and
		\begin{equation}\label{E:Ddxy}
			\D_d(f,g; x,y)^2:= 2\, \sum_{\substack{x< p\leq y, \\ p\in \CP_d}} \frac{1}{p}\, (1-\Re(f(p)\cdot \overline{g(p)})).
		\end{equation}
		\subsection{Statements of linear and quadratic concentration estimates}
		We will make use of the following ``linear'' concentration estimate, which is an immediate consequence of \cite[Lemma 2.5]{KMPT21}.
		\begin{definition}	
			Throughout, 	if $f\colon \N\to \U$ is a multiplicative function we extend it to $\Z$ by letting $f(-n)=f(n)$ for $n\in \N$. We do a similar thing for additive functions $h\colon \N\to \R$.
		\end{definition}
		For each $K\in\N$ let
		\begin{equation}\label{E:PhiK}
			\Phi_K	:=\Big\{\prod_{p\leq K}p^{a_p}\colon K< a_p\leq 2K\Big\}.
		\end{equation}
		\begin{proposition}\label{P:concentration1}
			Let 
			$f\colon \mathbb{N}\to\mathbb{U}$ be a multiplicative function such that $f\sim \chi \cdot  n^{it}$ for some $t\in \R$ and Dirichlet character $\chi$ with  period $q$. Let $K\in\N$ and $\Phi_K$ be the set described in \eqref{E:PhiK}. 
			Suppose that $K$ is sufficiently large so that $q$ divides all elements of $\Phi_K$. 
			Then  
			\begin{multline*}
				\limsup_{N\to\infty} \max_{Q\in \Phi_K,a\in A_Q} \E_{n\in [N]}|f(Qn+a)- \chi(a)\cdot (Qn)^{it}\cdot \exp\big(F_N(f,K)\big)| \\ \ll\D(f,\chi\cdot n^{it}; K,\infty)+K^{-1/2},
			\end{multline*}
			where for given $Q\in \N$ the set $A_Q$ consists of all  $a\in \Z$ with  $|a|\leq Q$ and $(a,Q)=1$, the implicit constant is absolute, $\D$ is as in \eqref{E:DPn}, and
			\begin{equation}\label{E:FNfQdef}
				F_N(f,K):=\sum_{K< p\leq N} \frac{1}{p}\,\big(f(p)\cdot \overline{\chi(p)}\cdot p^{-it} -1\big).
			\end{equation}
			(We suppress the dependence on $t$ and $\chi$).
		\end{proposition}
		Our main goal in this section is to prove an analog of this result for irreducible binary quadratic forms.  The case where $Q(m,n)=m^2+n^2$ was essentially covered in \cite{FKM23}. We follow a similar general strategy, but the more general setting we want to cover presents some additional challenges that we need to address. In particular, unlike in \cite{FKM23}, we need to invoke several results from algebraic number theory and also work over real quadratic extensions. 
		\begin{proposition}\label{P:ConcentrationQuanti}
			Let $P\in \Z[m,n]$ be an irreducible binary quadratic form, 
			$A,K,N\in \N$, and  $f\colon \Z\to\U$ be an even multiplicative function,  $t\in \R$,  $\chi$ be a Dirichlet character with period $q,$   $Q= \prod_{p\leq K}p^{a_p}$ for some  $a_p\in \N$.   
			Let also  $a,b\in \Z$ be  such that  $-A\leq a,b\leq A$, 	$c:=(P(a, b),Q)$, and  
			suppose that $q$ and $\prod_{p\leq K}p$ divide $Q/c$.	
			Then there exist constants $C_{A,P,Q}>0$ such that if  $K$ is sufficiently large, depending only on $P$, and   $N$ is  sufficiently large, depending only on  $A$, $P$, $Q$, we have
			\begin{multline}\label{E:generalfP}
				\E_{m,n\in [N]}\, \big|f\big(P_c(Qm+a,Qn+b)\big)- \chi(P_c(a,b))\cdot   
				(P_c(Qm+a,Qn+b))^{it}\cdot   \exp\big( G_{P,N}(f,K)\big)\big| \\ 	\ll_P (\D_P+\D_P^2)(f,\chi\cdot n^{it}; K,\sqrt{N})+
				C_{A,P,Q}\cdot\mathbb{D}_P(f,\chi\cdot n^{it};\sqrt{N},C_{A,P,Q} N^2)+K^{-1/2},
			\end{multline}
			where $P_c:=P/c$ and $G_{P,N},\D_P$ are as in   \eqref{E:GNP}, \eqref{E:DPxy}.
		\end{proposition}
		\begin{remarks}
			$\bullet$
			The uniformity in the parameters is important for our applications. In particular, it is crucial that the  implicit constant depends only on $P$ and not on  $A,K$, and that  $\exp\big(G_{P,N}(f,K)\big)$ is the same for all $Q$ that are divisible by $\prod_{p\leq K}p$. 
			
			$\bullet$ In the proof of part~\eqref{I:a} of \cref{T:DensityRegularQuadratic}  we apply this result  and the next one for $a=1,b=0$ in which case some of our assumptions hold trivially. The full strength of \cref{P:ConcentrationQuanti}, for varying $a,b$ that depend on $Q$,  is needed in the proof of part~\eqref{I:b} of \cref{T:DensityRegularQuadratic}
			to cover the case where the irreducible polynomials  $P_1,P_2$ have different discriminants.
			
			$\bullet$		If $P$ is reducible, then $P=L_1\cdot L_2$, where $L_1, L_2$ are linear forms, and we  get concentration estimates  using \cref{P:concentration1}.

			$\bullet$	The averaging over both variables $m,n\in\mathbb{N}$ is crucial for our argument and allows to overcome issues with large primes that cause concentration estimates along single-variable quadratic polynomials to fail in general. 
			On the other hand, in \cite{Te24} Ter\"{a}v\"{a}inen proved a version of the concentration estimates for values of $f(P(Qn+a))$, where  $P\in\mathbb{Z}[x]$ is  arbitrary and the  multiplicative functions $f$ satisfies  $f(p)=p^{it}\chi(p)$ for $p>N$.  In our setting, however, we cannot afford to make such assumptions on $f$.
			
		\end{remarks} 

		For $P$-pretentious multiplicative functions we deduce the following result (we also use that $\lim_{m,n\to\infty}\big((P_c(Qm+a,Qn+b))^{it}-(P_c(Qm,Qn))^{it}\big)=0$).
		\begin{proposition}\label{P:concentration2General}
			Let $P\in \Z[m,n]$ be an irreducible binary quadratic form and 	$f\colon \Z\to\U$ be an even multiplicative function such that $f\sim_P \chi\cdot n^{it}$ for some $t\in \R$ and Dirichlet character $\chi$ with period $q$. Let also    $\Phi_K$ be as in \eqref{E:PhiK},  $A_K\in \N$,  and suppose that
			$K$ is  sufficiently large,  depending only on $P$, and  so that $\D_P(f,\chi\cdot n^{it}; K,\infty)\leq 1$ and $q$ divides all elements of $\Phi_K$. Then
			\begin{multline*}
				\limsup_{N\to\infty}  \max_{Q\in \Phi'_K, a,b\in S_K }	\E_{m,n\in [N]}\, \big|f\big(P_c(Qm+a,Qn+b)\big)- \\ 
				\chi(P_c(a,b))\cdot   
				(P_c(Qm,Qn))^{it}\cdot   \exp\big( G_{P,N}(f,K)\big)\big|\ll_P  
				\D_P(f,\chi\cdot n^{it}; K,\infty)+K^{-1/2},
			\end{multline*}
			where  $c:=(P(a, b),Q)$, 		
			$P_c:=P/c$, $\Phi'_K$ is a subset of $\Phi_K$, $S_K$ consists of all  $a,b\in [-A_K,A_K]$ such that   $q$ and $\prod_{p\leq K}p$ divide $Q/c$ for all $Q\in \Phi'_K$,   and  $G_{P,N},\D_P$ are as in   \eqref{E:GNP}, \eqref{E:DPxy}.
		\end{proposition}
		\begin{remark}
			It is important for our argument that $\exp\big(G_{P,N}(f,K)\big)$ is the same for all $Q\in \Phi_K$ that are divisible by $q$. It is also important for our applications that we get some uniformity over the  $Q\in \Phi_K$. Also note that if $(P(a,b),Q)=1$, then $P_c=P$ and $S_K=[-A_K,A_K]$. 
		\end{remark}
		In the remaining part of this section  we give some preliminary results needed to prove \cref{P:ConcentrationQuanti} and then complete its proof in \cref{S:Concentration2}.

		\subsection{Preliminary results I}  
		We will need the following elementary result. 
		\begin{lemma}\label{L:omegap2bound}
			Let $P\in \Z[m,n]$ be an irreducible  homogeneous polynomial    
			and $\omega_P$ be  as in \cref{D:omega}. 
			Then 
			$\omega_P(p^2)\leq \deg(P)$ for all but finitely many $p\in \P$.
		\end{lemma}
		\begin{proof}  
			Let $\tilde{P}(n):=P(n,1)$. 
			By Hensel's lemma, if $\tilde{P}(n)\equiv 0 \!\pmod{p^2}$ has more than $\deg(\tilde{P})$ solutions, then either $\tilde{P}(n)\equiv 0 \!\pmod{p}$ has more than $\deg(\tilde{P})$ solutions (which can only happen for finitely many $p$), or the polynomials $\tilde{P}$ and $\tilde{P}'$ have a common root $\!\!\!\pmod{p}$.
			Since $\tilde P$ is irreducible in $\Z[n]$, it shares no common factors with $\tilde P'$ and hence there are polynomials $P_1,P_2\in \Z[n]$ such that $P_1\tilde{P}+P_2\tilde{P}'=k$ for some non-zero $k\in \Z$.
			Therefore, if $\tilde{P}$ and $\tilde{P}'$ share a root $\!\!\!\pmod{p}$, $k$ must be a multiple of $p$, which can  only happen for finitely many primes $p$.
		\end{proof}  
		\begin{lemma}\label{L:wNPQ}
			Let $P\in \Z[m,n]$ be a  homogeneous polynomial that is nonlinear and irreducible, and let $\omega_P$ be as in \cref{D:omega}. 
			For $Q,N\in \N$, $a,b\in \Z$, 
			and   $p,q\in \P$, let
			\begin{equation}\label{E:wNQpq}
				w_{P,Q,N}(p,q):=\frac{1}{N^2}\, \sum_{\substack{m,n\in [N],\\   p,q\, \mid\mid\,  P(Qm+a,Qn+b)} } 1.
			\end{equation}
			(Note that $w$ also depends on $a$ and $b$. To ease the notation, we suppress the dependence.)
			Then there exists a finite subset  $F$ of the primes,  consisting of the prime divisors of $Q$  together with a finite set $F'$ depending only on $P$,   such that for $p\notin F$ 
			we have 
			\begin{equation}\label{E:wNPQ1}
				w_{P,Q,N}(p,p)=\Big(\frac{\omega_P(p)}{p}-\frac{\omega_P(p^2)}{p^2}\Big)\, \Big(1-\frac{1}{p}\Big) + O_P\Big(\frac{1}{N}\Big),
			\end{equation}
			and  if $p,q\not\in F$ and  $p\neq q$ we have  
			\begin{equation}\label{E:wNPQ2}
				w_{P,Q,N}(p,q)=\Big(\frac{\omega_P(p)}{p}-\frac{\omega_P(p^2)}{p^2}\Big)\, \Big(\frac{\omega_P(q)}{q}-\frac{\omega_P(q^2)}{q^2}\Big) \, \Big(1-\frac{1}{p}\Big)\,  \Big(1-\frac{1}{q}\Big) +O_P\Big(\frac{1}{N}\Big).       
			\end{equation}
			Furthermore, if $p,q\not\in F$, then for every   $c\in \N$  such that $c\mid (Q,P(a,b))$, we have~\footnote{Note  that if  
				$c\in \N$  is such that $c\mid (Q,P(a,b))$, then $c\mid P(Qm+a,Qn+b)$ for every $m,n\in \N$.} 
			\begin{equation}\label{E:c=1}
				w_{P,Q,N}(p,q)=\frac{1}{N^2}\, \sum_{\substack{m,n\in [N],\\   p,q\, \mid\mid\,  P_c(Qm+a,Qn+b)} } 1
			\end{equation}
			where 	$P_c:=P/c$.
		\end{lemma}
		\begin{remark}
			Note that $w_{P,Q,N}(p,q)=0$ unless  $p,q\in \CP_P$.

		\end{remark}
		\begin{proof}
			Throughout the proof, we will use $\epsilon$  to denote  a number in $\{0,1,\ldots, 2\deg(P)\}$.  Let $F'$ be the set that consists of the finite set of primes given by Lemma~\ref{L:omegap2bound} and  the prime divisors of $ P(1,0)$ (which is nonzero since $P$ is irreducible), and let 
			$F$ be the finite set that consists of  $F'$ and  the prime divisors of  $Q$.  Henceforth, we assume that $p\not\in F$ is a prime.
			
			We first establish \eqref{E:c=1}. Since $p\notin F$ we have $p\nmid Q$, hence  $p\nmid c$. This implies that $p\mid\mid P_c(Qm+a,Qn+b)$ if and only if 
			$p\mid\mid P(Qm+a,Qn+b)$, and as a consequence   for $p,q\notin F$,  on the right  side of  \eqref{E:c=1}  we can replace $P_c$ by $P$, proving \eqref{E:c=1}.
			
			Next we establish  \eqref{E:wNPQ1}. We will use  that $\omega_P(p)\leq \deg(P)$ and $\omega(p^2)\leq \deg(P)$ throughout.  
			First, note  that if  $p\mid Qn+b$ and $p\mid P(Qm+a,Qn+b)$, then also $p\mid Qm+a$ (we used that $P$ is homogeneous and $p\nmid P(1,0)$ since $p\not\in F$) and consequently $p^2\mid P(Qm+a,Qn+b)$  (we used that $P$ is homogeneous and non-linear), hence  we get no contribution to the sum \eqref{E:wNQpq} in this case. So we can assume that $p,n$ are such that   $p\nmid Qn+b$. Since $(p,Q)=1$,  and $n$ is such that
			$p\nmid Qn+b$, and $P$ is homogeneous (hence, $P(m,n)=0\pmod{p}$ if and only if  $P(mn^*,1)=0\pmod{p}$ where $nn^*\equiv 1\pmod{p}$),   we get by \eqref{E:omega} that for this fixed value of $n\in [N]$ 
			we have exactly $\omega_P(p)$ solutions $m\pmod{p}$ of the congruence
			\begin{equation}\label{E:congruence}
				P(Qm+a,Qn+b)\equiv 0\pmod {p}.
			\end{equation}
			Hence, for those $n\in [N]$ we have   $	\omega_P(p)[N/p]+\epsilon$ solutions in the variable $m\in [N]$ to \eqref{E:congruence}.
			Since there are $N-[N/p]+\epsilon$ integers $n\in [N]$ with $p\nmid Qn+b$ (we used that $(p,Q)=1$ here), we get  a  total of
			$$
			\omega_P(p)[N/p]\, (N-[N/p]) +O(N)=\omega_P(p)(N^2/p-N^2/p^2)+O_P(N)
			$$ solutions of $m,n\in [N]$ to the congruence \eqref{E:congruence}. Similarly,    we get that if
			$p\nmid Qn+b$, then the number of solutions $m,n\in [N]$ to the  congruence  
			$$
			P(Qm+a,Qn+b)\equiv 0\pmod {p^2}
			$$ 
			is
			$$
			\omega_P(p^2)[N/p^2]\, (N-[N/p])+O(N)=\omega_P(p^2)(N^2/p^2-N^2/p^3)+O_P(N).
			$$
			These solutions should be subtracted from the previous solutions of \eqref{E:congruence} in order to count the number of solutions of $m,n\in [N]$ for which $p\, \mid\mid \, 	P(Qm+a,Qn+b)$ holds. We deduce that
			\begin{equation}\label{E:pndivn}
				\frac{1}{N^2}\, \sum_{\substack{m,n\in [N],\\   p\, \mid\mid\,  P(Qm+a,Qn+b)} } 1=\Big(\frac{\omega_P(p)}{p}-\frac{\omega_P(p^2)}{p^2}\Big)\Big(1-\frac{1}{p}\Big) +O_P\Big(\frac{1}{N}\Big),
			\end{equation}
			which proves  \eqref{E:wNPQ1}. 
			
			Next,  we establish \eqref{E:wNPQ2}. 	Let $p,q $ satisfy the assumptions.
			As explained in the previous case, those $n\in [N]$ for which $p\mid Qn+b$ or $q\mid Qn+b$ do not contribute  to the sum \eqref{E:wNQpq} defining $w_{Q,N}(p,q)$, so we can assume that
			$(pq,Qn+b)=1$.
			Let
			$$
			A_{r,s}:=\frac{1}{N^2}\, \sum_{\substack{m,n\in [N],\\   r,s\mid  P(Qm+a,Qn+b),\, (rs,Qn+b)=1 }} 1
			$$
			and note that
			\begin{equation}\label{E:wNQA}
				w_{P,Q,N}(p,q)=A_{p,q}-A_{p^2,q}-A_{p,q^2}+A_{p^2,q^2}.
			\end{equation}
			
			First we compute $A_{p,q}$.
			Since $p,q\in \CP_P$ we get using  \eqref{E:omegamulti} 
			that for each $n\in [N]$ with $p,q\nmid Qn+b$ and $(pq,Q)=1$  we have $\omega_P(p)\, \omega_P(q)$ solutions $m\pmod {pq}$ to the congruence
			\begin{equation}\label{E:congruencepq}
				P(Qm+a, Qn+b)\equiv 0\pmod {pq}.
			\end{equation}
			We deduce that  for  every $n\in [N]$ with $(pq,Qn+b)=1$  we have $\omega_P(p)\, \omega_P(q)\, [N/(pq)]+\epsilon$ solutions in the variable $m  \in [N]$ to the congruence \eqref{E:congruencepq}.
			Since  the number of $n\in [N]$ for which  $(pq,Qn+b)=1$ is $N-[N/p]-[N/q]+[N/pq]$,
			we get that the total number of solutions of the congruence \eqref{E:congruencepq} with $m,n\in[N]$ and $(pq,n)=1$ is
			\begin{multline}\label{E:pqndivn}
				\omega_P(p)\, \omega_P(q)  [N/(pq)] \, (N-[N/p]-[N/q]+[N/(pq)])+ O_P(N)=\\ N^2 \cdot (\omega_P(p)\, \omega_P(q)/(pq))\cdot (1-1/p-1/q+1/(pq)) +O_P(N).
			\end{multline}
			Hence,
			$$
			A_{p,q}:=
			\frac{\omega_P(p)\, \omega_P(q)}{pq}\Big(1-\frac{1}{p}\Big) \Big(1-\frac{1}{q}\Big) +O_P\Big(\frac{1}{N}\Big).
			$$
			Similarly, using again \eqref{E:omegamulti}, 
			we find that
			$$
			A_{p^2,q}=
			\frac{\omega_P(p^2)\, \omega_P(q)}{p^2q}\Big(1-\frac{1}{p}\Big) \Big(1-\frac{1}{q}\Big) +O_P\Big(\frac{1}{N}\Big),
			$$
			and
			$$
			A_{p,q^2}=
			\frac{\omega_P(p)\, \omega_P(q^2)}{pq^2}\Big(1-\frac{1}{p}\Big) \Big(1-\frac{1}{q}\Big) +O_P\Big(\frac{1}{N}\Big).
			$$
			Also,
			$$
			A_{p^2,q^2}=
			\frac{\omega_P(p^2)\, \omega_P(q^2)}{p^2q^2}\Big(1-\frac{1}{p}\Big) \Big(1-\frac{1}{q}\Big) +O_P\Big(\frac{1}{N}\Big).
			$$
			Using the last four identities and  \eqref{E:wNQA}, we deduce that \eqref{E:wNPQ2} holds. This completes the proof.
		\end{proof}
		If we combine \cref{L:wNPQ} with the estimate of Lemma~\ref{L:omegap2bound}
		we immediately deduce the following result.
		\begin{corollary}\label{C:wNPQ}
			Let $P\in \Z[m,n]$ be a  homogeneous polynomial that is nonlinear and irreducible  and   let $\omega_P$, $w_{P,Q,N}$ be  as in \cref{D:omega} and \eqref{E:wNQpq} respectively. 
			Then there exists a finite subset  $F$ of the primes,  consisting of the prime divisors of $Q$  together with a finite set $F'$ depending only on $P$,   such that for $p\notin F$ we have 
			\begin{equation}\label{E:wpbound}
				\Big|w_{P,Q,N}(p,p)-\frac{\omega_P(p)}{p}\Big|\ll_P \frac{1}{p^2}+\frac{1}{N}, \quad w_{P,Q,N}(p,p)\leq \frac{\omega_P(p)}{p}+ O_P\Big(\frac{1}{N}\Big),
			\end{equation}
			and  if  $p,q\not\in F$ and  $p\neq q$ we have
			\begin{equation}\label{E:wpqindep}
				|w_{P,Q,N}(p,q)- w_{P,Q,N}(p,p)\cdot w_{P,Q,N}(q,q)|\ll_P \frac{1}{N}.
			\end{equation}
		\end{corollary}

		\subsection{Preliminary  results II}
		Let $P\in \Z[m,n]$ be homogeneous.    	For $l, Q,N \in \N$ and  $a,b\in \Z$ with  $-Q\leq a,b\leq Q$,   let
		\begin{equation}\label{E:wPQNl}
			w^*_{P,Q,N}(l):=\frac{1}{N^2} \, \sum_{\substack{m,n\in [N],\\   l\mid  P(Qm+a,Qn+b)} } 1.
		\end{equation}
		We use $w^*$ instead of  $w$ because $w_{P,Q,N}(p)$ was used in \eqref{E:wNQpq} to denote a similar expression with $\mid\mid$ instead of  $\mid$. Note  that $w^*$ also depends  on $a,b$, but to ease the notation a bit we  will suppress  this dependence.   In this subsection we restrict  ourselves to the case where $P$ is an irreducible binary quadratic form,  and our  goal is to get
		in \cref{L:wNQl} useful bounds
		for $w^*_{P,Q,N}(l)$ when $l$ is composite. We note that 
		these bounds are better than those in \eqref{E:wpbound} and \eqref{E:wpqindep} when $l=p>N$ and $l=pq>N$ respectively, which will be used in the sequel.

		\subsubsection{Algebraic number theory preliminaries}
		Let $P=P_d$  be as in \eqref{E:Qd} for some square-free $d\in \Z$, namely, 
		\begin{equation}\label{E:Qd'}
			P_d(m,n):=\begin{cases} m^2+dn^2, \quad & \, \text{ if } d\equiv 1,2\! \pmod{4} \\
				m^2+mn+\frac{d+1}{4}n^2, & \,  \text{ if } d\equiv 3\! \pmod{4}.
			\end{cases}
		\end{equation}
		For square-free $d\in \Z$, 	let $\Z[\tau_d]$ be the ring of integers of the number field $\Q(\sqrt{-d})$
		where $\tau_d$ is as in \eqref{E:taud}, namely, 
		\begin{equation}\label{E:taud'}
			\tau_d:=\begin{cases} \sqrt{-d}, \quad & \, \text{ if } d\equiv 1,2\! \pmod{4} \\
				\frac{1+\sqrt{-d}}{2}, & \,  \text{ if } d\equiv 3\! \pmod{4},
			\end{cases}
		\end{equation} 
		and recall that
		for $m,n\in \Z$ the norm of the element $m+n\tau_d$ is  $\CN(m+n\tau_d)=P_d(m,n)$.
		We want to estimate 
		\begin{equation}\label{E:CdNk}
			C_{d,N}(k):=|\{m,n\in [-N,N]\colon \CN(m+n\tau_d)=k\}|,
		\end{equation}
		and it turns out to  be more convenient to work with ideals  in $\Z[\tau_d]$ and estimate the corresponding quantities.
		Recall that the norm of an ideal is a nonnegative number, and if 
		the ideal  is principal and   generated by $a\in \Z[\tau_d]$, then its norm is equal to $|\CN(a)|$. For $k\in \N$ we let 
		\begin{equation}\label{E:CCdk}
			\CC_d(k):=|\{\text{ideals of } \Z[\tau_d] \text{ with norm } k\}|.
		\end{equation}
		
		\begin{lemma}\label{L:IdealCount}
			Let $d\in \Z$ be a square free integer and $\CC_d(k)$ be as in \eqref{E:CCdk}. Then 
			\begin{equation}\label{E:idealcount}
				\sum_{n\in [N]} \CC_d(n)\ll_d N, \footnote{In fact we have 
					$
					\sum_{n\in [N]}\CC_d(n)=\rho_d N+O_d(N^{1/2})$ for some $\rho_d>0$, but this more precise result is not needed.}  
			\end{equation}
			and   if 
			$l\in \N$ is a product of at most $s$ not necessarily distinct  primes in $\CP_d$ we have 
			\begin{equation}\label{E:CCdd}
				\CC_d(ln)\leq 2^s\, \CC_d(n) \quad \text{ for every } n\in \N.	
			\end{equation}
		\end{lemma}
		\begin{proof}
			The first estimate follows from the ideal counting theorem (see e.g. \cite[Example 10.2.9]{ME05}).
			
			We prove \eqref{E:CCdd}.	It is known that $\CC_d(k)$ is multiplicative,
			$\CC_d(p^{a+1})\leq \CC_d(p^a)+1$    for every $a\in \Z_+$, and $\CC_d(p^a)\geq 1$ for every $p\in \CP_d$ and $a\in \Z_+$ (see for example Exercises 10.2.1 and 10.2.4 in \cite{ME05}).
			It follows easily from these two properties that for any prime $p$  and $n\in\N$ (not necessarily relatively prime to $p$) we have 
			$$
			\CC_d(pn)\leq 2\, \CC_d(n).	
			$$	
			So if $l$ is a product of at most $s$ not necessarily distinct  primes,  we get that \eqref{E:CCdd} holds.  
		\end{proof}

		The next lemma allows us to pass estimates from $\CC_d(k)$ to $C_{d,N}(k)$. 
		\begin{lemma}\label{L:CC}
			Let  $d$ be a square-free integer, $N\in \N$, and $	C_{d,N}$ and $\CC_d$ as in \eqref{E:CdNk} and \eqref{E:CCdk} respectively. 
			\begin{enumerate}
				\item \label{I:CC1} If $d>0$, then 	$C_{d,N}(k)\leq  6 \cdot  \CC_d(|k|)$, for every $k\in \Z$, $N\in \N$.
				
				\item \label{I:CC2} If $d<0$, then 	
				there exists $c_d>0$ 
				such that  $C_{d,N}(k)\leq   \log_+ (c_d\, N/\sqrt{|k|}) \cdot  \CC_d(|k|)$, for every nonzero $k\in \Z$, 
				and $N\in \N$, 	where  $\log_+(x)={\bf 1}_{[1,+\infty)}(x)\cdot \log x$.\footnote{We remark that 	the weaker  estimate   $C_{d,N}(k)\leq   \log_+ (c_d\, N) \cdot  \CC_d(|k|)$ is not sufficient for our purposes.} 
				
			\end{enumerate}
		\end{lemma}
		\begin{proof}
			We  establish \eqref{I:CC1}. Since $d>0$, it is well known that    $\Z[\tau_d]$ has at most $6$ units,  and  since 
			two elements belong to the same principal ideal if and only if their quotient is a unit, we get  the asserted bound.
			(We overcount the non-principal ideals with norm $|k|$.)

			We  establish \eqref{I:CC2}.  
			Let  $k\in \Z$ be nonzero. 	Dirichlet's unit theorem implies that in the case $d<0$ the group of units has rank $1$ and hence there exists a unit $u\in \Z[\tau_d]\subset\R$ such that 
			$u>1$ and every other unit has the form $\pm u^t$ for some $t\in \Z$. 
			Thus, if $m+n\tau_d$ and $m'+n'\tau_d$ with 
			$m,n,m',n'\in [N]$ both have norm $k$ and generate the same ideal, then  $m+n\tau_d=(m'+n'\tau_d)u^t$ for some $t\in \Z$. 
			
			We claim that there exists $c_d>0$  such that 
			\begin{equation}\label{E:tlog}
				|t|\ll_d \log_+ (c_d\, N/\sqrt{|k|}).
			\end{equation}
			Using the symmetries of the polynomials $P_d$ in \eqref{E:Qd'}, which  define the norm $\CN(m+n\tau_d)$, we can assume that  $m,n\in[N]$ as long as we allow
			$P_d(m,n)$ to be either $m^2+dn^2$ or $m^2\pm mn+\frac{d+1}{4}n^2$. 
			Since $P_d(m,n)=P_d(m',n')=k$, we get by \eqref{E:Qd'} that if $k>0$, then   
			$m,m'\geq \sqrt{k/2}$ (since  $P_d(m,n)\leq 2m^2$), while if $k<0$, then  $n,n'\geq \sqrt{|k|}/\sqrt{|d|}$ (since $P_d(m,n)\geq dn^2$).
			Also,  $(m'+n'\tau_d)u^t=m+n\tau_d \ll_dN$ if $t$ is positive 
			and
			$(m+n\tau_d)u^{-t}=m'+n'\tau_d\ll_d N$ if $t$ is negative.
			We deduce that   
			$\sqrt{|k|} u^{|t|} \ll_d N$  (we used here that $m,n,m',n', \tau_d$ are all non-negative). Since $u>1$, this establishes  \eqref{E:tlog}  and proves the claim.

			It follows from the above that 	$C_{d,N}(k)$ is bounded by $\log_+ (c_d\, N/\sqrt{|k|})$  times the number of principal ideals with norm $k$, which gives the asserted bound. 
		\end{proof}
		
		\subsubsection{Main estimate}
		The next estimate will be crucial for us and will be used to estimate the contribution of  ``large'' prime divisors in subsequent arguments. 
		\begin{lemma}\label{L:wNQl}
			Let $P\in\Z[m,n]$ be  an irreducible binary quadratic form. 
			For $l,A, Q,N \in \N$ and  $a,b\in \Z$ with  $-A\leq a,b\leq A$,   let	$w^*_{P,Q,N}(l)$ be as in 
			\eqref{E:wPQNl}. 
			If $l$ is a product of $s$, not necessarily distinct primes, in $\CP_P$, 
			then 
			\begin{equation}\label{E:wNQppq}
				w^*_{P,Q,N}(l)\ll_{A,P,s}  \frac{Q^2}{l}.
			\end{equation}
		\end{lemma}
		\begin{proof}
			{\bf Case where $P=P_d$, $Q=1,a=b=0$.} 	We first consider the case where $Q=1$, $a=b=0$, and $P=P_d$  is as in \eqref{E:Qd'} for some square-free $d\in \Z$. In this case, 
			we will  show that if $l$ is a product of $s$  not necessarily distinct  primes in $\CP_d$,	then 
			\begin{equation}\label{E:lN}
				|\{m,n\in [-N,N]\colon   l\mid  \CN(m+n\tau_d) \}|\ll_{d,s} \frac{N^2}{l}.
			\end{equation}
			Since $P_d(m,n)=\CN(m+n\tau_d)$, this clearly implies that \eqref{E:wNQppq} holds when  $P=P_d$ and $Q=1$, $a=b=0$.
			(The estimate  \eqref{E:wNQppq} only requires to consider positive $m,n$, but the extended range will be  needed in order to reduce the general polynomial $P$ to $P=P_d$.)
			
			For $k\in \Z$ 	let $C_{d,N}(k)$ be as in \eqref{E:CdNk}. Since  $|\CN(m+n\tau_d)|\ll_d N^2$
			if $m,n\in [-N,N]$,	 the left  side in \eqref{E:lN} is bounded by
			\begin{align*}
				\sum_{|r|\ll_d\frac{N^2}{l}} |\{m,n\in [-N,N]\colon  \CN(m+n\tau_d)=lr\}|
				=\sum_{|r|\ll_d \frac{N^2}{l}}	C_{d,N}(lr).
			\end{align*}
			Therefore, we will be done if  we show that
			\begin{equation}\label{E:Cdlr}
				\sum_{|r|\ll_d\frac{N^2}{l}}	C_{d,N}(lr)\ll_{d,s} \frac{N^2}{l}.
			\end{equation}
			We will prove this by considering the two cases $d>0$ and $d<0$ separately.

			Suppose first that $d>0$.
			Using part~\eqref{I:CC1} of \cref{L:CC} and then \eqref{E:CCdd} and \eqref{E:idealcount} of \cref{L:IdealCount},  we deduce that 
			$$
			\sum_{|r|\ll_d\frac{N^2}{l}}	C_{d,N}(lr)\ll
			\sum_{|r|\ll_d\frac{N^2}{l}}	\CC_d(l|r|)\ll_s \sum_{r\ll_d \frac{N^2}{l}}\CC_d(r)
			\ll_{d} \frac{N^2}{l},
			$$	
			establishing  \eqref{E:Cdlr} when $d>0$.

			Suppose now  that $d<0$. 
			Using part~\eqref{I:CC2} of \cref{L:CC} and then \eqref{E:CCdd} of \cref{L:IdealCount},  we deduce that 
			$$
			\sum_{1\leq |r|\ll_d \frac{N^2}{l}}	C_{d,N}(lr)\ll_{d} 		\sum_{1\leq r\ll_d \frac{N^2}{l}} \log_+(c_d\, N/\sqrt{lr})\cdot \CC_{d}(lr)\ll_{s} 		\sum_{1\leq r\leq c_d^2\,  \frac{N^2}{l}} \log (c_d\, N/\sqrt{lr})\cdot \CC_{d}(r).
			$$
			Let 
			$$
			S(r):=\sum_{1\leq n\leq r}\CC_d(n).
			$$
			Using partial summation,
			the estimate $\log(1+x)\leq x$ for $x\geq 0$, and \eqref{E:idealcount}, 	we get that the right  side is bounded by  
			\begin{equation}\label{E:SNl}
				\ll_d S\big([c_d^2\,  N^2/l]\big)+ \sum_{1\leq r\leq c_d^2\frac{N^2}{l}}\frac{S(r)}{r}\ll_d \frac{N^2}{l}.
			\end{equation}
			Combining the above, we get that \eqref{E:Cdlr} holds when $d<0$. 
			
			\smallskip

			{\bf (General case).}
			Let $P(m,n)=\alpha m^2+\beta mn+\gamma n^2$,  where $\alpha,\beta, \gamma \in \Z$,  be an irreducible   binary quadratic form and $a,b\in \Z$ with $-A\leq a,b\leq A$. We will reduce to the case covered in 
			the previous discussion by following an argument similar to the one used at the end of 
			\cref{SS:PfThm31}. 
			Since $P$ is irreducible, its discriminant is not a square.  
			Let $d$ be the unique square-free integer that satisfies 
			$$
			4\alpha \gamma -\beta ^2=r^2d
			$$ 
			for some $r\in \N$, and note that 
			$$
			4\alpha\, P(m,n)= P_d(A_d(m,n)),
			$$
			where  $A_d\colon \Z^2\mapsto \Z^2$ is the injective map given  in  \eqref{E:Ad} and $P_d$ is as in \eqref{E:Qd'}. 
			It follows that if  $l \mid P(Qm+a,Qn+b)$ for some $m,n\in [N]$, then  $l\mid P_d(m',n')$
			where $(m',n')=A_d(Qm+a,Qn+b) \in [-CQN, CQN]$ for some   $C:=C(A,\alpha, \beta, r)$ (we used that $|a|,|b|\leq A$).
			Note also that the map $(m,n)\mapsto (m',n')$ is injective  (since $A_d$ is injective). 
			We deduce that  
			\begin{multline*}
				\big|(m,n)\in [N]\times [N]\colon l\mid  	 P(Qm+a,Qn+b) \big| \leq\\   
				\big|(m',n')\in CQ\cdot ([-N,N]\times [-N,N])\colon l\mid  	 P_d(m',n') \big| \ll_{d,s} (CQ N)^2/l,
			\end{multline*}
			where the last estimate follows because  we have already covered the case $P=P_d$, $Q=1$,  $a=b=0$, and we also used the fact that $\CP_P=\CP_d$ up to a finite set of primes that have negligible contribution. This completes the proof. 
		\end{proof}
		
		\section{Concentration estimates - Proof of main results}\label{S:Concentration2}
		
		In this section we complete the proof of \cref{P:ConcentrationQuanti} and thus of \cref{P:concentration2General}.	Although the results and proofs presented are similar to those in [5], which cover the case where $P(m,n)=m^2+n^2$ and $c=1$, there are also a number of technical differences in the argument that necessitate giving it in full detail in order to facilitate proper understanding. 
		
		\subsection{Concentration estimates for additive functions} 
		We start with a  concentration estimate for additive functions, which  will eventually get  lifted to a concentration estimate for multiplicative functions.
		\begin{definition}
			We say that  $h\colon \mathbb{N}\to\mathbb{C}$ is {\em  additive}, if   it satisfies
			$h(mn)=h(m) +h(n)$ whenever $(m,n)=1$. We extend additive functions to the integers by $h(-n):=h(n)$ for $n\in \N$. 
		\end{definition}

		\begin{lemma}[Tur\'an-Kubilius inequality for quadratic forms]\label{L:TKadditive}	
			Let $P\in \Z[m,n]$ be an irreducible binary quadratic form 	and $\CP_P$ be as in  \eqref{E:CPP'}. 
			Let $K\in\N$ be sufficiently large (depending only on $P$) and let $N$ be sufficiently large (depending only on $K$). For any $A\in \N$, $a,b\in\Z$ with $-A\leq a,b\leq A$, $Q= \prod_{p\leq K}p^{a_p}$ for some $a_p\in \N$, and $c\in \N$ with $c\mid (Q,P(a,b))$, there exists $C_{A,P,Q}>0$  such that for any additive even function $h\colon \mathbb{Z}\to\C$ that is bounded by $1$ on primes and  satisfies
			\begin{enumerate}
				\item \label{I:hp1}  $h(p)=0$ for all primes $p\leq K$ and $p>N$;
				
				\item \label{I:hp2} $h(p)=0$  for all  primes $p\notin \CP_P $;
				
				\item \label{I:hp3} $h(p^k)=0$  for all   primes $p$ and  $k\geq 2$,
			\end{enumerate}
			we have
			\begin{multline}\label{E:varianceh}
				\E_{m,n\in [N]}\, \big|h\big(P_c(Qm+a,Qn+b)\big)-   H_N(h,K) \big|^2\ll_P \\  \D^2(h; K,\sqrt{N})+C_{A,P,Q}\cdot \D^2(h; \sqrt{N},N)+K^{-1},
			\end{multline}
			where $P_c:=P/c$, 
			\begin{equation}\label{E:HNhP0}
				H_{N}(h,K):=2\,   \sum_{K< p\leq N}\, \frac{ h(p)}{p}
			\end{equation}
			and
			\begin{equation}\label{E:D2}
				\D^2(h; K,N):=\sum_{K< p\leq N} \frac{|h(p)|^2}{p}.
			\end{equation}
		\end{lemma}
		\begin{proof} 
			Henceforth, we take $K$ sufficiently  large, depending only on $P$,  so that it exceeds all elements of the  finite set
			$F'$ considered in  	\cref{C:wNPQ}     and such that $\omega_P(p)= 2$ for every $p\in \CP_P$ with $p> K$  (for this it suffices to ensure that $p$ does not divide the discriminant of $P$, which is nonzero since $P$ is irreducible). Note also that since $p\mid Q$ implies $p\leq K$, for $p>K$ we have $p\notin F$ where $F$ is the finite set considered in \cref{C:wNPQ}.

			We consider the additive functions $h_1,h_2$, which are the restrictions  of $h$ on the primes $K<p\le \sqrt{N}$ and $\sqrt{N}<p\le N$.
			More precisely,
			$$
			h_1(p^k):= \begin{cases}h(p), \quad  &\text{if }\  k=1 \text{ and }   K<p\leq \sqrt{N}\\
				0, \quad &\text{otherwise,}  \end{cases}
			$$
			and
			$$
			h_2(p^k):=\begin{cases}h(p), \quad  &\text{if }\  k=1 \text{ and } \sqrt{N}<p\leq  N\\
				0, \quad &\text{otherwise,}  \end{cases},
			$$
			and we extend $h_1,h_2$ to even functions in $\Z$.   Then $h=h_1+h_2$.
			We also  define
			\begin{equation}\label{E:HiN}
				H_{N}(h_i,K):=2\, \sum_{K< p\leq N}\, \frac{h_i(p)}{p}, \quad i=1,2,
			\end{equation}
			and the  technical variant
			\begin{equation}\label{E:H'iN}
				H'_{N}(h_1,Q,K):= \sum_{K< p\leq N }\, w_{P,Q,N}(p)\cdot h_1(p),
			\end{equation}
			where 
			\begin{equation}\label{E:wNQp}
				w_{P,Q,N}(p):= 	\frac{1}{N^2}\, \sum_{\substack{m,n\in [N],\\   p\, \mid\mid\,  P(Qm+a,Qn+b)} } 1.
			\end{equation}
			(Note that $w_{P,Q,N}(p)=w_{P,Q,N}(p,p)$ where $w_{P,Q,N}(p,q)$ is as in \eqref{E:wNQpq}.)
			Note that as shown in  \eqref{E:c=1}
			of \cref{L:wNPQ} for $p\nmid Q$, which holds for $p>K$,  we have 
			$$
			w_{P,Q,N}(p):=\frac{1}{N^2}\, \sum_{\substack{m,n\in [N],\\   p\, \mid\mid\,  P_c(Qm+a,Qn+b)} } 1.
			$$
			
			The reason for introducing 	$H'_{N}(h_1,Q,K)$ is that  it gives the mean value of $h_1$ along
			integers of the form $P_c(Qm+a, Qn+b)$. Indeed,
			using  the properties \eqref{I:hp1}-\eqref{I:hp3}, and since $h_1$ is additive and even, we have
			\begin{multline}\label{E:meanhi}
				\E_{m,n\in [N]}\, h_1(P_c(Qm+a,Qn+b))=\E_{m,n\in [N]}\, \sum_{p\, \mid\mid\,  P(Qm+a, Qn+b)}h_1(p)\\=
				\frac{1}{N^2} \sum_{K<  p \leq N} h_1(p) \,  \sum_{\substack{m,n\in [N],\\   p\, \mid\mid\,  P_c(Qm+a,Qn+b)} } 1
				=
				H'_{N}(h_1,Q,K).
			\end{multline}

			Using the estimate \eqref{E:wpbound} of  \cref{C:wNPQ}  (which holds since  $p\notin F$ for $p>K$), and the fact that  $\omega_P(p)= 2$ for every $p\in \CP_P$ with $p> K$ and $h_1(p)=0$ if  $p\notin \CP_P$ or  $p> \sqrt{N}$ and $|h_1(p)|\leq 1$ for every $p\in \P$, we get
			\begin{equation}\label{E:H1H1'}
				|H_{N}(h_1,K)-H'_{N}(h_1,Q,K)|\ll
				\sum_{K< p\leq \sqrt{N}} \frac{1}{p^2}+\frac{1}{\sqrt{N}}
				\leq 	\frac{1}{K} +\frac{1}{\sqrt{N}}.
			\end{equation}
			Hence, in order to prove \eqref{E:varianceh},
			it suffices to estimate
			\begin{equation}\label{E:h1H1}
				\E_{m,n\in [N]}\big|h_1(P_c(Qm+a,Qn+b))-	H'_{N}(h_1,Q,K)\big|^2
			\end{equation}
			and
			\begin{equation}\label{E:h2H2}
				\E_{m,n\in [N]}\big|h_2(P_c(Qm+a,Qn+b))\big|^2+	|H_{N}(h_2,K)|^2.
			\end{equation}

			First, we deal with the expression \eqref{E:h1H1}.
			Using  \eqref{E:meanhi} and expanding the square below we get
			\begin{multline}\label{E:variance}
				\E_{m,n\in [N]} \, \big|h_1(P_c(Qm+a,Qn+b))- H'_{N}(h_1,Q,K)\big|^2=\\
				\E_{m,n\in [N]}\, \big|h_1(P_c(Qm+a,Qn+b))\big|^2 - |H'_{N}(h_1,Q,K)|^2.
			\end{multline}
			To estimate this expression, first note that since $h_1$ is additive, even,  and $h_1(p^k)=0$ for $k\geq 2$, we have
			\begin{equation}\label{E:h1Qsquare}
				\E_{m,n\in [N]}\, \big|h_1(P_c(Qm+a,Qn+b))\big|^2=
				\E_{m,n\in [N]}\,  \Big|\sum_{p\,  \mid \mid \,P_c(Qm+a,Qn+b)} h_1(p)\Big|^2.
			\end{equation}
			Expanding the square, using the  fact that $h_1(p)=0$ unless $K<p\leq \sqrt{N}$, and the definition of $w_{Q,N}(p,q)$ given in \eqref{E:wNQpq}, we get that the right side is equal to
			$$
			\sum_{K<p\leq \sqrt{N}} |h_1(p)|^2\cdot w_{P,Q,N}(p,p)+
			\sum_{K<p,q\leq \sqrt{N}, \,  p\neq q} h_1(p)\cdot  \overline{h_1(q)}\cdot w_{P,Q,N}(p,q).
			$$
			Using  the estimate \eqref{E:wpbound} of  \cref{C:wNPQ}, 
			we get that the first term is
			at most
			$$
			\sum_{K<p\leq \sqrt{N}}\frac{|h_1(p)|^2}{p}+ O_P(N^{-1/2}).
			$$
			Using  \eqref{E:wpqindep} of  \cref{C:wNPQ}  we get that the second term is equal to (crucially we use the bound $p,q\leq \sqrt{N}$ here and the prime number theorem)
			\begin{multline*}
				\sum_{K<p,q\leq \sqrt{N}, \,  p\neq q} h_1(p)\cdot \overline{h_1(q)}\cdot w_{P,Q,N}(p,p)\cdot w_{P,Q,N}(q,q) +O((\log{N})^{-2})\leq \\(H'_{N}(h_1,Q,K))^2+O((\log{N})^{-2}),
			\end{multline*}
			where to get the last estimate we  added to the sum the contribution of the diagonal terms $p=q$  (which is non-negative), used \eqref{E:H'iN}, and the fact that $h_1(p)=0$ for $p>\sqrt{N}$.
			Combining \eqref{E:variance} with the previous estimates, we are led to the bound
			\begin{multline}\label{E:variance1}
				\E_{m,n\in [N]} \, \big|h_1(P_c(Qm+a,Qn+b))- H'_{N}(h_1,Q,K)\big|^2\ll_P \\
				\D^2(h_1; K, \sqrt{N})+O((\log{N})^{-2}).
			\end{multline}
			
			Next  we estimate  the expression \eqref{E:h2H2}. Since $h_2$ is additive, even, and satisfies the properties \eqref{I:hp1}-\eqref{I:hp3}, we get by using \eqref{E:h1Qsquare} (with $h_2$ instead of $h_1$)   and expanding the square
			$$
			\E_{m,n\in [N]} \, \big|h_2(P_c(Qm+a,Qn+b))\big|^2= 	 \sum_{\sqrt{N}<   p, q \leq N} h_2(p)\, \overline{h_2(q)} \, w_{P,Q,N}(p,q).
			$$
			Since $h_2(p)\neq 0$ only when $p\in \CP_P$ (by property~\eqref{I:hp2}), 
			using  \cref{L:wNQl}  when $l$ is a product of at most two primes in $\CP_P$, we get that  the right  side is bounded  by
			\begin{multline*}
				\ll_{A,P,Q}  \sum_{\sqrt{N}< p,q\leq N} \frac{|h_2(p)| \, |h_2(q)|}{pq}+ \sum_{\sqrt{N}< p\leq N} \frac{|h_2(p)|^2}{p}= \\
				\Big(\sum_{\sqrt{N}< p\leq N} \frac{|h_2(p)|}{p}\Big)^2+\D^2(h_2; \sqrt{N},N) \ll \\  \sum_{\sqrt{N}< p\leq N} \frac{|h_2(p)|^2}{p}\cdot
				\sum_{\sqrt{N}< p\leq N} \frac{1}{p}+\D^2(h_2; \sqrt{N},N) \ll \D^2(h_2; \sqrt{N},N),
			\end{multline*}
			where we used crucially the estimate
			$$
			\sum_{\sqrt{N}< p\leq N} \frac{1}{p}\ll 1.
			$$
			Similarly we find 
			$$
			|H_{N}(h_2,K)|^2=4\, \Big(\sum_{\sqrt{N}< p\leq N} \frac{|h_2(p)|}{p}\Big)^2\ll \D^2(h_2; \sqrt{N},N).
			$$
			Combining the previous estimates, we get  the following bound for the expression in  \eqref{E:h2H2}
			\begin{equation}\label{E:h2H2bound}
				\E_{m,n\in [N]}\big|h_2(P_c(Qm+a,Qn+b))\big|^2+	|H_{N}(h_2,K)|^2\ll_{A,P,Q}   \D^2(h_2; \sqrt{N},N).
			\end{equation}
			
			Combining the bounds  \eqref{E:H1H1'}, \eqref{E:variance1}, \eqref{E:h2H2bound}, we get the asserted bound \eqref{E:varianceh}, which completes the proof.
		\end{proof}

		\subsection{Concentration estimates for multiplicative functions}
		Next we use Lemma~\ref{L:TKadditive} to get a variant that deals with multiplicative functions.
		\begin{lemma}\label{L:TKmultiplicative}	
			Let $P\in \Z[m,n]$ be an irreducible binary quadratic form
			and $\CP_P$ be as in  \eqref{E:Pd},	$A,K,N\in \N$,  $a,b\in \Z$ with  $-A\leq a,b\leq A$, $c\in \N$ with $c\mid (Q,P(a,b))$,  and  $f\colon \mathbb{Z}\to\mathbb{U}$ be an even multiplicative function such that
			\begin{enumerate}
				\item \label{I:fp1}  $f(p)=1$ for all primes $p\leq K$ and $p> N$;
				
				\item \label{I:fp2} $f(p)=1$  for all  primes $p\notin \CP_P$;
				
				\item \label{I:fp3} $f(p^k)=1$  for all   primes $p$ and  $k\geq 2$.
			\end{enumerate}
			Let also  $Q= \prod_{p\leq K}p^{a_p}$ with  $a_p\in \N$. 	Then there exists $C_{A,P,Q}>0$  such that  if $K$ is sufficiently large, depending only on $P$,
			and   $N$ is sufficiently  large, depending  only on   $K$, we have 
			\begin{multline}\label{E:mainestimate'}
				\E_{m,n\in [N]}\, \big|f\big(P_c(Qm+a,Qn+b)\big)-   \exp\big(G_{N}(f,K)\big)\big|\ll_P \\
				(\D+\D^2)(f,1; K,\sqrt{N})+ C_{A,P,Q}\cdot \D(f,1; \sqrt{N},N)+K^{-\frac{1}{2}},
			\end{multline}
			where $P_c:=P/c$,  $\D$ is as in \eqref{E:D2}, and
			\begin{equation}\label{E:GNfP01}
				G_{N}(f,K):=2\, \sum_{ K< p\leq  N}\, \frac{1}{p}\,(f(p) -1).
			\end{equation}
		\end{lemma}
		\begin{proof}
			Henceforth, we take $K$  sufficiently large, depending only on $P$,  so
			that it exceeds all elements of the  finite set
			$F'$ considered in  	\cref{C:wNPQ}, and also so that      \cref{L:TKadditive} holds  and $\omega_P(p)\leq 2$ for  $p> K$.

			Let $h:\mathbb{N}\to \mathbb{C}$ be the additive function given on prime powers by
			$$
			h(p^k):=f(p^k)-1, \quad k\in \N.
			$$
			Note that  due to our assumptions on $f,$ the properties \eqref{I:hp1}-\eqref{I:hp3} of Lemma~\ref{L:TKadditive} are satisfied.
			
			Using that $z=e^{z-1}+O(|z-1|^2)$ for $|z|\leq 1$ and property \eqref{I:hp3}, we have
			\begin{align*}
				f(P_c(m,n))=\prod_{p^k\, \mid \mid \, P_c(m,n)}f(p^k)=  \prod_{p\, \mid \mid \, P_c(m,n)}\big(\exp(h(p))+O(|h(p)|^2)\big).
			\end{align*}
			Applying the estimate $|\prod_{j\leq k}z_j-\prod_{j\leq k}w_j|\leq \sum_{j\leq k}|z_j-w_j|$, which holds if $|z_j|\leq 1$ and  $|w_j|\leq 1$ for $j=1,\ldots, k$, we deduce that for all $m,n\in\N$ we have
			\begin{align*}
				f(P_c(m,n))=\exp(h(P_c(m,n)))+O\Big(\sum_{p \, \mid \mid\,  P_c(m,n)}|h(p)|^2\Big).
			\end{align*}
			Using this and since $G_{N}(f,K)=H_{N}(h, K)$, where $H_{N}(h,K)$ is given by \eqref{E:HNhP0}, we get
			\begin{multline}\label{equ13}
				\E_{m,n\in [N]}\, \big|f\big(P_c(Qm+a,Qn+b)\big)-   \exp\big(G_{N}(f,K)\big)\big|
				\ll \\  \E_{m,n\in [N]}|\exp\big(h(P_c(Qm+a,Qn+b))\big)-\exp(H_{N}(h, K))|+\\
				\E_{m,n\in [N]}\sum_{p\, \mid \mid\,  P_c(Qm+a,Qn+b)}|h(p)|^2,
			\end{multline}
			where the implicit constant is absolute. 
			Next, we  use the inequality $|e^{z_1}-e^{z_2}|\le |z_1-z_2|$, which holds for $\Re z_1,\Re z_2\le 0$, to bound the last expression by
			\begin{align}\label{equ14}
				\E_{m,n\in [N]}|h(P_c(Qm+a,Qn+b))-H_{N}(h,K)|+\E_{m,n\in [N]}\sum_{p \, \mid \mid \, P_c(Qm+a,Qn+b)}|h(p)|^2.
			\end{align}
			To bound the first term in \eqref{equ14} we use Lemma~\ref{L:TKadditive}. It gives that 	there exist $C_{A,P,Q}>0$  such that  for all sufficiently large  $N$, depending only on  $K$, we have
			\begin{multline}\label{E:estimate1}
				\E_{m,n\in [N]}\, \big|h\big(P_c(Qm+a,Qn+b)\big)-   H_{N}(h,K) \big|\ll_P\\
				\D(h; K,\sqrt{N})+C_{A,P,Q}\cdot \D(h; \sqrt{N},N)+K^{-\frac{1}{2}}
				\ll \\
				\D(f,1; K,\sqrt{N})+ C_{A,P,Q}\cdot \D(f,1; \sqrt{N},N)+K^{-\frac{1}{2}},
			\end{multline}
			where for  the last bound we used that  $|h(p)|^2\leq 2-2\,\Re(f(p))$, which holds since $|f(p)|\leq 1$.
			To bound the second term in \eqref{equ14}, we note  that
			using the properties \eqref{I:hp1}-\eqref{I:hp3} of Lemma~\ref{L:TKadditive}, we have
			\begin{multline}\label{E:estimate2}
				\E_{m,n\in [N]}\sum_{p\, \mid \mid\,  P_c(Qm+a,Qn+b)}|h(p)|^2=
				\sum_{ K< p \leq N}\ |h(p)|^2\, w_{P,Q,N}(p)\ll_P \\
				\\	 \sum_{ K< p \leq N}\frac{|h(p)|^2}{p}+O((\log{N})^{-1}) \ll_P 	\mathbb{D}^2(f, 1;K,N)+O((\log{N})^{-1}),
			\end{multline}
			where $w_{P,Q,N}(p)$ is as in \eqref{E:wNQp} and we used equation  \eqref{E:wpbound} of \cref{C:wNPQ} (which holds since $p>K$ implies $p\nmid Q$ and $p\notin F'$), the bound $\omega_P(p)\leq 2$ (which holds for $p>K$),   and the prime number theorem to get the first bound.
			Combining \eqref{equ13}-\eqref{E:estimate2} we get the asserted bound. 	
		\end{proof}
		We use the previous result to derive the following improved version.
		\begin{lemma}\label{L:TKmultiplicative2}	
			Let $P\in \Z[m,n]$ be an irreducible binary quadratic form, 
			$\CP_P$ be as in \eqref{E:CPP'},
			$A,K,N\in \N$, and  $f\colon \mathbb{N}\to\mathbb{U}$ be an even multiplicative function such that   $f(p)= 1$ for all primes $p> N$ with $p\in \CP_P$. 	Let also  $Q= \prod_{p\leq K} p^{a_p}$ with  $a_p\in \N$. There exists $C_{A,P,Q}>0$ such that  if $K$ is  sufficiently large, depending only on $P$,
			and $N$ is sufficiently large, depending  only on  $K$,
			then   for all  $a,b\in \Z$ with  $-A\leq a,b\leq A$ and  
			such that  $ \prod_{p\leq K} p$ divides $Q/c$ where $c:=(P(a,b),Q)$, we have 
			\begin{multline}\label{E:mainestimate}
				\E_{m,n\in [N]}\, \big|f\big(P_c(Qm+a,Qn+b)\big)- \exp\big(G_{P,N}(f,K)\big)\big|\ll_P \\	(\D_P+\D_P^2)(f,1; K,\sqrt{N})+ C_{A,P,Q}\cdot \D_P(f,1; \sqrt{N},N)+K^{-\frac{1}{2}},
			\end{multline}
			where 
			\begin{equation}\label{E:GNfP01'}
				G_{P,N}(f,K):=\sum_{ \substack{K< p\leq  N}}\, \frac{\omega_P(p)}{p}\,(f(p) -1),
			\end{equation}
			$$
			\D_P(f,1; x,y)^2:= \sum_{\substack{x< p\leq y}} \frac{\omega_P(p)}{p}\, (1-\Re(f(p))).
			$$
			(Note that both sums are supported on the primes in $\CP_P$ and $\omega_P(p)=2$ for $p\in \CP_P$ with $p> K$ as long as $K$ is  sufficiently large.)
		\end{lemma}
		\begin{proof}		
			Let $P(m,n)=\alpha m^2+\beta mn + \gamma n^2$.  Note that since $P$ is irreducible we have  $\alpha\neq 0$. 	We can assume that $K$ is greater than $|\alpha|$ and that $\omega_P(p)\leq 2$ for all $p>K$;  this holds as long as $p$ does not divide the discriminant of $P$, which is nonzero since $P$ is irreducible. We also assume that $K$ is sufficiently large  for  \cref{L:TKmultiplicative} to hold. 
			
			First, we  define the multiplicative function $\tilde{f}\colon \N\to \U$ on prime powers as follows
			$$
			\tilde{f}(p^k):= \begin{cases}f(p^k), \quad  &\text{if }\   p> K,\,  k\in \N \\
				1, \quad &\text{otherwise}  \end{cases}.
			$$
			Since by assumption    every prime $p\leq K$  divides  $Q/c$ and since  $(P(a,b),Q)=c$,
			we have   
			\begin{equation}\label{E:pPc}
				p\nmid P_c(Qm+a,Qn+b) \text{  for every } p\leq K  \text{ and } m,n\in \N.
			\end{equation} 
			Since $f$ is multiplicative we  deduce that
			$$
			f(P_c(Qm+a,Qn+b))=\tilde{f}(P_c(Qm+a,Qn+b)) \quad \text{for every } m,n\in\N.
			$$
			Note also that $G_{P,N}(\tilde{f},K)=G_{P,N}(f,K)$ and $\D_P(\tilde{f},1; K,N)=\D_P(f,1;K,N)$.
			It follows that in order to establish \eqref{E:mainestimate}, it is  enough to show that there exist $C_{A,P,Q}>0$ such that for all sufficiently large  $N$, depending only on $A$ and  $K$, we have
			\begin{multline}\label{E:mainestimate2}
				\E_{m,n\in [N]}|\tilde{f}(P_c(Qm+a,Qn+b))- \exp(G_{P,N}(\tilde{f},K))|\ll_P  \\	(\D_P+\D_P^2)(\tilde{f},1; K,\sqrt{N})+ C_{A,P,Q}\cdot \D_P(\tilde{f},1; \sqrt{N},N)+	K^{-\frac{1}{2}}.
			\end{multline}
			In order to establish \eqref{E:mainestimate2} we
			make a series of further reductions that will eventually allow us to apply Lemma~\ref{L:TKmultiplicative}. By the definition of $\CP_P$ (see \eqref{E:CPP'}), for each $p\not\in\CP_P$ that does not divide $\alpha$
			(recall that $K>|\alpha|$ so $p>K$ implies that $p\nmid \alpha$), we have that $p\mid P(m,n)$ implies that $p\mid m$ and $p\mid n$. 
			Using this and \eqref{E:pPc}, we get that 	 the contribution to the average in  \eqref{E:mainestimate2}  of those $m,n\in [N]$ for which $P_c(Qm+a,Qn+b)$ is divisible by some prime  $p\not\in \CP_P$   is
			$$
			\ll \frac{1}{N^2}\sum_{K< p\leq N} \left[\frac{N}{p}\right]^2\ll \frac{1}{K},
			$$
			which is acceptable.
			
			Next we show  that the contribution to the average  in  \eqref{E:mainestimate2} of those $m,n\in [N]$ for which $ P_c(Qm+a, Qn+b)$ is divisible by $p^2$ for some  prime $p\in \CP_P$ with $p>K$ (hence $p\nmid Q$) is also acceptable. Note that  this contribution is maximized for $c=1$, hence  it is sufficient to check this for $c=1$, that is, for the polynomial $P$ instead of $P_c$, assuming that $(P(a,b),Q)=1$.  For
			fixed $n\in [N]$ such that  $p\nmid Qn+b$  there are at most  $\omega_P(p^2)\cdot ([N/p^2]+1)$ values of $m\in [N]$ such that $p^2\mid P(Qm+a,Qn+b).$ On the other hand,  if $p>K$ satisfies $p\mid Qn+b$ and $p\mid  P(Qm+a, Qn+b)$,  then also $p\mid Qm+a$ (recall that $p>K$ implies that $p\nmid \alpha$).  Since  $P(n,1)$ is an irreducible quadratic polynomial  we have $\omega_P(p^2)\leq 2$  for all but finitely many $p\in \P$.
			So  the contribution to the average  in \eqref{E:mainestimate2} of those  $m,n\in [N]$
			for which $ P(Qm+a,Qn+b)$ is divisible by $p^2$ for some  prime $p\in \CP_P$
			is bounded by (note again that our assumptions imply that $P(Qm+a,Qn+b)$ is only divisible by primes $p> K$)
			$$
			\ll_P \frac{1}{N^2} \Big(\sum_{K< p\leq N}\Big(\left[\frac{N}{p^2}\right]+1\Big)N+\sum_{K< p\leq N}\left[\frac{N}{p}\right]^2\Big)\ll_P \frac{1}{K}+\frac{1}{\log{N}},
			$$
			where  
			we used the prime number theorem to bound $\frac{1}{N}\sum_{K<p\leq N} 1$.

			Combining the above reductions, we deduce that in order to establish the estimate  \eqref{E:mainestimate2}, we can further assume that
			\begin{equation}\label{E:freduced}
				\tilde{f}(p^k)=1 \ \text{ for all } \ p\in \P,\,  k\ge 2, \ \text{  and }\  \tilde{f}(p^k)=1 \ \text{ for all } \ p\not\in \CP_P, \, k\in \N.
			\end{equation}
			We are now in a situation where Lemma~\ref{L:TKmultiplicative} applies, giving that
			there exist $C_{A,P,Q}>0$  such that  for all  sufficiently large  $K$, depending only on $P$, and 
			all sufficiently large $N$, depending only on $K$, if $\D_P(f,1;K,N)\leq 1$, we have
			(note that \eqref{E:freduced} implies that $\mathbb{D}_P(\tilde{f}, 1;K,N)=\mathbb{D}(\tilde{f}, 1;K,N)$ where $\D$ is as in \eqref{E:D2})
			\begin{multline*}
				\E_{m,n\in [N]}|\tilde{f}(P_c(Qm+a,Qn+b))- \exp(G_{P,N}(\tilde{f},K))|\ll_P \\
				(\D_P+\D_P^2)(\tilde{f},1; K,\sqrt{N})+ C_{A,P,Q}\cdot \D_P(\tilde{f},1; \sqrt{N},N)+K^{-\frac{1}{2}}.
			\end{multline*}
			Combining this bound with the bounds we got in order to arrive at this reduction, we get that \eqref{E:mainestimate2} is satisfied. This completes the proof.
		\end{proof}
		\subsection{Proof of \cref{P:ConcentrationQuanti}}	
		We begin with some reductions. Assuming  that the statement holds when $\chi=1$ and $t=0$, we will show that it holds for arbitrary $q$-periodic Dirichlet character $\chi$ and $t\in\mathbb{R}$. Recall that $P_c=P/c$ where $c=(P(a,b),Q)$ and by assumption $q\mid Q/c$.
		Let 
		$\tilde{f}:=f \cdot \overline{\chi}\cdot n^{-it}$,  and apply the conclusion for $\chi=1, t=0$.
		We get the following bound for $\tilde{f}$
		\begin{multline}\label{E:tildef}
			\E_{m,n\in [N]}\, \big|\tilde{f}\big(P_c(Qm+a,Qn+b)\big)-  \exp\big(G_{P,N}(\tilde{f},K)\big)\big|\ll_P \\
			(\D_P+\D_P^2)(\tilde{f},1; K,\sqrt{N})+C_{P,Q}\cdot \D_P(\tilde{f},1;N,C_{P,Q}N^2)+	 C_{P,Q} \cdot \D_P(\tilde{f},1; \sqrt{N},N)+K^{-1/2}.
		\end{multline}
		Note that since $\chi$ is  $q$-periodic and  $q$ divides $Q/c$, we have $\chi(P_c(Qm+a,Qn+b))=\chi(P_c(a,b))$ for every $m,n\in \N$. Furthermore,
		since   $(P_c(a,b), Q)=1$ and $q\mid Q$, we have $(P_c(a,b), q)=1$,  so $|\chi(P_c(a,b))|=1$.
		Also,   
		$\D_P(\tilde{f}, 1;x,y)=\D_P(f, \chi\cdot n^{it};x,y)$ and
		$$
		G_{P,N}(\tilde{f},K)=\sum_{ K< p\leq N}\, \frac{\omega_P(p)}{p}\,(\tilde{f}(p) -1)=
		\sum_{ K< p\leq N}\, \frac{\omega_P(p)}{p}\,(f(p)\cdot \overline{\chi(p)}\cdot p^{-it} -1)=G_{P, N}(f,K).
		$$
		After inserting this information into \eqref{E:tildef}, we get that \eqref{E:generalfP} is satisfied.
		
		So it suffices to show that if    $Q= \prod_{p\leq K}p^{a_p}$ for some  $a_p\in \N$, then there exists $C_{A,P,Q}>0$ such that  if $N$ is  sufficiently large, depending only on  $A,P,Q$, and $\mathbb{D}_P(f, 1;K,N)\leq 1$, we have
		\begin{multline}\label{E:mainreduced}
			\E_{m,n\in [N]}\, \big|f\big(P_c(Qm+a,Qn+b)\big)- \exp\big(G_{P,N}(f,K)\big)\big|\ll_P\\ 
			(\D_P+\D_P^2)(f,1; K,\sqrt{N})+ C_{A,P,Q} \cdot D_P(f,1;N,C_{A,P,Q}N^2)+	 C_{A,P,Q}\cdot \D_P(f,1; \sqrt{N},N)+K^{-1/2}.
		\end{multline}

		For every $N\in \N$ we  decompose $f$ as  $f=f_{N,1} \cdot f_{N,2}$, where the multiplicative functions $f_{N,1},f_{N,2}\colon \N\to \U$  are defined on prime powers as follows
		\begin{align*}
			f_{N,1}(p^k):=& \begin{cases}f(p), \quad  &\text{if }\ k=1 \text{ and }   p> N,\,  p\in \CP_P \\
				1, \quad &\text{otherwise}  \end{cases}, \\
			f_{N,2}(p^k):=& \begin{cases}1, \quad  &\text{if }\ k=1 \text{ and }   p> N, \, p\in \CP_P \\
				f(p^k), \quad &\text{otherwise}  \end{cases}.
		\end{align*}

		We first study the function $f_{N,1}$. Following the notation of Lemma~\ref{L:wNQl}, for $l,Q,N\in \N$ we let
		$$
		w^*_{P,Q,N}(l)=\frac{1}{N^2} \, \sum_{\substack{m,n\in [N],\\   l\mid  P(Qm+a,Qn+b)} } 1.
		$$
		We apply Lemma~\ref{L:wNQl}. We get   that if $l$ is a product of at most two  primes in $\CP_P$, then
		\begin{equation}\label{E:wNQ}
			w^*_{P,Q,N}(l)\ll_{A,P,Q} \frac{1}{l}.
		\end{equation}
		Since for $N$ sufficiently large, depending only on $A,P,Q$,  we have $f_{N,1}(P_c(Qm+a,Qn+b))-1\neq 0$ only if $P_c(Qm+a,Qn+b)$ is divisible  by one or two prime numbers $p>N$,\footnote{For $m,n\in [N]$ and $-A\leq a,b\leq A$  we have $P(Qm+a,Qn+b)\ll_{A,P,Q} N^2$,  so if $P(Qm+a,Qn+b)$ was divisible by three or more primes greater than $N$, we would have $N^3\ll_{A,P,Q} N^2$, which fails if $N$ is sufficiently large, depending only on $A$, $P$,  $Q$.}
		we get\footnote{Note that for $N>c$ we have  $f_{N,1}(c)=1$ so 
			$f_{N,1}\big(P_c(Qm+a,Qn+b)\big)=f_{N,1}\big(P(Qm+a,Qn+b)\big)$ for those values of $N$.}
		\begin{multline}\label{E:f11}
			\E_{m,n\in[N]}
			|f_{N,1}\big(P_c(Qm+a,Qn+b)\big)-1|
			\leq \\
			\sum_{\substack{N<p\ll_{A,P,Q}
					N^2,\\ p\in \CP_P}} |f(p)-1|\, w^*_{P,Q,N}(p)
			+ \sum_{\substack{N<p,q\ll_{A,P,Q}
					N^2,\, p\neq q\\ p,q\in \CP_P}} |f(pq)-1|\,  w^*_{P,Q,N}(pq),
		\end{multline}
		where we have used that $f_{N,1}(p)=f(p)$ for all $p>N$ and in the second sum we have ignored the contribution of the diagonal terms $p=q$ since, by construction, $f_{N,1}(p^2)=1$ for all primes $p$.
		Using  \eqref{E:wNQ} for $l:=p\in \CP_P$,   we estimate the first term as follows\footnote{Bounding $w_{P,Q,N}(p)$ using \eqref{E:wNPQ1}
			would lead to unacceptable errors here, because the  summation range is much larger than $N$.}
		\begin{multline*}
			\sum_{\substack{N<p\ll_P
					N^2,\\ p\in \CP_P}} |f(p)-1|\, w^*_{P,Q,N}(p) \ll_{A,P,Q}	\sum_{\substack{N<p\ll_{A,P,Q}
					N^2,\\ p\in \CP_P}} \frac{|f(p)-1|}{p}\ll \\
			\Big(\sum_{\substack{N<p\ll_{A,P,Q}
					N^2,\\ p\in \CP_P}}
			\frac{|f(p)-1|^2}{p}\Big)^{\frac{1}{2}}\cdot
			\Big(\sum_{\substack{N<p\ll_{A,P,Q}
					N^2,\\ p\in \CP_P}} \frac{1}{p} \Big)^{\frac{1}{2}} \ll_{A,P,Q}
			\D_P(f,1;N,O_{A,P,Q}(N^2)),
		\end{multline*}
		where we used  that $\sum_{N\le p\ll_{A,P,Q} N^2}\frac{1}{p}\ll_{A,P,Q} 1$.
		Similarly, using  \eqref{E:wNQ} for $l:=pq$ with $p,q\in \CP_P$, we estimate the second  term in \eqref{E:f11} as follows (note  that since $p\neq q$,  we have $f(pq)=f(p)f(q)$)
		\begin{multline*}
			\sum_{\substack{N<p,q\ll_{A,P,Q}
					N^2,\, p\neq q, \\ p,q\in \CP_P}} |f(pq)-1|\, w^*_{P,Q,N}(pq) \ll_{A,P,Q}	\sum_{\substack{N<p,q\ll_{A,P,Q}
					N^2,\\ p,q\in \CP_P}}\frac{|f(p)-1|+|f(q)-1|}{pq}\ll \\
			\Big(\sum_{\substack{N<p,q\ll_{A,P,Q}
					N^2,\\ p,q\in \CP_P}}
			\frac{|f(p)-1|^2}{pq}\Big)^{\frac{1}{2}}\cdot
			\Big(\sum_{\substack{N<p,q\ll_{A,P,Q}
					N^2,\\ p,q\in \CP_P}} \frac{1}{pq} \Big)^{\frac{1}{2}} \ll_{A,P,Q}
			\mathbb{D}_P(f,1;N,O_{A,P,Q}(N^2)),
		\end{multline*}
		where
		we used  that $\sum_{N\le p\ll_{A,P,Q} N^2}\frac{1}{p}\ll_{A,P,Q} 1$. 	
		Combining the above estimates and \eqref{E:f11}, we deduce that for $N\gg_{P,Q} 1$ we have
		\begin{equation}\label{E:f1estimate}
			\mathbb{E}_{m,n\in [N]}\big|f_{N,1}\big(P_c(Qm+a,Qn+b)\big)-1|\ll_{A,P,Q}   \mathbb{D}_P(f,1;N,O_{A,P,Q}(N^2)).
		\end{equation}

		Next, we move to the function $f_2$. Since $f_2(p)=1 $ for all primes $p\geq N$,
		Lemma~\ref{L:TKmultiplicative2}  applies.  We get that if $N$ is  sufficiently large, depending only on $A$ and $K$, we have
		\begin{multline}\label{E:f2estimate}
			\mathbb{E}_{m,n\in [N]}|f_{N,2}\big(P_c(Qm+a,Qn+b)\big)- \exp(G_{P,N}(f,K))|
			\ll_P \\
			(\D_P+\D_P^2)(f,1; K,\sqrt{N})+ O_{A,P,Q}( \D_P(f,1; \sqrt{N},N))+K^{-\frac{1}{2}},
		\end{multline}
		where we used that $f_{N,2}(p)=f(p)$ for all primes $p\in \CP_P$ with $p\leq N$, hence
		we have 
		$G_{P,N}(f_{N,2},K)=G_{P,N}(f,K)$ and $\mathbb{D}_P(f_{N,2},1;K,N)=\mathbb{D}_P(f,1;K,N)$.

		Finally, since $f=f_{N,1}\cdot f_{N,2}$, using  the triangle inequality and combining \eqref{E:f1estimate} and  \eqref{E:f2estimate},
		we get  that the left  side  in \eqref{E:mainreduced} is bounded by
		\begin{multline*}	
			\mathbb{E}_{m,n\in [N]}\big(|f_{N,1}\big(P_c(Qm+a,Qn+b)\big)- 1|		
			+|f_{N,2}\big(P_c(Qm+a,Qn+b)\big)- \exp(G_{P,N}(f,K))|\big) \\\ll_P  
			(\D_P+\D_P^2)(f,1; K,\sqrt{N})+ C_{A,P,Q}\cdot \mathbb{D}_P(f,1;N,C_{A,P,Q}N^2)+\\ C_{A,P,Q} \cdot \D_P(f,1; \sqrt{N},N)+K^{-1/2}
		\end{multline*}	
		for some $C_{A,P,Q}>0$. 
		Thus \eqref{E:mainreduced} holds, completing the proof.

	\section{Proof of partition regularity results}\label{S:PRproofs}
	Our main goal in this section is to complete the  proof of  Theorem~\ref{T:MainMulti'}. 
	As explained in Section~\ref{SS:Plan},  it remains to prove Propositions \ref{P:aperiodicQ12} and \ref{P:nonnegative}.  We  devote the rest of this section to this task.
	
	\subsection{Some elementary facts}
	We will use the following  fact from \cite[Lemma~3.1]{FKM23}.
	\begin{lemma}\label{L:lN}
		Let $a\colon \Z\to \U$ be   an even sequence and $l_1,l_2\in \Z$, not both of them $0$. Suppose that for some $\varepsilon>0$ and for some sequence $L_N\colon \N\to \U$ we have
		$$
		\limsup_{N\to\infty} \E_{n\in [N]} |a(n) -L_N|\leq \varepsilon.
		$$
		Then
		$$
		\limsup_{N\to\infty} \E_{m,n\in [N]} |a(l_1m+l_2n)-L_{lN}|\leq 2l \cdot  \varepsilon,
		$$
		where $l:=|l_1|+|l_2|$.
	\end{lemma}
	We will also use the following simple result from \cite[Lemma~3.2]{FKM23}.
	\begin{lemma}\label{L:Fol0}
		Let  $(\Phi_K)$ be a multiplicative  F\o lner  sequence.
		If $f\colon \N\to \U$ is a completely multiplicative function
		and $f\neq 1$, then
		$$
		\lim_{K\to\infty}\E_{n\in \Phi_K}\, f(n)=0.
		$$
	\end{lemma}
	\subsection{A Borel measurability result}
	Recall that 	if $f$ is pretentious, then  there exist a  unique $t=t_f\in \R$ and a  Dirichlet character $\chi$ such that $f\sim \chi\cdot n^{it}$. Similarly, if $P$ is an irreducible binary quadratic form and $f$ is $P$-pretentious,  then  it follows from \cref{L:tunique} that there exist a  unique $t=t_f\in \R$ and a  (not necessarily unique)  Dirichlet character $\chi$ such that $f\sim_P \chi \cdot  n^{it}$. Moreover, $t_f$ does not depend on the irreducible binary quadratic form $P$, as long as $f$ is $P$-pretentious.
	\begin{lemma}\label{L:Borel}
		Let $P_1,P_2\in \Z[m,n]$ be  irreducible binary quadratic forms. Then 
		\begin{enumerate}
			\item\label{I:Borel1}	The sets $\CM_p$ and $\CM_{p,P_1,P_2}$,
			defined in \eqref{E:Mp} and \eqref{E:MP12}, respectively,  are Borel subsets of $\CM$.
			
			\item \label{I:Borel2} The maps $f\mapsto t_f$ from $\CM_p$ to $\R$ and from $\CM_{p,P_1,P_2}$  to $\R$ are Borel measurable.
			
			
		\end{enumerate}
	\end{lemma}
	\begin{proof}
		Both parts were proved for the set $\CM_p$ in \cite[Lemma~3.6]{FKM23}. Thus, 
		we  deal only  with the results related to  $\CM_{p,P_1,P_2}$.

		We remark that we are not able to show directly that  $\CM_{p,P_1,P_2}$ is Borel (as we did for $\CM_p$ in  \cite[Lemma~3.6]{FKM23}),  because we do not have a convenient characterization for this set that depends on countably many conditions. 	So we  take a different approach  and first establish a variant of   property \eqref{I:Borel2}. 
		
		We  show that the following map is Borel: $f\mapsto t_f$ from $\CM$ to $\R\cup \{\infty\}$ ($\infty$ is just an isolated point added to $\R$), which maps $f$ to $\infty$ if $f$ is not $P_1$-pretentious or $f$ is not  $P_2$-pretentious, and otherwise maps  $f$ to the unique $t_f\in \R$ for which $f\sim_{P_j} \chi_j\cdot n^{it_f}$, $j=1,2$,  for some Dirichlet character $\chi_j$. 
		By  \cite[Theorem~14.12]{Ke12}, it suffices to show that the graph
		$$
		\Gamma:=\{(f,t_f)\in \CM\times (\R\cup \{\infty\})\}
		$$
		is a Borel subset of $\CM\times  (\R\cup \{\infty\})$.
		If $\chi_k$, $k\in\N$, is an enumeration of all Dirichlet characters, and for $k_1,k_2\in \N$ we let 
		$$
		\Gamma_{k_1,k_2}:=\{(f,t_f)\in \CM\times (\R\cup \{\infty\})\colon f\sim_{P_j} \chi_{k_j}\cdot n^{it_f} \text{ for } j=1,2\},
		$$
		and 
		$$
		\Gamma_\infty=\{(f,\infty)\colon f \text{ is not } P_j\text{-pretentious for  } j=1 \text{ or } 2\},
		$$
		then
		$$
		\Gamma=\cup_{k_1,k_2\in\N} \Gamma_{k_1,k_2}\cup \Gamma_\infty.
		$$
		Therefore, it suffices to show that for the sets $\Gamma_\infty$ and  $\Gamma_{k_1,k_2}$, $k_1,k_2\in \N,$ are Borel.
		Note that for $k_1,k_2\in \N$ we have  
		$$
		\Gamma_{k_1,k_2}:=\{(f,t)\in\CM\times \R\colon \D_{P_j}(f, \chi_{k_j}\cdot n^{it})<\infty \text{ for } j=1,2\}
		$$
		and 
		\begin{multline*}
			\Gamma_\infty :=\{(f,t)\in\CM\times  \{+\infty\}\colon \text{for }  j=1 \text{ or } 2 \text{ we have } \\ \D_{P_j}(f,\chi_k\cdot n^{it})=\infty \text{ for every } t\in \R, k\in \N\}.
		\end{multline*}
		It easily follows  from the definition of the distance $\D_P$ 	that for $k\in \N$  and $j=1,2$  the map $(f,t)\mapsto \D_{P_j}(f,\chi_k\cdot n^{it})$ is  Borel. Hence,   the sets $\Gamma_\infty$ and  $\Gamma_{k_1,k_2}$, $k_1,k_2\in \N,$ are Borel.
		
		Now  it follows that $\CM_{p,P_1,P_2}=\{f\in \CM\colon t_f\neq \infty\}$ is Borel
		and the map $t\mapsto t_f$ from $\CM_{p,P_1,P_2}$ to $\R$ is Borel. 
		This completes the proof.
	\end{proof}
	
	\subsection{Some useful weights}
	In the proof of  Theorems~\ref{T:PairsPartition} and \ref{T:PairsDensity} we utilize weighted averages.  The weights are used to ensure that
	the averages 
	$L_{\delta,c,N}(f,Q;a,b)$,  as in  \eqref{E:LQdab},   have a positive real part if $f$ is an Archimedean character and   $\delta$ is sufficiently small.
	
	We now introduce these weights.\rem{Changed the definition of the weights and made adjustments in  the proof similar to those in the Pythagorean pairs article. Replaced reducibility condition in $(i)$ with  $\deg \tilde{P}_1>\deg \tilde{P}_2$, which is what is really  needed, and changed the proof a bit.}
		For every $\delta\in (0,1/2)$,    let $F_\delta \colon \R\to [0,1]$ be the continuous  trapezoid function that is equal to $1$ on $[-\delta/2, \delta/2]$, equal to  $0$ outside $[-\delta,\delta]$, and linear on the remaining two intervals. For $c\in \R$ let 
	\begin{equation}\label{E:weight2}
		w_{\delta,c,P_1,P_2}(m,n):=F_\delta \Big(\log \frac{P_1(m,n)}{P_2(m,n)}-c\Big) \cdot {\bf 1}_{P_j(m,n)>0, j=1,2}, \quad m,n\in\N.
	\end{equation}
	\begin{lemma}\label{L:SdeltaGeneral}  Let $P_1, P_2\in \Z[m,n]$ be homogeneous polynomials of the same degree. 
		For $j=1,2$ let $\tilde{P}_j(x):=P_j(x,1)$ and suppose   that the leading coefficients of the polynomials $\tilde{P}_1$, $\tilde{P}_2$ are positive.  Let also 	$w_{\delta,c,P_1,P_2}$ be as in \eqref{E:weight2}.	The following properties hold:
				\begin{enumerate}[(i)]
			\item  If  $\deg \tilde{P}_1>\deg \tilde{P}_2$, then there exists $L_0\in \R$ such that for every $c\geq L_0$ and every $\delta \in (0,1/2)$ we have 
			$$
			\mu_{\delta,c,P_1,P_2}:=	\lim_{N\to\infty} \E_{m,n\in [N]}\,  w_{\delta,c,P_1,P_2}(m,n)>0.
			$$
			
			\item If the polynomials $\tilde{P}_1$, $\tilde{P}_2$ have the same degree and the same leading coefficients, then for every $\delta \in (0,1/2)$ we have 
				$$
			\mu_{\delta,0,P_1,P_2}:=	\lim_{N\to\infty} \E_{m,n\in [N]}\,  w_{\delta,0,P_1,P_2}(m,n)>0.
			$$
		\end{enumerate}
	\end{lemma}
	\begin{proof}	
		Using the fact that the polynomials $P_1,P_2$ are homogeneous and have the same degree, we get that  the limit we want to evaluate is equal to
		$$
		\lim_{N\to\infty} \E_{m,n\in [N]}\,
		F_{\delta} \Big(\log\frac{P_1(m/N,n/N)}{P_2(m/N,n/N)}-c\Big)
		\cdot 
		{\bf 1}_{P_j(m/N,n/N)>0, j=1,2}.
		$$
		Let
		$\tilde{F}_{\delta,c}\colon [0,1]\times [0,1]\to [0,1]$ be given by
		$$
		\tilde{F}_{\delta,c}(x,y):=F_\delta\Big(\log\frac{P_1(x,y)}{P_2(x,y)}-c\Big) \cdot {\bf 1}_{P_j(x,y)>0,j=1,2}, \quad x,y\in [0,1].
		$$
		Then $\tilde{F}_{\delta,c}$ is Riemann integrable on $[0,1]\times [0,1]$ since  it is bounded and continuous except for a set of Lebesgue measure $0$.  Therefore, the limit we want to compute exists and is equal to the  Riemann integral
		$$
		\int_0^1\int_0^1 \tilde{F}_{\delta,c}(x,y)\, dx\, dy.
		$$
		Since  $\tilde{F}_{\delta,c}$ is non-negative,  in order to show that 
		the integral is positive, it suffices to show that   $\tilde{F}_{\delta,c}$  does not vanish almost everywhere.  
		
		We prove $(i)$.  Since the leading coefficients of $\tilde{P}_1$ and $\tilde{P}_2$ are positive and $\deg \tilde{P}_1>\deg \tilde{P}_2$, we have 
		$
		\lim_{x\to\infty} \tilde{P}_1(x)/\tilde{P}_2(x)=+\infty
		$
		and for all sufficiently large $x$ we have $\tilde{P}_j(x)>0$ for $j=1,2$. Hence, there exists $L_0\in \R$ such that  for $c\geq L_0$ there exists $\alpha_c>0$ such that $\tilde{P}_j(\alpha_c)>0$ for $j=1,2$ and $\log (\tilde{P}_1(\alpha_c)/\tilde{P}_2(\alpha_c))=c$. Using homogeneity and that $P_1,P_2$ have the same degree, we conclude that for every $t>0$ we have 
		$$
		\log \frac{P_1(\alpha_c t,t)}{P_2(\alpha_c t,t)}=c, \quad P_j(\alpha_c t, t)>0, \quad j=1,2. 
		$$
		Using continuity away from the zero sets of  $P_1$ and $P_2$, we deduce that on a neighborhood of the line segment $(\alpha_c t, t)$ that lies within the square $[0,1]\times [0,1]$ we have 
		$$
		\Big|\log\frac{P_1(x,y)}{P_2(x,y)}-c\Big|\leq \frac{\delta}{2}, \qquad P_j(x,y)>0, \quad j=1,2. 
		$$ 
		Therefore, for $c\geq L_0$ we have  $\tilde{F}_{\delta,c}(x,y)=1$ on a subset of $[0,1]\times [0,1]$ of positive measure, as needed.

		We prove  $(ii)$.  Our assumption that the polynomials $\tilde{P}_1$, $\tilde{P}_2$ have the same degree and the same leading coefficients
		gives that 
		$\lim_{k\to+\infty} P_1(k,1)/P_2(k,1)=1,$ and  since the polynomials $P_j(m,n)$ are homogeneous and have the same degree, we get  that 
		$\lim_{k\to+\infty} P_1(kt,t)/P_2(kt,t)=1$ for every  $t\in (0,1]$.  We deduce that  for $k$ 
		sufficiently large  we have that 
		$$
		\Big|\log \frac{P_1(kt,t)}{P_2(kt,t)}\Big|\leq
		 \frac{\delta}{2}, \qquad P_j(kt,t)>0, \quad j=1,2
		$$
		for every $t>0$. 
		Using this and continuity, we get that  on a neighborhood of  the line segment $(kt,t)$, $t\in[k^{-2},k^{-1}]$ we have  $\tilde{F}_{\delta,0}(x,y)>0$ on a subset of $[0,1]\times [0,1]$ of positive measure,  as needed. 
	\end{proof}
	Whenever necessary,  we will  use the following property, which  allows us to swiftly change from $w_{\delta,c,P_1,P_2}(m,n)$ to  $w_{\delta,c,P_1,P_2}(Qm+a,Qn+b)$.  
	\begin{lemma}\label{L:Pab}
		Let $P_1,P_2\in \Z[m,n]$ be  homogeneous polynomials of the same degree.  Then for every $\delta\in (0,1/2)$, $c\in \R$,  and $a,b\in \Z$, we have 
		$$
		\lim_{N\to\infty}\E_{m,n\in [N]} \sup_{Q\in \N}|w_{\delta,c,P_1,P_2}(Qm+a,Qn+b)-w_{\delta,c,P_1,P_2}(m,n)|=0.
		$$
	\end{lemma}
	\begin{proof}\rem{Argument changed quite a bit.}
		By homogeneity and \eqref{E:weight2}, we have $w_{\delta,c,P_1,P_2}(Qm,Qn)=w_{\delta,c,P_1,P_2}(m,n)$, hence it suffices to show that 
		\begin{equation}\label{E:wQ}
		\lim_{N\to\infty}\E_{m,n\in [N]} \sup_{Q\in \N}|w_{\delta,c,P_1,P_2}(Qm+a,Qn+b)-w_{\delta,c,P_1,P_2}(Qm,Qn)|=0.
		\end{equation}
		
		Let $\mathcal R$ be the finite set of real numbers $\alpha$ for which $m-\alpha n$ is a real linear factor of $P_1P_2$. For $H\geq1$ let 
		\begin{equation}\label{E:EH}
			E_H:=\{(m,n)\in\N^2:n\leq H\}\cup
			\bigcup_{\alpha\in\mathcal R}\{(m,n)\in\N^2:|m-\alpha n|\leq H\}, 
		\end{equation}
	and note that 	for fixed $H$, the set $E_H$ has zero natural density in $\N^2$. 
		
		We claim that
		\begin{equation}\label{E:Log}
			\lim_{H\to\infty}
			\sup_{(m,n)\notin E_H,\ Q\in\N}
			\left|\log\left|\frac{P_j(Qm+a,Qn+b)}{P_j(Qm,Qn)}\right|\right|=0 \quad \text{for } j=1,2.
		\end{equation}
		Factoring  the homogeneous polynomials $P_j$, it suffices to establish the claim for their linear and irreducible (over $\R$) quadratic factors. 
		 For a linear factor $L(m,n)=m-\alpha n$ with $\alpha \in \mathcal{R}$, this follows easily from 
		 \begin{multline*}
		\big|\log|L(Qm+a,Qn+b)|-\log|L(Qm,Qn)|\big|=\\
		\big|\log|m-\alpha n+(a-\alpha b)/Q|-\log|m-\alpha n |\big|\ll C_{a,\alpha,b} \, H^{-1},
		\end{multline*}
	where the last estimate follows by combining  the mean value theorem 	 with  the fact that  $|m-\alpha n|>H$ for $(m,n)\notin E_H$.  If the factor has the form $L(m,n)=\alpha n$,  we argue similarly, using that  $n>H$ for $(m,n)\notin E_H$. 
		 Likewise,  irreducible quadratic factors can be put in the form  $R(m,n)=(m-\beta n)^2+\gamma n^2$ for some  $\beta, \gamma\in \R$ with $\gamma>0$. Then the needed vanishing property  follows as in the linear case   using that $|R(m,n)|\geq C (m^2+n^2)$ for some $C=C_{\beta,\gamma}>0$ , which implies that 
		 	for some $C'= C'_{a,b,\beta,\gamma}>0$ we have 
		 $$
	\left|	 \frac{R(Qm+a,Qn+b)}{R(Qm,Qn)}-1\right| \leq \frac{ C'}{m+n}\leq \frac{C'}{H} \qquad \text{for  all }\,  (m,n)\notin E_H.
		 $$

		We also claim that, after increasing $H$ if necessary, the signs of $P_j(Qm+a,Qn+b)$ and $P_j(Qm,Qn)$ agree for $j=1,2$ whenever $(m,n)\notin E_H$. Otherwise, for some real root $\alpha$ of $P_j(x,1)$, the two numbers $m/n$ and $(Qm+a)/(Qn+b)$ lie on opposite sides of $\alpha$, which forces $|m-\alpha n|\leq |a-\alpha b|$. Also,  the possible sign change of the factor $n$ forces  $n\leq |b|$. Hence, taking $H$ to be greater than $|a-\alpha b|$ and $|b|$ accomplishes our goal.
		
		Using the previous two claims,  the uniform continuity of $F_\delta$, the bound
		 $ w_{\delta,c,P_1,P_2}\leq1$,   and that $E_H$ has zero density, we easily deduce that \eqref{E:wQ} holds. 
 \end{proof}
 
	\subsection{Proof of \cref{P:aperiodicQ12}}\label{SS:proof1}
	We  prove  \cref{P:aperiodicQ12}, whose statement we now recall.
	\begin{proposition}\label{P:aperiodicQ12'}
		Let $P_1,P_2$ be homogeneous polynomials of the same degree such that the pair  $(P_1,P_2)$ is good for vanishing of correlations of aperiodic multiplicative functions (see \cref{D:CorVan}) and let $w_{\delta,c,P_1,P_2}$ be as in \eqref{E:weight2}.	Let also $f_1,f_2\colon \N\to \U$ be  multiplicative functions such that   either $f_1$ is $P_1$-aperiodic or $f_2$ is $P_2$-aperiodic (see \cref{D:Paperiodic}). Then  for every $\delta\in (0,1/2)$, $c\in \R$,  $Q\in \N$, and $a,b\in \Z_+$  we have
		\begin{equation}\label{E:wdmnB'}
			\lim_{N\to\infty} \E_{m,n\in [N]} \,
			w_{\delta,c,P_1,P_2}(m,n)\cdot f_1\big(P_1(Qm+a,Qn+b)\big)\cdot \overline{f_2\big(P_2(Qm+a,Qn+b)\big)}=0.
		\end{equation}
	\end{proposition}
	\begin{proof}
		\rem{Quite a few changes in this proof} 
		Using 	 \cref{L:Pab} we get that it suffices to verify   \eqref{E:wdmnB'} with 
		$w_{\delta,c,P_1,P_2}(m,n)$ replaced by  $w_{\delta,c,P_1,P_2}(Qm+a,Qn+b)$.

 	We first reduce to the case where the function $F_\delta$  is restricted on a bounded interval. 
	This will be a consequence of the following claim 
		\begin{equation}\label{E:T0}
		\lim_{T\to\infty} \limsup_{N\to\infty} \E_{m,n\in[N]} \, {\bf 1}_{(Qm+a,Qn+b)\in S} \cdot {\bf 1}_{|L_{Q,a,b,c}(m,n)|>T}=0,
		\end{equation}
		where
			$$
		S:=\{(m,n)\in \N^2\colon P_1(m,n)>0, P_2(m,n)>0\},
		$$
		and	for $m,n\in \N$ with $(Qm+a,Qn+b)\in S$,  we let 
		$$
	L_{Q,a,b,c}(m,n):=\log \frac{P_1(Qm+a,Qn+b)}{P_2(Qm+a,Qn+b)}-c.
	$$
		To prove the claim, let $Z$ be the zero set of $P_1P_2$ in $[0,1]^2$. 
		Since $Z$ is closed and has Lebesgue measure $0$, it is contained in an open set $U_\varepsilon\subset [0,1]^2$ with Lebesgue measure at most $\varepsilon$ and boundary with Lebesgue measure $0$. Then for some $A,B>0$ we have 
		$$
		A\leq |P_j(x,y)|\leq B \qquad \text{for } (x,y)\in [0,1]^2\setminus U_\varepsilon, \, j=1,2. 
		$$
	Since $P_1,P_2$ are homogeneous of the same degree, say $d$, for fixed $Q,a,b$  we have 
		$$
\lim_{N\to\infty}	\max_{m,n\in [N]} \left| \frac{P_j(Qm+a,Qn+b)}{N^d}-
	Q^dP_j\left(\frac{m}{N},\frac{n}{N}\right)\right|=0, \qquad j=1,2.
	$$
		Therefore if 
		$(m/N,n/N)\notin U_\varepsilon$ and $N$ is sufficiently large, we have 
		$$
		\frac{Q^dA}{2}\leq \Big|	\frac{P_j(Qm+a,Qn+b)}{N^d}\Big|\leq 2 Q^dB, \qquad j=1,2.
		$$
		As a consequence, for those $m,n\in [N]$ we have  that $|L_{Q,a,b,c}(m,n)|$ is bounded by a constant that depends on $\varepsilon, Q,a,b,c,P_1,P_2$. 
		  Since the proportion of $(m,n)\in [N]^2$ for which $(m/N,n/N)\in U_\varepsilon$ is at most $2\varepsilon$ for sufficiently large $N$, the claim follows by first letting $N\to\infty$,  then $T\to \infty$, and finally $\varepsilon\to 0$. 
		
		For fixed $T>1$ let $F_{\delta,T}$ be the $(2T)$-periodic extension of $F_\delta \cdot {\bf 1}_{[-T,T]}$ where $F_\delta \colon \R\to [0,1]$ is the continuous function defined  before Lemma~\ref{L:SdeltaGeneral}. We let 
			$$
		w_{\delta,c,T,P_1,P_2}(m,n):=F_{\delta,T} \Big(\log \frac{P_1(m,n)}{P_2(m,n)}-c\Big) \cdot {\bf 1}_{P_j(m,n)>0, j=1,2}, \quad m,n\in \N.
		$$
		Using \eqref{E:T0} and the definition of  $w_{\delta,c,P_1,P_2}$ in \eqref{E:weight2},  we get that in order to prove 
		\eqref{E:wdmnB'}, it suffices to show that for any fixed $T>1$ we have 
		\begin{multline}\label{E:wT}
			\lim_{N\to\infty} \E_{m,n\in [N]} \,
		w_{\delta,c,T,P_1,P_2}(Qm+a,Qn+b)\cdot f_1\big(P_1(Qm+a,Qn+b)\big)\cdot\\ \overline{f_2\big(P_2(Qm+a,Qn+b)\big)}=0.
		\end{multline}
		Since $F_{\delta,T}$ is continuous and $(2T)$-periodic, it can be uniformly approximated by trigonometric polynomials. It is therefore enough to prove \eqref{E:wT} with $F_{\delta,T}(x)$ replaced by $e^{itx}$  for an arbitrary $t\in \R$. Thus, it remains to prove that  for every $t\in \R$ we have
		\begin{equation}\label{E:fk02}
			\lim_{N\to\infty} \E_{m,n\in [N]} \, {\bf 1}_S(Qm+a,Qn+b)\cdot   f_{1,t}(P_1(Qm+a,Qn+b))\cdot \overline{f_{2,t}(P_2(Qm+a,Qn+b))}=0,
		\end{equation}
		where  
		 for $j=1,2$, $f_{j,t}(n):=f_j(n)\cdot n^{it}$, $n\in \N$.
		Since either $f_1$ is $P_1$-aperiodic or $f_2$ is $P_2$-aperiodic, the same holds for the multiplicative functions $f_{1,t}$ and $\overline{f_{2,t}}$.
		By our assumption that the pair $(P_1,P_2)$ is good for vanishing of correlations of aperiodic multiplicative functions,  we  deduce
		that 		\eqref{E:fk02} holds, which completes the proof.
	\end{proof}
	
	\subsection{Proof of part~\eqref{I:anonnegative} of \cref{P:nonnegative}}\label{SS:anonnegative}
	The  next  result is a key ingredient in the proof of part~\eqref{I:anonnegative} of \cref{P:nonnegative}. 
	\begin{lemma}\label{L:keyQQa}
		Let $P_1,P_2\in \Z[m,n]$ be as  in  case \eqref{I:a} of Theorem~\ref{T:DensityRegularQuadratic} and
		let $f\in \CM_p$ be such that  $f\sim \chi \cdot  n^{it}$ for some $t\in \R$ and  Dirichlet character $\chi$. Let also $\delta\in (0,1/2)$ and $c\in \R$ be fixed,    and $(\Phi_K)$, $w_{\delta,c, P_1,P_2}$, $\tilde{w}_{\delta,c,P_1,P_2}$ be  as in     \eqref{E:PhiK},  \eqref{E:weight2}, \eqref{E:wtilde}  respectively. For $Q,N\in \N$  we let 
		\begin{equation}\label{E:ANK}
			L_{\delta, c, N}(f, Q):= \E_{m,n\in [N]} \, \tilde{w}_{\delta,c,P_1,P_2}(m,n)\cdot f(P_1(Qm+1,Qn))\cdot 	\overline{f(P_2(Qm+1,Q n))}
		\end{equation}
		and
		\begin{equation}\label{E:LNfQ}
			\tilde{L}_{\delta,c,  N}(f, Q):= f(Q)\cdot Q^{-it}\cdot L_{\delta, c, N}(f, Q).
		\end{equation}
		Then
		\begin{equation}\label{E:tildeANK}
			\lim_{K\to\infty} \limsup_{N\to\infty}	\max_{Q,Q'\in \Phi_K}|\tilde{L}_{\delta,c, N}(f,Q)-\tilde{L}_{\delta,c, N}(f,Q')|=0.
		\end{equation}
	\end{lemma}
	\begin{proof}
		By assumption, the polynomials $P_1,P_2$ have the form 
		$$
		P_1(m,n)=\ell(m^2+\alpha mn+\beta n^2), \quad P_2(m,n)= \ell'(m+\gamma  n)n,
		$$
		for some $\alpha,\beta,\gamma\in \Z$ and  $\ell, \ell'\in \N,$ and we can assume that $P_1$ is irreducible.
		Factoring out $f(\ell)\cdot\overline{f(\ell')}$ from both $\tilde{L}_{\delta,c, N}(f,Q)$ and $\tilde{L}_{\delta,c, N}(f,Q')$ we reduce to the case where $\ell=\ell'=1$.
		
		For $K\in \N$, let $F_{N}(f,K)$,  $G_{P_1,N}(f,K)$,  be as in \eqref{E:FNfQdef} and
		\eqref{E:GNP}, respectively.

		We extend $f$ to an even function on $\Z$ and apply the concentration estimates  of   Propositions~\ref{P:concentration1} and \ref{P:concentration2General} (note that $f\sim \chi \cdot  n^{it}$ implies  $f\sim_{P_1} \chi \cdot  n^{it}$). Using also \cref{L:lN} 
		we get that  
		$$
		\lim_{K\to\infty}	\limsup_{N\to\infty}\max_{Q\in \Phi_K} \E_{m,n\in [N]}|f(Q(m+\gamma n)+1)- Q^{it}\cdot (m+\gamma n)^{it}\cdot \exp\big(F_{N}(f,K)\big)|=0
		$$
		and
		$$
		\lim_{K\to\infty}	\limsup_{N\to\infty}\max_{Q\in \Phi_K}	\E_{m,n\in [N]}\big| f\big(P_1(Qm+1, Qn)\big)-  Q^{2it} \cdot (P_1(m,n))^{it} \cdot  \exp\big(G_{P_1,N}(f,K)\big)\big|=0.
		$$
		(Since $\ell=1$, we have $P_1(1,0)=1$ so \cref{P:concentration2General} can be used for $c=1$.) 
		We deduce that  if  
		$$
		M_{\delta, c, N}( f):= \E_{m,n\in [N]}\, \tilde{w}_{\delta,c, P_1,P_2}(m,n)\cdot  (P_1(m,n))^{it}\cdot  (m+\gamma n)^{-it}\cdot \overline{f(n)},
		$$
		then 
		$$
		\lim_{K\to\infty}	\limsup_{N\to\infty}\max_{Q\in \Phi_K} |\tilde{L}_{\delta, c, N}(f,Q)- M_{\delta,c, N}(f)\cdot  \exp(G_{P_1,N}(f,K))\cdot \overline{\exp(F_N(f,K))}|=0.
		$$
		Using this identity and the triangle inequality,  we deduce that   \eqref{E:tildeANK} holds.
	\end{proof}
	
 Part~\eqref{I:anonnegative} of \cref{P:nonnegative} is an immediate consequence of the next result
 and the bounded convergence theorem. 
	\begin{proposition}\label{P:nonnegative'}
		Let $P_1,P_2\in \Z[m,n]$ be as  in  the  case \eqref{I:a} of Theorem~\ref{T:DensityRegularQuadratic}, and  $(\Phi_K)$,  $L_{\delta,c, N}(f,Q)$  be  as in     \eqref{E:PhiK},  \eqref{E:ANK}, respectively. Let also $\CM_p$ and $\CA$ be as in \eqref{E:Mp}, \eqref{E:CA} respectively.  
		The following properties hold:
		\begin{enumerate}
			\item\label{I:aQQ1}	If  $f\in \CM_p\setminus \CA$, then for every $\delta\in (0,1/2)$, $c\in \R$,  we have 
			\begin{equation}\label{E:aQ1}
				\lim_{K\to\infty} 	\limsup_{N\to \infty}\Big|\E_{Q\in \Phi_{K}}\, 	L_{\delta, c, N}(f, Q)  \Big|=0.
			\end{equation}
			\item\label{I:aQQ2}	If  $\sigma$ is a bounded Borel measure on $\CM_p$ such that $\sigma(\{1\})>0$, then there exists $c_0=c_0(\sigma,P_1,P_2)\in \R$ such that 
			\begin{equation}\label{E:IntPositive}
				\liminf_{\delta\to 0^+}  	\liminf_{N\to \infty}\inf_{Q\in \N}\, \Re\Big(\int_\CA L_{\delta,c_0,N}(f,Q)\, d\sigma(f)   \Big)> 0.
			\end{equation}
		\end{enumerate}
	\end{proposition}
	\begin{proof}
		We establish \eqref{I:aQQ1}. Suppose that $f\sim \chi\cdot n^{it}$ for some Dirichlet character $\chi$ and $t\in \R$. 
		Let    $\delta\in (0,1/2)$, $c\in \R$,  and  note that for $Q,N\in \N$ we have 
		$$
		L_{\delta,c, N}(f, Q)=\overline{f(Q)\cdot Q^{-it}}\cdot \tilde{L}_{\delta,c,N}(f,Q),
		$$
		where 	$\tilde{L}_{\delta,c, N}(f, Q)$ is as in \eqref{E:LNfQ}. Hence, it suffices to show that 
		$$
		\lim_{K\to\infty} 	\limsup_{N\to \infty}\Big|\E_{Q\in \Phi_{K}}\,\overline{f(Q)\cdot Q^{-it}}\cdot \tilde{L}_{\delta,c, N}(f,Q)   \Big|=0.
		$$	
		For $K\in \N$, let  $Q_K$  be an arbitrary  element of $\Phi_K$. Since by \eqref{E:tildeANK} in Lemma~\ref{L:keyQQa} we have
		$$
		\lim_{K\to\infty} \limsup_{N\to\infty}	\max_{Q\in \Phi_K}|\tilde{L}_{\delta,c,N}(f,Q)-\tilde{L}_{\delta,c,N}(f,Q_K)|=0,
		$$
		it suffices to show that 
		$$
		\lim_{K\to\infty} 	\limsup_{N\to \infty}\Big|\E_{Q\in \Phi_{K}}\,\overline{f(Q)\cdot Q^{-it}}\cdot \tilde{L}_{\delta,c,N}(f,Q_K)   \Big|=0.
		$$	
		Factoring out $\tilde{L}_{\delta,c,N}(f,Q_K)$ and using that $\lim_{N\to\infty}\E_{m,n\in[N]}\, \tilde{w}_{\delta,c,P_1,P_2}(m,n)=1$, we get that it suffices to show that 
		$$
		\lim_{K\to\infty}  \Big|\E_{Q\in \Phi_{K}}\,\overline{f(Q)\cdot Q^{-it}}  \Big|=0.
		$$	
		Since $f\neq n^{it}$, we get that  the multiplicative function $n\mapsto f(n)\cdot n^{-it}$ is not 
		equal to $1$, and the required  vanishing property follows from Lemma~\ref{L:Fol0}.

	\rem{An averaging argument to pick a good value of $c$ used here.}	We establish part~\eqref{I:aQQ2}.
		Let $L_0\in \R$ be as in part~$(i)$ of \cref{L:SdeltaGeneral}, then 
		 $\mu_{\delta,c,P_1,P_2}$ is positive for all $c\geq L_0$. For $c\geq L_0$ let 
		$$
		I(c):=\int_{\CA} e^{ict_f}\,d\sigma(f),
		$$
		where $f(n)=n^{it_f}$ on $\CA$.
		
		We first use an averaging argument to find a  value of $c$ for which the real part of the previous integral is positive. Using Fubini, for $M>0$ we have
		$$
		\frac1M\int_{L_0}^{L_0+M} I(c)\,dc
		=\sigma(\{1\})+
		\int_{\CA\setminus\{1\}} e^{iL_0t_f}
		\frac{e^{iMt_f}-1}{iMt_f}\,d\sigma(f).
		$$
		The second integral tends to $0$ as $M\to\infty$ by the bounded convergence theorem, since $t_f\neq0$ on $\CA\setminus\{1\}$. Since $\sigma(\{1\})>0$, it follows that there exists $c_0\geq L_0$ such that
		\begin{equation}\label{E:eta}
		\eta:=\Re(I(c_0))>0.
		\end{equation}
		We will show that this value of $c_0$ gives the required positivity.

	 For $T>0$  let 
		$$
		\CA_T:=\{f\in\CA: |t_f|\leq T\}.
		$$
		If $f(n)=n^{it}$ with $|t|\leq T$, then
		$$
		L_{\delta,c_0,N}(f,Q)=
		\E_{m,n\in[N]} \widetilde w_{\delta,c_0,P_1,P_2}(m,n)
		(P_1(Qm+1,Qn))^{it}(P_2(Qm+1,Qn))^{-it}.
		$$
		Using that $P_j(Qm,Qn)=Q^2P_j(m,n)$ for  $j=1,2,$
and  arguing as in \cref{L:Pab} (the argument is much simpler in this case), 	we get 
		\begin{equation}\label{E:Lnit}
		\lim_{N\to\infty}\sup_{Q\in\N}\sup_{|t|\leq T}
		\left|L_{\delta,c_0,N}(n^{it},Q)-
		\E_{m,n\in[N]} \widetilde w_{\delta,c_0,P_1,P_2}(m,n)
		\left(\frac{P_1(m,n)}{P_2(m,n)}\right)^{it}\right|=0.
		\end{equation}
		On the support of $\widetilde w_{\delta,c_0,P_1,P_2}$ we have
		$\left|\log\frac{P_1(m,n)}{P_2(m,n)}-c_0\right|\leq\delta$, hence 
		 for $|t|\leq T$ we have 
		\begin{equation}\label{E:Tdelta}
		\left|\left(\frac{P_1(m,n)}{P_2(m,n)}\right)^{it}-e^{ic_0t}\right|
		\leq T\delta \qquad  \text{ whenever } \tilde{w}_{\delta,c_0,P_1,P_2}(m,n)\neq 0.
	\end{equation}
	 Since
		$$
		\lim_{N\to\infty}\E_{m,n\in[N]}\widetilde w_{\delta,c_0,P_1,P_2}(m,n)=1,
		$$
		we deduce from \eqref{E:Lnit} and \eqref{E:Tdelta} that 
		\begin{equation}\label{E:AT}
		\lim_{\delta\to0^+}\limsup_{N\to\infty}
		\sup_{Q\in\N}\sup_{f\in\CA_T}
		\left|L_{\delta,c_0,N}(f,Q)-e^{ic_0t_f}\right|=0.
		\end{equation}
		
	We finally  pass from $\CA_T$ to $\CA$. For  $T$ large enough we have  		
		\begin{equation}\label{E:ATeta}
		\sigma(\CA\setminus\CA_T)\leq \frac{\eta}{8},
			\end{equation}
		hence \eqref{E:eta} gives
		\begin{equation}\label{E:eta'}
	\Re\left(	\int_{\CA_T} e^{ic_0t_f}\,d\sigma(f)\right)\geq \frac{7\eta}{8}.
		\end{equation}
			Combining \eqref{E:ATeta} with  the bound  $\limsup_{N\to\infty}\sup_{Q\in \N}|L_{\delta,c_0,N}(f,Q)|\leq 1$ to 
 control the tail $\CA\setminus \CA_T$, and  \eqref{E:AT} and 
  \eqref{E:eta'}  to  
  bound from below the  integral over $\CA_T$, we get
		$$
		\liminf_{\delta\to0^+}\liminf_{N\to\infty}
		\inf_{Q\in\N}\Re\left(\int_{\CA}L_{\delta,c_0,N}(f,Q)\,d\sigma(f)\right)
		\geq \frac{\eta}{2}>0.
		$$
		This completes the proof.
	\end{proof}

	\subsection{Proof of part~\eqref{I:bnonnegative} of \cref{P:nonnegative}}\label{SS:bnonnegative}
	Our goal in this last  subsection is to prove part~\eqref{I:bnonnegative} of \cref{P:nonnegative}. This will be a consequence of \cref{P:bnonnegative'}  below. 
	An important difference in this case,  compared to  what happened 
	in the proof of  part~\eqref{I:aQQ1} of \cref{P:nonnegative}, is that 
	we will use concentration estimates along 
	the lattice $Q\N+a,Q\N+b$ for suitably chosen $a,b\in\N$, that depend on $Q$,   as given by the next  lemma 
	(the remark after the lemma motivates our  choices).
	\begin{lemma}\label{L:Pi12} 
		Let $r\in \N$ and  $\alpha,\alpha',\beta,\beta'\in \Z$ be such that the polynomials
		$$
		P_1(m,n)= m^2+\alpha mn+\beta n^2, \quad  P_2(m,n)= m^2+\alpha' mn+\beta' n^2,
		$$  
		are distinct  and  irreducible. 
		There exists a finite set of primes $F=F_{r,P_1,P_2}$, containing  the prime divisors of $r!$,  such that the following holds: Let  
		$$
		\CP_1:=\{p\in \P\setminus F\colon \omega_{P_1}(p)=2,  \omega_{P_2}(p)=0 \}
		$$
		$$
		\CP_2:=\{p\in \P\setminus F\colon \omega_{P_1}(p)=0,  \omega_{P_2}(p)=2 \},
		$$ 
		where $\omega_P(p)$ is as in \eqref{D:omega}, and for  $K\in \N$ let 
		\begin{equation}\label{E:QK}
			Q=Q_{K}:= \prod_{p\leq K}p^{2K}.
		\end{equation}
		Let also $l_{j,p}\in[1,3K/2]$, $p\in \P\cap [K]$, $j=1,2$, and 
		$$
		Q_j=Q_{j,K}:= \prod_{p\in \CP_j\cap [K]}p^{l_{j,p}}, \qquad j=1,2.
		$$  
		Then 
		there exist $a=a_{r,K,Q_1,Q_2},b=b_{r,K,Q_1,Q_2}\in [Q]$,   such that  for $j=1,2$ we have
		\begin{equation}\label{E:Pi1} 
			P_j(a,b)\equiv 1 \! \! \! \pmod{r!}
		\end{equation}
		and 
		\begin{equation}\label{E:Pi2} 
			(P_j(a,b), Q)=Q_{j}.
		\end{equation}
	\end{lemma}
	\begin{remark}
		Let us explain a bit the motivation behind our choice of $a,b$. Suppose that for $j=1,2$ we have that  $f\sim_{P_j} \chi_j$ for some Dirichlet characters $\chi_j$ with period $q$ and let $Q_{j}:=\prod_{p\in \CP_j\cap [K],\,  p\notin F}p^{l_{j,p}}$. If $a,b$ satisfy the properties of the lemma  and $K$ is 
		sufficiently large, we get by \cref{P:concentration2General}  that for $j=1,2$ and  $Q=Q_K$ (we use  that   $Q_j=(P_j(a,b), Q)$, $j=1,2$, and  $q\prod_{p\leq K}p$ divides $Q/Q_{j}$) the values of  $f(P_j(Qm+a,Qn+b))$ are concentrated at   
		$f(Q_{j})\cdot \chi_j(P_j(a,b)/Q_j) \cdot \exp(G_{P_j,N}(f,K))$, and since   $(Q_j,q)=1$ (hence $|\chi_j(Q_j)|=1$), and $P_j(a,b)\equiv 1\pmod{q}$ (which holds if $r\geq q$),  we get concentration at 
		$f(Q_j)\cdot \overline{\chi_j(Q_{j})} \cdot \exp(G_{P_j,N}(f,K))$. This will give that if for $j=1$ or $j=2$ we have $f(p)\neq \chi_j(p)$ for some $p\in \CP_j$, then   a suitable multiplicative average will be  $0$.  On the other hand, if for $j=1,2$ we have 
		$f(p)=\chi_j(p)$ for all $p\in \CP_j$, then the oscillatory factors $\exp(G_{P_j,N}(f,K))$, which  appear in our analysis, will  ``cancel out'', and  
		we will get  concentration  at $f(Q_{j})\cdot \overline{\chi_j(Q_{j})}=1$.  Combining the above we will get that part~\eqref{I:bnonnegative} of \cref{P:nonnegative} holds. 
	\end{remark}
	\begin{proof}
		First, we  define the finite set of primes $F$ that we want to avoid for reasons that will become clear later in the argument. We let\footnote{The expression  $P_2(\beta-\beta',\alpha'-\alpha)$ is not needed since it is equal to $P_1(\beta'-\beta,\alpha-\alpha')$.}
		\begin{equation}\label{E:A}
			A:=	 2r!\beta\beta'(\alpha^2-4\beta)(\alpha'^2-4\beta') P_1(\beta'-\beta,\alpha-\alpha')
		\end{equation}
		and
		$$
		F:=\{p\in \P\colon p\mid A\}.
		$$
		Let us first explain why the set $F$ is finite, which is equivalent to showing that  $A\neq 0$. 
		Since the polynomials $P_1,P_2$ are irreducible we have 
		$\beta\beta'(\alpha^2-4\beta)(\alpha'^2-4\beta')\neq 0$.
		Since the polynomials $P_1,P_2$ are distinct, we have that either $\alpha\neq \alpha'$ or $\beta\neq \beta'$. Using this and the fact that  $P_1,P_2$ are irreducible, we deduce that  $P_j(\beta'-\beta,\alpha-\alpha')\neq 0$ for $j=1,2$. Hence, $A\neq 0$.

		For this set $F$ we proceed to show that if $K\in \N$ and  $l_{j,p}\in \N$, $j=1,2$, $p\in \P\cap [K]$, are given, then a choice of $a,b\in \N$ can be made that satisfies the asserted properties. 
		
		To this end, first note  that if $a\equiv 1 \pmod{r!}$ and $b\equiv 0 \pmod{r!}$, then 
		$P_j(a,b) \equiv1 \pmod{r!}$ for $j=1,2$. 	Using this and the Chinese remainder theorem (and since the prime divisors of $r!$ are contained in $F$), we see that  it is sufficient  to show that the following three properties hold:
		\begin{enumerate}
			\item 	 For every $p\in \CP_1$,  there exist  $a,b\in [p^{l_{1,p}+1}]$, depending on $p$, such that
			$$
			p^{l_{1,p}}\mid \mid  P_1(a,b)\quad  \text {and} \quad p\nmid P_2(a,b). 
			$$
			
			\item 	 For every $p\in \CP_2$,  there exist $a,b\in [p^{l_{2,p}+1}]$, depending on $p$,  such that 
			$$
			p^{l_{2,p}}\mid \mid  P_2(a,b) \quad  \text {and} \quad p\nmid P_1(a,b). 
			$$
			
			\item  If  $\tilde{Q}:=\prod_{p\leq K,\, p\not\in (\CP_1\cup \CP_2\cup F)}p^{l_{1,p}+l_{2,p}}$, then  there exist $a,b\in [\tilde{Q}]$, depending on $\tilde{Q}$,  such that 
			$$
			P_j(a,b)\equiv 1\! \! \pmod{\tilde{Q}} \quad  \text{for  } j=1,2. 
			$$
		\end{enumerate}
		Note that the third property is satisfied if we choose $a\equiv\!\! 1\pmod{\tilde{Q}}$ and $b\equiv \!\! 0\pmod{\tilde{Q}}$ and implies that   $p\nmid P_j(a,b)$ for all $p\leq K$ that are not in $ \CP_1\cup \CP_2$. 
		
		We prove the first property;   the second one can be proved similarly. Let $p\in \CP_1$. We take $b=1$ and choose $a$ as follows:   Since  for  $p\in \CP_1$ we have $\omega_{P_1}(p)=2$,  there exists $a\in \N$ such that $p\mid P_1(a,1)$. Since $p\in \CP_1$ implies $p\not\in F$, we have  $p\nmid 2(\alpha^2-4\beta)$, hence $p\nmid P_1'(a,1)$.
		Using Hensel's lemma, we get that  there exists $a\in \N$, $a\leq p^{l_{1,p}+1}$,  so that $p^{l_{1,p}}\mid \mid P_1(a,1)$. It remains to show that $p\nmid P_2(a,1)$. Equivalently, it suffices to show that if
		\begin{equation}\label{E:P12}
			p\mid P_1(a,1)=a^2+\alpha a+\beta \quad \text{and} \quad 
			p\mid P_2(a,1)=a^2+\alpha'a+\beta',
		\end{equation}
		then $p\in F$. We  will  prove this  by showing that $p\mid A$  where $A$ is as in \eqref{E:A}. 
		Subtracting the two relations  in \eqref{E:P12}
		gives 
		\begin{equation}\label{E:aa'}
			p\mid (\alpha-\alpha')a+\beta-\beta',
		\end{equation}
		while  multiplying the second relation by $\beta$ and the first by $\beta'$ and subtracting, we get that 
		$
		p\mid (\beta-\beta')a^2+(\alpha'\beta-\alpha\beta')a.
		$ 
		So either 
		$p\in F$, in which case we are done,  
		or   $p\nmid a$ (otherwise \eqref{E:P12} gives $p\mid \beta$, which cannot happen for $p\notin F$), in which case we have   
		\begin{equation}\label{E:bb'}
			p\mid (\beta-\beta')a+(\alpha'\beta-\alpha\beta').
		\end{equation}
		Multiplying  \eqref{E:aa'} by $\beta-\beta'$ and \eqref{E:bb'} by $\alpha-\alpha'$ and subtracting,  we deduce that
		$$
		p\mid (\beta-\beta')^2-(\alpha'\beta-\alpha\beta')(\alpha-\alpha')= P_1(\beta'-\beta,\alpha-\alpha').
		$$
		Hence, 		
		$p\mid A$. This completes the proof. 
	\end{proof}
	Using  \cref{L:Pi12} we get the following consequence of
	\cref{P:concentration2General} that will be crucial for our purposes. 
	\begin{proposition}\label{P:PQ12}  Let $r\in \N$ and  $F,P_1,P_2,\CP_1,\CP_2$  be as in \cref{L:Pi12}.   For $j=1,2$ and $K\in \N$ let  
		\begin{equation}\label{E:PhijK}
			\Phi_{j,K}:=\big\{ \prod_{p\in \CP_j\cap [K]} p^{l_p}\colon K< l_p\leq  3K/2\big\},
		\end{equation}
		(the sets $\Phi_{j,K}$ also depend  on $r,P_1,P_2$ but we suppress this dependence)	and 
		$ Q_{K}:= \prod_{ p\leq K}p^{2K}$. Then for all $Q_j\in \Phi_{j,K}$, $j=1,2$,  there exist 
		$a=a_{r,K,Q_1,Q_2}\in [Q_K],b=b_{r,K,Q_1,Q_2}\in [Q_K]$,  such that the following holds: 
		If 
		$f\colon \Z\to\U$ is an even multiplicative function such that for $j=1,2$ we have  $f\sim_{P_j} \chi_j\cdot n^{it}$ for some $t\in \R$ and Dirichlet characters $\chi_j$ with period $q$, then for all 
		sufficiently large  $r$, depending only on $q$, we have 
		\begin{multline}\label{E:Pjab0}
			\lim_{K\to\infty}	\limsup_{N\to\infty} \max_{Q_1\in \Phi_{1,K}, Q_2\in \Phi_{2,K} }	\E_{m,n\in [N]}\, \big|f\big(P_j(Q_Km+a_{r,K,Q_1,Q_2},Q_Kn+b_{r,K,Q_1,Q_2})\big)- \\ 
			\tilde{f}_j(Q_j)\cdot Q_K^{2it}\cdot  
			(P_j(m,n))^{it}\cdot   \exp\big( G_{P_j,N}(f,K)\big)\big|=0,
		\end{multline}
		where $\tilde{f}_j(n):=f(n)\cdot \overline{\chi_j(n)}\cdot n^{-it}$, $j=1,2$, 
		and  $G_{P_j,N}$, $j=1,2$,  are as in   \eqref{E:GNP}.
	\end{proposition}
	\begin{remark}
		Note that the choice of $a_{r,K,Q_1,Q_2}$ and $b_{r,K,Q_1,Q_2}$ is independent of the multiplicative function $f$.
	\end{remark}
	\begin{proof}
		By   \cref{L:Pi12} we  get that for $Q_j\in \Phi_{j,K}$, $j=1,2$, there exist  $a=a_{r,K,Q_1,Q_2},b=b_{r,K,Q_1,Q_2}\in [Q_K]$ that satisfy \eqref{E:Pi1} and \eqref{E:Pi2}. 
		We apply \cref{P:concentration2General} for $A_K:=Q_K$ and $c_j:=(P_j(a,b) , Q_K)=Q_j$ for $j=1,2$,
		and note that for  sufficiently large $K$, depending only on $q$,  	we have that $q$ and $\prod_{p\leq K}p$ divide $Q_K/Q_j$ for $j=1,2$. We get  that for $j=1,2$ we have 
		\begin{multline*}
			\lim_{K\to\infty}	\limsup_{N\to\infty}   \max_{Q_1\in \Phi_{1,K}, Q_2\in \Phi_{2,K} }		\E_{m,n\in [N]}\, \big|f\big(P_j(Q_Km+a,Q_Kn+b)\big)- \\ 
			f(Q_j)\cdot \chi_j(P_j(a,b)/Q_j)\cdot Q_K^{2it}\cdot  
			(P_j(m,n)/Q_j)^{it}\cdot   \exp\big( G_{P_j,N}(f,K)\big)\big|=0.
		\end{multline*}
		Now let $r$ be sufficiently large so that $q\mid r!$. 
		By construction, for $j=1,2$ the set $\CP_j$ does not contain prime divisors   of $q$ (since it avoids $F$, which contains all prime divisors of $r!$ and $q\mid r!$), hence $(Q_j,q)=1$, which implies $|\chi_j(Q_j)|=1$. It follows that  for $j=1,2$ we have  $\chi_j(P_j(a,b)/Q_j)= \chi_j(P_j(a,b))\cdot \overline{\chi_j(Q_j)}=\overline{\chi_j(Q_j)}$, where the last identity follows because 
		$P_j(a,b)\equiv 1 \pmod{q}$, which implies  $\chi_j(P_j(a,b))=1$.  
	\end{proof}
	We are now ready to prove part~\eqref{I:bnonnegative} of \cref{P:nonnegative}.
	\begin{proposition}\label{P:bnonnegative'}
		Let $r\in \N$ and  $F,P_1,P_2,\CP_1,\CP_2$  be as in \cref{L:Pi12} and $(\Phi_{j,K})$, $j=1,2$,  be as in \cref{P:PQ12}. 
		Let $f\in \CM_{p,P_1,P_2}$ be   such that for $j=1,2,$ we have  $f\sim_{P_j} \chi_j\cdot n^{it}$ for some $t\in \R$ and  Dirichlet characters $\chi_j$ with period $q$. Let also $\delta\in (0,1/2)$ be fixed and 
		$w_{\delta,0,P_1,P_2}$ and  $\tilde{w}_{\delta,0,P_1,P_2}$ be as in \eqref{E:weight2} and \eqref{E:wtilde} respectively. For $K\in \N$  we let  $Q_{K}$, $a=a_{r,K,Q_1,Q_2}\in [Q_K],b=b_{r,K,Q_1,Q_2}\in [Q_K]$ be as in \cref{P:PQ12}, where $Q_j\in \Phi_{j,K}$, $j=1,2$.  Finally,  let 
		\begin{multline}\label{E:ANKb'}
			L_{\delta,0, N}(f, r,K,Q_1,Q_2):= \E_{m,n\in [N]} \, \tilde{w}_{\delta,0,P_1,P_2}(m,n)\cdot \\ 
			f(P_1(Q_Km+a_{r,K,Q_1,Q_2},Q_Kn+b_{r,K,Q_1,Q_2}))\cdot 	\overline{f(P_2(Q_Km+a_{r,K,Q_1,Q_2}, Q_Kn+b_{r,K,Q_1,Q_2}))}.
		\end{multline}
		If $r$ is  sufficiently large, depending only on $q$,  the following properties hold: 
		\begin{enumerate}
			\item\label{I:QQ1}	If  $f(p)\neq \chi_j(p)\cdot p^{it}$ for some $p\in \CP_j$  and some $j\in \{1,2\}$, then  
			\begin{equation}\label{E:Q120}
				\lim_{K\to\infty} 	\limsup_{N\to \infty}\Big|\E_{Q_1\in \Phi_{1,K}, Q_2\in \Phi_{2,K}}\, L_{\delta,0,N}(f,r,K,Q_1,Q_2)   \Big|=0.
			\end{equation}

			\item\label{I:QQ2}	If  $f(p)= \chi_j(p)\cdot p^{it}$ for every  $p\in \CP_j$ and $j=1,2$,  then 
			\begin{equation}\label{E:tildeANKb'}
				\lim_{\delta\to 0^+}	\limsup_{K\to\infty} \limsup_{N\to\infty}	\max_{Q_1\in \Phi_{1,K}, Q_2\in \Phi_{2,K}}|L_{\delta,0,N}(f,r,K,Q_1,Q_2)-1|=0.
			\end{equation}
			
			\item \label{I:QQ3} Without additional assumptions on $f\in \CM_{p,P_1,P_2}$  we have  
			\begin{equation}\label{E:Q12positive}
				\liminf_{\delta\to 0^+} \liminf_{K\to\infty} 	\liminf_{N\to \infty}\E_{Q_1\in \Phi_{1,K}, Q_2\in \Phi_{2,K}}\, \Re\big(L_{\delta,0,N}(f,r,K,Q_1,Q_2)   \big)\geq 0.
			\end{equation}
		\end{enumerate}
	\end{proposition}
	\begin{proof} 
		We start by  using \cref{P:PQ12}. Since $\lim_{N\to\infty} \E_{m,n\in [N]}\, \tilde{w}_{\delta,0,P_1,P_2}(m,n)=1$,	we deduce that  if  
		\begin{multline}\label{E:MdKNf}
			M_{\delta, K, N}( f):= \exp(G_{P_1,N}(f,K))\cdot \overline{ \exp(G_{P_2,N}(f,K))}\cdot  \\ 
			\E_{m,n\in [N]}\, \tilde{w}_{\delta,0,P_1,P_2}(m,n)\cdot  (P_1(m,n))^{it}\cdot  (P_2(m,n))^{-it},
		\end{multline}
		where 	$G_{P_j,N}$, $j=1,2$,  are as in   \eqref{E:GNP},
		then for all  sufficiently large  $r$, depending only on $q$, for every $\delta\in (0,1/2)$ we have
		\begin{equation}\label{E:LQ120}
			\lim_{K\to\infty}	\limsup_{N\to\infty}\max_{Q_1\in \Phi_{1,K},  Q_2\in \Phi_{2,K}}|L_{\delta,0, N}(f,r,K,Q_1,Q_2)- \tilde{f}_1(Q_1) \cdot \overline{\tilde{f}_2(Q_2)}\cdot M_{\delta,  K,N}(f)|=0, 
		\end{equation}
		where $\tilde{f}_j(n):=f(n)\cdot \overline{\chi_j(n)}\cdot n^{-it}$, $j=1,2$.
		
		We establish~\eqref{I:QQ1}.
		Suppose that  for some $j\in \{1,2\}$ we have $f(p)\neq \chi_j(p)\cdot p^{it}$ for some $p\in \CP_j$.
		Say that this happens for $j=1$ (the argument is similar for $j=2$) and that $f(p_0)\neq \chi_j(p_0)\cdot p_0^{it}$ for some $p_0\in \CP_1$. 
		It follows from \eqref{E:LQ120}, that in order to establish \eqref{E:Q120} it is sufficient  to show that 
		$$
		\lim_{K\to\infty}\limsup_{N\to\infty}\big|\E_{Q_1\in \Phi_{1,K}, Q_2\in \Phi_{2,K}}\,  \tilde{f}_1(Q_1) \cdot \overline{\tilde{f}_2(Q_2)} \cdot M_{\delta, K,N}(f)\big|=0.
		$$
		Factoring out $M_{\delta, K,N}(f)$, which does not depend on $Q_1$ or $Q_2$, we see that  it suffices to show that 
		$$
		\lim_{K\to\infty} 	\E_{Q_1\in \Phi_{1,K}}\,  \tilde{f}_1(Q_1)=0.
		$$	
		Since  $\lim_{K\to\infty}|\Phi_{1,K} \cap (\Phi_{1,K}\cdot p_0)|/|\Phi_{1,K}|= 1$  and $\tilde{f_1}(p_0)\neq 1$, this follows exactly as in the proof of  \cite[Lemma~3.2]{FKM23}.

		We establish \eqref{I:QQ2}.
		Suppose that  $f(p)= \chi_j(p)\cdot p^{it}$ for every  $p\in \CP_j$ and $j=1,2$.  	 Note that by \eqref{E:PhijK}, for $j=1,2$, if $Q_j\in \Phi_{j,K}$ for some $K\in \N$,  then the prime factors of $Q_j$  belong to $\CP_j$. 
		It follows  that   
		\begin{equation}\label{E:Qj}
			\tilde{f}_j(Q_j)=1    \quad  \text{ for all  } \, Q_j\in \Phi_{j,K},\,  K\in \N,  \, j=1,2. 
		\end{equation}
		Note also that for all but finitely many $p\not\in \CP_1\cup \CP_2$, depending only on $P_1,P_2,r$,  we have that $p\in   \CP_{P_1}\cap \CP_{P_2}$ (where $\CP_P$ is as in \eqref{E:CPP'})  and $\omega_{P_1}(p)=\omega_{P_2}(p)=2$.
		It follows from \eqref{E:GNP}  that for all sufficiently large  $K$,  depending only on $P_1,P_2,r$,   					we have
		$$
		G_{P_j,N}(f,K)= \sum_{\substack{ K< p\leq N,\\ p\in \CP_{P_1}\cap \CP_{P_2}}}\, \frac{2}{p} \,(f(p)\cdot \overline{\chi_j(p)}\cdot p^{-it} -1), \quad j=1,2.
		$$
		Using this, and remembering  that by \cref{L:tunique} for all but finitely many  $p\in \CP_{P_1}\cap \CP_{P_2}$,  depending only on $P_1,P_2$, we have $\chi_1(p)=\chi_2(p)$, we deduce that for all sufficiently large $K$,  depending only on $P_1,P_2, r$,   we have 		$G_{P_1,N}(f,K)=G_{P_2,N}(f,K)$. Hence,
		\begin{multline*}
			| \exp\big( G_{P_1,N}(f,K)\big)\cdot   \overline{ \exp\big( G_{P_2,N}(f,K)\big)}-1|=
			| \exp\big( 2\, \Re(G_{P_1,N}(f,K))\big)-1|\leq
			\\ 2\, |\Re(G_{P_1,N}(f,K))|\leq 2\, \D_{P_1}(f, \chi_1\cdot n^{it}; K,N)^2\leq  2\, \D_{P_1}(f, \chi_1\cdot n^{it}; K,N),
		\end{multline*}					
		where we used that $\Re(G_{P_1,N}(f,K))\leq 0$,  $|e^x-1|\leq |x|$ for $x\leq 0$, and $\D_{P_1}(f, \chi_1\cdot n^{it}; K,N)\leq 1$ for $K$ sufficiently large. 
		We deduce from this and \eqref{E:MdKNf} that  
		$$
		|M_{\delta, K,  N}( f)-  \E_{m,n\in [N]}\, \tilde{w}_{\delta,0,P_1,P_2}(m,n)\cdot  (P_1(m,n))^{it}\cdot   (P_2(m,n))^{-it}|\ll   \D_{P_1}(f, \chi_1\cdot n^{it}; K,N).
		$$
		Moreover, from  \eqref{E:wtilde} and the definition of $w_{\delta, 0,P_1,P_2}$ in \eqref{E:weight2}, it follows  that \rem{crucial change here} 
		$$
		\left|\log\frac{P_1(m,n)}{P_2(m,n)} \right|\leq \delta \quad \text{whenever } \tilde{w}_{\delta,0,P_1,P_2}(m,n)\neq 0, 
		$$
		hence  using that $|e^{ix}-1|\leq |x|$ for all $x\in \R$	we 	get 
		$$
		|(P_1(m,n))^{it}\cdot (P_2(m,n))^{-it}-1|\leq |t| \delta \quad \text{whenever } \tilde{w}_{\delta,0,P_1,P_2}(m,n)\neq 0.
		$$
		Using this and since $\lim_{N\to\infty} \E_{m,n\in [N]}\, \tilde{w}_{\delta,0, P_1,P_2}(m,n)=1$, 	we get  that  for all sufficiently large $K$,  depending only on $P_1,P_2, r$,   	and for all $N\geq K$, we have 
		\begin{equation}\label{E:MdNf}
			|M_{\delta,   K, N}( f)- 1|\ll_t \delta +\D_{P_1}(f, \chi_1\cdot n^{it}; K,N). 
		\end{equation}
		Combining  \eqref{E:LQ120},  \eqref{E:Qj}, \eqref{E:MdNf}, and since  our assumption  $f\sim_{P_1} \chi_1\cdot n^{it}$ implies that $\lim_{K\to\infty}\D_{P_1}(f, \chi_1\cdot n^{it}; K,+\infty)=0$,  we get that \eqref{E:tildeANKb'} holds, completing the proof.
		
		We establish~\eqref{I:QQ3}. If we are in the case \eqref{I:QQ1}, then the limit is $0$. If not, then 
		case \eqref{I:QQ2} applies and we get that the limit is at least  $1$. This completes the proof. 
	\end{proof}
	
	\subsection{Proof of \cref{T:levelsetsparam}}\label{SS:levelset}  
	For notational simplicity we give the proof for $r=1$, the argument is similar for general $r\in \N$. 
	We let $F\colon \S^1\to [0,1]$ be a smooth function that is equal to 
	$1$ on an open  subarc $I'$ of $I$  that contains $1$ and zero outside $I$, and let  
	$\Phi=(\Phi_K)_{K\in \N}$ be a multiplicative  F\o lner sequence in $\N$ along which the limits 
	\begin{equation}\label{E:Crsdef}
		C(r,s):=\lim_{K\to\infty}\E_{k\in \Phi_K} F(f_1(kr))\cdot F(f_1(ks))
	\end{equation}
	exist for every $r,s\in \N$. 
	Since $F$ is supported on a subset of $I$, it suffices   to show that there exist $m,n\in\N$ such that the integers $P_1(m,n), P_2(m,n)$ are distinct, positive,  and 
	$$
	\lim_{K\to\infty}\E_{k\in \Phi_K} F(f_1(kP_1(m,n)))\cdot F(f_1(kP_2(m,n)))>0.
	$$ 
	Arguing  as in \cite[Section~10.2]{FH17} we get that there exists a finite (positive) Borel measure $\sigma$ on $\CM$ such that 
	\begin{equation}\label{E:Crs}
		C(r,s)=	\int_\CM f(r)\cdot \overline{f(s)}\, d\sigma(f)
	\end{equation}
	for every $r,s\in \N$. Moreover, since 
	$$
	\lim_{K\to\infty}\E_{k\in \Phi_K} F(f_1(kr))=\int F\, dm_{\S^1}=a>0,
	$$
	arguing as in \cite[Section~10.2]{FH17} we get   that   $\sigma(\{1\})\geq a^2>0.$ 
	
	We claim that the measure $\sigma$ is supported
	on the set $\CF:=\{f_1^j,j\in\Z\}$. Note first that since the function $F\colon \S^1\to [0,1]$ is smooth (in fact, $F\in C^2(\S^1)$ suffices), 	it can be written as 
	$$
	F(z)=\sum_{l=-\infty}^{+\infty} c_l\,  z^l
	$$	
	for some $c_l\in \C$, $l\in \Z$, with $\sum_{l\in\Z}|c_l|<\infty$.
	Hence, 
	$$
	F(f_1(n))=\sum_{l=-\infty}^{+\infty} c_l\,  (f_1(n))^l, \quad n\in\N. 
	$$
	After possibly rearranging and combining terms and redefining the sequence $c_l$, $l\in \Z$,  we can assume that for those $l\in \Z$ for which $c_l\neq 0$ the multiplicative functions $f_1^l$ are distinct.\footnote{The new $c_l$ in this representation are either $0$ or sum of disjoint ``groups'' of $c_l$'s.} Using that $\sum_{l\in\Z}|c_l|<\infty$ and $F=\bar{F}$,  we easily deduce that 	
	$$
	C(r,s)=\sum_{l,l'\in \Z} c_l\cdot c_{l'}\cdot f_1^l(r)\cdot \overline{f_1^{l'}(s)} \lim_{K\to\infty}\E_{k\in \Phi_K}f_1^l(k)\cdot \overline{f_1^{l'}(k)}
	$$
	for every $r,s\in \N$. 	We have assumed that if $l,l'\in \Z$ are distinct and  
	satisfy $c_l\cdot c_l'\neq 0$, then the  multiplicative function  $f_1^l\cdot \overline{f_1^{l'}}$ is non-trivial (i.e. not equal to $1$), hence we get by \cref{L:Fol0} that for those $l,l'\in \Z$ we have 
	$\E_{k\in \Phi_K}f_1^l(k)\cdot \overline{f_1^{l'}(k)}=0$. 
	It follows that 
	\begin{equation}\label{E:Crs''}
		C(r,s)=\sum_{l\in \Z} |c_l|^2\cdot f_1^l(r)\cdot \overline{f_1^{l}(s)}=\int_\CM f(r)\cdot \overline{f(s)}\, d\sigma'(f)
	\end{equation}
	for every $r,s\in \N$, 	where 
	\begin{equation}\label{E:Crs'}
		\sigma':=\sum_{l\in \Z} |c_l|^2\cdot \delta_{f_1^l}.
	\end{equation}
	Comparing the identities \eqref{E:Crs} and \eqref{E:Crs''},
	and using the uniqueness of the Fourier coefficients of finite Borel  measures on $\CM$, we deduce that $\sigma=\sigma'$. Thus,  the measure $\sigma$ is supported on $\CF$, which proves the claim.

	It follows now from the third remark after \cref{T:MainMulti'} that there exist $k,m,n\in\N$ such that $P_1(m,n), P_2(m,n)$ are distinct, positive, and satisfy
	$$
	C(P_1(m,n),P_2(m,n))>0,
	$$
	where we used the integral representation of $C(r,s)$ given in \eqref{E:Crs}. Recalling the defining relation for  $C(r,s)$ in \eqref{E:Crsdef} we get that $f_1(kP_1(m,n))\in I$ and  $f_1(kP_2(m,n))\in I$, which  completes the proof. 
	
	\section{Necessity in \cref{C:conj1}}\label{S:Conj1}		\subsection{Proof of \cref{P:partialconverse1}}
	The proof of 	\cref{P:partialconverse1} follows by combining the following lemmas. 
	\begin{lemma}\label{lemma_PRandquadraticresidues}
		Let $a,b,c\in\Z$ and assume that there exists a prime $p$ such that none of the numbers $ac,bc,(a+b)c$ is a quadratic residue modulo $p$. Then \eqref{E:abc} is not partition regular with respect to $x,y$.
	\end{lemma}
	\begin{proof}
		We use the ``Rado $p$-coloring'', defined by $\chi(p^nr)=r\bmod p$ for any $n\geq0$ and $r$ co-prime to $p$.
		Suppose, for the sake of a contradiction, that $(x,y,z)$ is a solution to \eqref{E:abc} with $\chi(x)=\chi(y)$.
		From the assumptions, none of the coefficients $a,b,c$ is divisible by $p$, so if both $x$ and $y$ are multiples of $p$, then also $z$ is a multiple of $p$ and hence $(x/p,y/p,z/p)$ is another solution. Moreover $\chi(x/p)=\chi(x)$ so we can assume without loss of generality that not both of $x,y$ are multiples of $p$.
		
		If $p|x$, then \eqref{E:abc} reduces mod $p$ to $by^2\equiv cz^2$, which is equivalent to $bc\equiv(cz/y)^2$. This contradicts the assumption that $bc$ is not a quadratic residue mod $p$. The same argument shows that we cannot have $p|y$.
		
		If neither $x,y$ are divisible by $p$, then $x\equiv y\bmod p$ and hence \eqref{E:abc} reduces mod $p$ to $(a+b)x^2\equiv cz^2$, which is equivalent to $(a+b)c\equiv(cz/x)^2$, and contradicts the assumption that $(a+b)c$ is not a quadratic residue mod $p$.
	\end{proof}
	We combine this lemma with the following elementary fact.
	
	\begin{lemma}\label{lemma_quadraticreciprocity}
		Let $F_1,F_2$ be disjoint finite subsets of $\P\cup\{-1\}$.
		Then there exists a prime $p$ such that every element of $F_1$ is a quadratic residue mod $p$, and every element of $F_2$ is a quadratic nonresidue mod $p$.
	\end{lemma}
	
	\begin{proof}
		We separate the proof into two cases, according to whether $-1\in F_1$ or $-1\notin F_1$. 
		In the case $-1\in F_1$, we need $p\equiv1\bmod4$.
		For such $p$, and any odd prime $q$, quadratic reciprocity implies that $q$ is a quadratic residue mod $p$ if and only if $p$ is a quadratic residue mod $q$.
		Moreover, $2$ is a quadratic residue mod $p$ if and only if $p\equiv1\bmod8$.
		For each odd prime $q\in F_1$ let $i_q$ be a quadratic residue mod $q$, and for each odd prime $q\in F_2$ let $i_q$ be a quadratic nonresidue mod $q$.
		If $2\in F_1$ let $i=1$, and if $2\notin F_1$ let $i=5$.
		Using the Chinese remainder theorem and the Dirichlet theorem we can find a prime $p$ satisfying $p\equiv i\bmod8$ and $p\equiv i_q\bmod q$ for every odd prime $q\in F_1\cup F_2$; such $p$ will satisfy the desired conclusion.
		
		Next, assume that $-1\notin F_1$.
		In this case we will take $p\equiv3\bmod4$, so that in particular $-1$ is a quadratic nonresidue mod $p$. 
		For each odd prime $q\in F_1\cup F_2$ we let 
		$$i_q=\begin{cases}
			\text{a quadratic residue mod }q&\text{ if }q\equiv1\bmod 4\text{ and }q\in F_1\\
			\text{a quadratic nonresidue mod }q&\text{ if }q\equiv1\bmod 4\text{ and }q\in F_2\\
			\text{a quadratic nonresidue mod }q&\text{ if }q\equiv3\bmod 4\text{ and }q\in F_1\\
			\text{a quadratic residue mod }q&\text{ if }q\equiv3\bmod 4\text{ and }q\in F_2
		\end{cases}$$
		and furthermore assume that $i_q$ is chosen so that $(i_q,q)=1$.
		Also let $i=7$ if $2\in F_1$ and $i=3$ otherwise.
		Using the Chinese remainder theorem and the Dirichlet theorem we can find a prime $p$ satisfying $p\equiv i\bmod8$ and $p\equiv i_q\bmod q$ for every odd prime $q\in F_1\cup F_2$.
		In view of quadratic reciprocity, such $p$ will satisfy the desired conclusion.
	\end{proof}
	\begin{lemma}
		Let $A,B,C\in\Z$ and assume that none of $A,B,C,ABC$ is a perfect square. 
		Then there exists a prime $p$ such that none of $A,B,C$ is a quadratic residue modulo $p$.
	\end{lemma}
	\begin{proof}
		Dividing $A,B,C$ by any square factors does not change either the hypothesis or the conclusion of the lemma, so we may assume that $A,B,C$ are all squarefree numbers.
		If $A=B=C=-1$, then take $p=3$.
		Otherwise, the assumption that $ABC$ is not a perfect square  and $A,B, C$ are square-free implies that there exists a prime $p_0$ that either divides all of $A,B,C$ or divides exactly one of them.
		
		In the first case, apply \cref{lemma_quadraticreciprocity} with $F_2=\{p_0\}$ and $F_1$ the set of all primes that divide $ABC$ other than $p_0$, together with $-1$.
		
		Suppose now that we are in the second case, where $p_0$ divides (say) $A$ but not $B$ nor $C$.
		We further separate into two cases: if there exists a prime $p_1$ that divides both $B$ and $C$ (but not $A$), apply \cref{lemma_quadraticreciprocity} with $F_2=\{p_0,p_1\}$ and $F_1$ the set of all primes that divide $ABC$ other than $p_0,p_1$, together with $-1$.
		
		In the final case there exist primes $p_1$ dividing $B$ but not $C$ and $p_2$ dividing $C$ but not $B$. Both $p_1$ and $p_2$ are allowed to divide $A$.
		In this case we apply \cref{lemma_quadraticreciprocity} with the following sets. If exactly one of $p_1,p_2$ divides $A$, let $F_2=\{p_1,p_2\}$, otherwise let $F_2=\{p_0,p_1,p_2\}$. In either case, let $F_1$ be the set of all primes that divide $ABC$ other than those in $F_2$, together with $-1$.
	\end{proof}			
	\cref{P:partialconverse1} follows by 
	putting together the   lemmas above (with $A=ac$, $B=bc$ and $C=(a+b)c$).
	
	\subsection{Other necessary conditions} Finally, we prove Propositions~\ref{P:partialconverse2} and \ref{P:partialconverse3}.
	\begin{proof}[Proof of \cref{P:partialconverse2}]
		The assumptions imply that $c$ is even. 
		Since the equation $ax^2+by^2=cz^2$ is partition regular with respect to $x,y$ if and only if $ax^2+by^2=c(2z)^2$ is (as the color of $z$ is irrelevant), we may assume that $c\equiv2\bmod4$.
		
		Let $\chi:\N\to\{-1,1\}$ be the completely multiplicative function defined by $\chi(2)=-1$ and $\chi(p)=1$ for any other prime $p$.
		Suppose, for the sake of a contradiction, that $x,y,z$ satisfy \eqref{E:abc} and $\chi(x)=\chi(y)$.
		Since $a,c$ are even and $b$ is odd, it follows that $y$ is even.
		If $z$ is even, then $x$ must be even as well; in this case $(x/2,y/2,z/2)$ is another solution with $\chi(x)=\chi(y)$, so we may reduce to the case when $z$ is odd.
		Similarly, we can assume that $x$ is odd.
		Together with $\chi(x)=\chi(y)$ this implies that $y$ is a multiple of $4$.
		
		Analysing the equation \eqref{E:abc} modulo $16$ yields $ax^2\equiv cz^2$.
		Since odd squares are congruent to $1\bmod8$, after dividing by $2$ we have $(a/2)\equiv (c/2)\bmod8$.
		This implies that the odd perfect square $S:=(a/2)b(a+b)(c/2)$ satisfies $1\equiv S\equiv ab+b^2\equiv ab+1\bmod8$, and hence that $ab$ is a multiple of $8$.
		This contradicts our assumptions on $a,b$ and we conclude that in this case \eqref{E:abc} is not partition regular with respect to $x,y$.
	\end{proof}
	
	\begin{proof}[Proof of \cref{P:partialconverse3}]
		
		Let $d\in\N$ be such that $2^d||a+b$. 
		From the assumptions it follows that $d$ is odd and that $s:=ab(a+b)c/2^{d+1}$ is an odd perfect square.
		Let $\ell=d+3$ and consider the coloring $\chi:\N\to\{1,3,\dots,2^\ell-1\}$ described recursively by $\chi(n)=n\bmod 2^\ell$ if $n$ is odd, and $\chi(n)=\chi(n/2)$ if $n$ is even.
		
		Suppose $x,y,z\in\Z$ satisfy \eqref{E:abc} and $\chi(x)=\chi(y)$. 
		If both $x,y$ are even, then $ax^2+by^2$ is a multiple of $4$, forcing $z$ to be even. 
		In that case, dividing all of $x,y,z$ by $2$ we still obtain a solution with $\chi(x)=\chi(y)$, so we may assume without loss of generality that $x,y$ are not both even.
		Analysing mod $2$ shows that both $x$ and $y$ must be odd, and hence $x\equiv y\bmod 2^\ell$.
		
		It follows that $ax^2+by^2\equiv(a+b)x^2\bmod2^\ell$ and hence $2^d$ is the largest power of $2$ that divides $ax^2+by^2$.
		Using \eqref{E:abc} and the fact that $2||c$ it follows that $2^{d-1}||z^2$, and hence that $z':=z^2/2^{d-1}$ is an odd perfect square.
		Reducing the equation modulo $2^\ell$ we have
		$$2^d\frac{a+b}{2^d}x^2\equiv 2^d\frac c2\cdot z'\bmod 2^\ell\qquad\iff\qquad \frac{a+b}{2^d}x^2\equiv\frac c2\cdot z'\bmod 8.$$
	Since  $s=ab(a+b)c/2^{d+1}$, $x^2$,  $z'$  are odd perfect squares, and hence equal to $1\bmod8$, multiplying both sides above by the odd number $abc/2$ we get that $1\equiv abc^2/4\bmod8$. Since $c^2/4$ is another odd square, this yields the desired conclusion $ab\equiv1\bmod8$.
	\end{proof}
	
	

\end{document}